\documentclass[preprint,3p,times,12pt]{elsarticle}

\usepackage{amssymb}
\usepackage{lipsum}
\usepackage{amsmath}
\usepackage{cool}
\usepackage{mathtools}
\usepackage{dirtytalk}
\usepackage{cuted}
\usepackage{relsize}
\usepackage[font=small]{caption}
\usepackage{subcaption}
\usepackage{ragged2e}
\usepackage{ifthen}
\usepackage[normalem]{ulem}
\usepackage{xcolor}
\usepackage{url}
\usepackage[inkscapearea=page]{svg}
\usepackage{tabularx}
\usepackage{hyperref}
\usepackage{mathrsfs}
\usepackage{enumitem}
\usepackage{multirow}
\usepackage{stmaryrd}
\usepackage{algorithm, algorithmic}
\usepackage{diagbox}

\newcommand{\e}{\mathrm{e}}

\newcommand{\upi}{\mathrm{i}}

\newcommand{\w}{\mathrm{w}}

\newcommand{\mean}[1]{\langle#1\rangle}
\newcommand{\jump}[1]{\llbracket#1\rrbracket}
\newcommand{\partialn}{\partial_{\boldsymbol{n}}}
\newcommand{\Drig}{\boldsymbol{\rm D}}
\newcommand{\X}{\mathcal{X}}
\newcommand{\Ctr}{C_{\text{tr}}}
\newcommand{\xvec}{\mathbf{x}}
\newcommand{\yvec}{\mathbf{y}}
\newcommand{\vertiii}[1]{{\left\vert\kern-0.25ex\left\vert\kern-0.25ex\left\vert #1 \right\vert\kern-0.25ex\right\vert\kern-0.25ex\right\vert}}

\usepackage{stfloats}

\usepackage{amsthm}
\newtheorem{theorem}{Theorem}
\newtheorem{assumption}{Assumption}

\newtheorem{proposition}[theorem]{Proposition}%

\newtheorem{remark}{Remark}

\newtheorem{lemma}{Lemma}

\usepackage{tikz}
\usetikzlibrary{arrows.meta}

\usepackage{lineno}

\journal{}

\newcommand{\bnabla}{\boldsymbol{\nabla}}
\newcommand{\bn}{\boldsymbol{n}}

\begin{document}

\begin{frontmatter}

\title{A discontinuous finite element method for the hydroelastic analysis of submerged structures}

\author[qut]{Zachary J. Wegert\corref{cor}}
\ead{wegert@qut.edu.au}
\author[tudelft]{Oriol Colom\'es}
\author[UNC]{Michael H. Meylan}
\author[UNC]{Ngamta Thamwattana}
\author[qut]{Vivien J. Challis}

\cortext[cor]{Corresponding author}

\affiliation[qut]{organization={School of Mathematical Sciences, Queensland University of Technology},
            city={Brisbane},
            state={QLD},
            postcode={4000}, 
            country={Australia}}

\affiliation[tudelft]{organization={Faculty of Civil Engineering and Geosciences, Delft University of Technology},
            city={Delft},
            postcode={2628 CN}, 
            country={The Netherlands}}

\affiliation[UNC]{organization={School of Computer and Information Sciences, University of Newcastle},
            city={Callaghan}, 
            state={NSW},
            postcode={2308},
            country={Australia}}

\begin{abstract}
This paper develops a finite element method for wave-structure interaction problems arising in the hydroelastic modelling of submerged elastic plates. The approach is formulated for both two-dimensional settings and three-dimensional channels. For the fluid part of the problem, we use a symmetric discontinuous Galerkin scheme that captures the discontinuity in the potential across the plate. To impose the Sommerfeld radiation conditions, we take inspiration from the Dirichlet-to-Neumann map approach and use the analytic solution in the semi-infinite domains to construct appropriate non-local boundary conditions. For the plate, we use a continuous/discontinuous Galerkin method to resolve the 4$^{\rm th}$ order operator without using continuously differentiable finite elements. We prove well-posedness of the method, showing that the resulting sesquilinear form is bounded and satisfies the G{\aa}rding inequality, leading to a unique and stable solution. We implement our formulation as an open-source tool using the Gridap package in Julia. We validate our method against a hypersingular boundary integral method, finding excellent agreement. The flexibility of the methodology is demonstrated through simulations of a two-dimensional slanted plate and three-dimensional plate geometries including centred rectangular plates, offset rectangular plates, and annulus-shaped plates. The finite element method presented here can be readily applied to model wave energy conversion of piezoelectric bimorphs. Furthermore, the method could easily be extended to consider wave scattering from fixed or rigid structures or from variable bottom topography, and is not limited to hydroelastic applications.
\end{abstract}


\end{frontmatter}


\section{Introduction}\label{sec: Introduction}\noindent

The problem of linear water-wave scattering has numerous applications. Standard methods, as implemented in commercial software such as WAMIT \cite{LeeNewman2006}, are based primarily on Green's functions and boundary integral methods. It has long been recognised that the finite element method (FEM) can also be applied to such scattering problems, although this approach has been much less extensively developed. The FEM has several distinct advantages, including access to highly developed computational techniques and considerable flexibility in accommodating different physical state equations and geometries. These features make the FEM particularly well-suited to problems involving thin elastic plates, especially those arising in modelling proposed devices for ocean wave energy harvesting.

Ocean waves provide an abundant source of renewable energy that can be harnessed using devices known as wave energy converters (WECs). The development of WECs faces several challenges, including the need for devices that are durable enough to withstand harsh marine environments while maintaining high power take-off efficiency over long periods of operation. Consequently, considerable research has focused on the design of WECs, including their structural configurations, material selection, and operating principles. We refer the reader to \cite{ZHOU2026} for a comprehensive review of WEC technologies over the past two decades, including emerging challenges and future directions.  
In this paper, we develop an FEM approach that can be used to model submerged plate WECs, in which pressure differences and wave-induced flow around the submerged plate generate oscillatory motion or hydraulic effects for energy conversion. As mentioned in a  state-of-the-art review of submerged WECs \cite{ROACH2025}, a submerged plate is preferred  because it experiences lower wave loads, smoother water motion, and improved durability compared to a surface plate. This improves energy conversion efficiency while reducing structural stress and maintenance.  

    
     
Submerged plate WECs have been widely studied experimentally  \cite{Polly_etal2025,ORER2007} and mathematically \cite{HAYATDAVOODI2015,V20221382,Wegert_Wilks_Thamwattana_Challis_Koley_Meylan_2026}.   
In terms of  numerical modelling, most research on submerged plate WECs  employs the finite volume method (FVM) (e.g., ANSYS Fluent) \cite{jmse13081564,Thum_etal2024,Seibt_etal2023,Seibt_etal2014}, boundary element method (BEM) \cite{Panduranga_etal2023,CHEN2017,JIN2026}, or smoothed particle hydrodynamics (SPH) \cite{He_etal2018,HE2019,SPH_ref}.  In contrast, relatively few studies have applied the FEM directly to submerged plate WECs. Most of these studies are based on commercial software packages (e.g., COMSOL) or coupled CFD–FEM approaches \cite{Yang_etal2020,HUANG2022}. Therefore, this paper develops a dedicated FEM implementation relevant to submerged plate WECs. The development and implementation of our own FEM framework provides greater flexibility to incorporate different structural configurations, material properties, and numerical techniques, and will facilitate future improvement of WEC designs. 


Based on the FEM studies presented in \cite{OriolHydroelasticity,Colomes_preprint,Agarwal_etal2024}, this paper  develops a finite element formulation to solve the hydroelastic problem of a submerged plate described in \cite{Wegert_Wilks_Thamwattana_Challis_Koley_Meylan_2026}, neglecting piezoelectric effects.  
In   \cite{OriolHydroelasticity}, Colom\'{e}s et al. introduced  a monolithic finite element formulation that couples linear potential flow with structural models into a single system of equations, providing a stable and efficient approach for analysing wave–structure interactions in very large floating structures. Building on  \cite{OriolHydroelasticity}, Pradnyanata and Colom\'{e}s \cite{Colomes_preprint} extended  the formulation to include moored and hinged floating structures using Mindlin–Timoshenko beam theory, allowing accurate modelling of shear deformation, rotary inertia, mooring systems, and power take-off mechanisms for wave energy applications within a unified FEM framework. In addition, Agarwal et al.  \cite{Agarwal_etal2024} expanded  the method to investigate the dynamic behaviour of viscoelastic floating membranes, demonstrating the ability of monolithic FEM to capture the coupled effects of fluid loading, structural deformation, and material viscoelasticity. The ability of this approach to handle non-homogeneous structural properties was also demonstrated in \cite{el2026adjoint}. We note that these studies apply the monolithic finite element approach to hydroelastic problems involving floating structures, while this paper particularly formulates the FEM framework for submerged plates. The proposed formulation is directly relevant to modelling submerged plate piezoelectric WEC. Furthermore, the framework can be extended and adapted in several ways, including to the case of non-uniform ocean depth, non-homogeneous structures, and non-linear bending. 

  

    In the following section, mathematical models for the hydroelasticity problem of the submerged plate in a two-dimensional fluid and the three-dimensional channel are described that form the basis for the finite element formulations presented in Section~\ref{sec: FE formulation}. The numerical analysis for our finite element scheme is provided in Section~\ref{sec: analysis}. In Section~\ref{sec: results}, implementation and validation of the formulation shown in Section~\ref{sec: model} are given for submerged horizontal plate, slanted plates and three-dimensional channel. Finally, conclusions are given in Section~\ref{sec: Conclusions}.




\section{Problem setup}\label{sec: model}\noindent
This section describes the model, assumptions, and imposition of radiation conditions in finite domains.

\subsection{Hydroelastic model}\noindent
In the following, we consider the development of a finite element method to solve the hydroelastic model described by \citet{Wegert_Wilks_Thamwattana_Challis_Koley_Meylan_2026}, neglecting the piezoelectric effect. This describes the two-dimensional wave-structure interaction between the ocean and large submerged structures. We also consider the extension to the case of a three-dimensional channel. We note that the formulation readily extends to the hydroelectromechanical modelling of piezoelectric WECs discussed in \citep{Wegert_Wilks_Thamwattana_Challis_Koley_Meylan_2026} by including complex-valued rigidity coefficients, although we consider this to be future work.

\subsubsection{Two-dimensional model}\noindent
In two dimensions, the hydroelastic model approximates the bending of a one-dimensional plate $\Gamma$ in a two-dimensional fluid $\Omega$ that is inviscid, incompressible, and irrotational. We assume that the motion is time harmonic with an angular frequency $\omega$, thus posing  the problem in the frequency domain and reducing computational complexity. Figure~\ref{fig:schematic} shows a diagram of the geometric setup.
\begin{figure}[h]
\begin{center}
\begin{tikzpicture}[scale=2.3]

\def\waveAmp{0.15}
\def\waveEnd{0.75}
\def\freq{480}

\draw[dashed] (0,1) -- (0,3);
\draw[dashed] (6,1) -- (6,3);
\draw[dashed] (1.5,1) -- (1.5,3);
\draw[dashed] (4.5,1) -- (4.5,3);
\draw[dashed] (0,2) -- (6,2);
\draw[line width=0.5mm] (0,1) -- (6,1) node[above,xshift = -165] {$\partialn\phi = 0$ on $\Gamma_{\mathrm{b}}$};
\draw[line width=1mm, black] (1.5,2) -- (4.5,2) node[below,xshift = -85] {$D \partial_x^4 \w - I_b \omega^2 \: \w = \upi\omega\rho_w\llbracket\phi\rrbracket $ on $\Gamma$};
\node[text width=3cm] at (6.8,3.02) {$z = 0$};
\node[text width=3cm] at (6.8,2.02) {$z = -h$};
\node[text width=3cm] at (6.8,1.02) {$z = -H$};

\draw[line width=0.5mm, solid, blue] 
    plot[domain=0:\waveEnd, samples=50] (\x, {3 + \waveAmp*sin(\x*\freq)}) 
    -- (6,3);

\draw[dashed, blue!50, thin] (0,3) -- (6,3);

\node[text width=1cm] at (0.3,3.3) {$\phi^{\rm in}$};

\node[text width=3cm] at (1.9,3.2) {$\partialn\phi=\frac{\omega^2}{g} \phi$ on $\Gamma_{\mathrm{fs}}$};

\draw[line width=0.5mm, black,->] (3,3) -- (3,3.5) node[right] {$z$};
\draw[line width=0.5mm, black,->] (5,3) -- (5,3.25) node[right] {$\boldsymbol{n}$};
\draw[line width=0.5mm, black,->] (3,3) -- (3.5,3) node[above] {$x$};
\draw[line width=0.5mm, black,->] (6,2.5) -- (6.25,2.5) node[above] {$\boldsymbol{n}$};
\draw[line width=0.5mm, black,->] (0,2.5) -- (-0.25,2.5) node[above] {$\boldsymbol{n}$};
\draw[line width=0.5mm, black,->] (4.5,2) -- (4.75,2) node[below] {$\boldsymbol{n}_\Gamma$};
\draw[line width=0.5mm, black,->] (4.3,2) -- (4.3,2.25) node[right] {$\boldsymbol{n}$};
\draw[line width=0.5mm, black,->] (1.5,2) -- (1.25,2) node[below] {$\boldsymbol{n}_\Gamma$};
\node[text width=3cm] at (0.9,1.5){$\displaystyle{-\Delta\phi = 0}$ in $\Omega$};
\node[text width=6cm] at (3.75,2.2){${\partialn\phi }+\upi\omega \w=0$ on $\Gamma$};
\node[text width=2cm] at (1.65,0.8){$x=-L$};
\node[text width=2cm] at (4.75,0.8){$x=L$};
\node[text width=2cm] at (0.25,0.8){$x=-L_f$};
\node[text width=2cm] at (6.25,0.8){$x=L_f$};
\node[text width=1cm] at (0.3,2.3){$\Gamma_{\mathrm{in}}$};
\node[text width=1cm] at (5.9,2.3){$\Gamma_{\mathrm{out}}$};
\end{tikzpicture}
\end{center}
\caption{Visualisation of the submerged plate hydroelasticity problem based on \citet{Wegert_Wilks_Thamwattana_Challis_Koley_Meylan_2026}. The finite-domain case is visualised. The free surface is denoted by $\Gamma_\mathrm{fs}$ and the ocean floor by $\Gamma_\mathrm{b}$. Note that $\boldsymbol{n}_\Gamma$ denotes the normal along the plane $\Gamma$.}
\label{fig:schematic}
\end{figure}
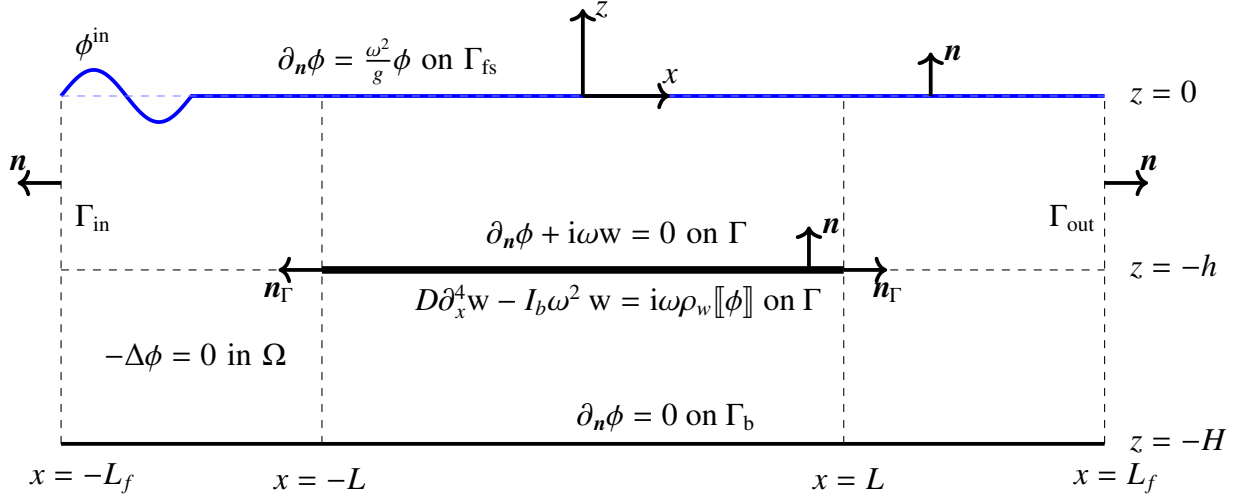
The model is given by the following equations \citep{Wegert_Wilks_Thamwattana_Challis_Koley_Meylan_2026}:

\begin{align}
-\Delta\phi=&\,0\quad \text{ in }\Omega,\label{eqn: CS Omega}\\
\partialn\phi=&\,0\quad \text{ on }\Gamma_{\mathrm{b}},\label{eqn: CS Gamma_b}\\
\partialn\phi-\alpha\phi =&\, 0\quad \text{ on }\Gamma_{\mathrm{fs}},\label{eqn: CS Gamma_fs}\\
\partialn\phi^{\rm sc}-\upi k_0\phi^{\rm sc}\rightarrow &\, 0\quad \text{ as }\lvert x\rvert\rightarrow\infty\label{eqn: CS Sommerfeld},
\\
{\partialn \phi}+\upi \omega \w=&\,0\quad \text{ on }\Gamma,\label{eqn: CS Gamma}\\
D\Delta^2\w - I_0 \omega^2\w - \upi\omega\rho_w\jump{\phi} =&\, 0\quad \text{ on }\Gamma.\label{eqn: CS Gamma plate}
\end{align}
In the above, \eqref{eqn: CS Omega}--\eqref{eqn: CS Sommerfeld} are the equations for the fluid and \eqref{eqn: CS Gamma} and \eqref{eqn: CS Gamma plate} are the kinematic and dynamic conditions on the plate. In addition to these,  $\w$ is  subject to either free boundary conditions:
\begin{equation}\label{eqn:free}
D\Delta^2\w = \bnabla(D\Delta^2\w)\cdot\boldsymbol{n}_\Gamma = 0\quad\text{on }\partial\Gamma,
\end{equation}
clamped boundary conditions:
\begin{equation}\label{eqn:clamped}
 \w = \partial_{\bn_{\Gamma}}\w = 0\quad\text{on }\partial\Gamma,
\end{equation}
or simply-supported boundary conditions:
\begin{equation}\label{eqn:simply-supported}
 \w = D\Delta^2\w = 0\quad\text{on }\partial\Gamma.
\end{equation}
In the above, $\Delta$ is the Laplacian operator, $\phi$ is the complex-valued fluid potential, $\boldsymbol{n}$ is the outward-facing unit normal to $\partial\Omega$, $\alpha=\frac{\omega^2}{g}$ is the dispersion coefficient for gravity $g$ and the given frequency $\omega$, $\w$ is the complex-valued out-of-plane plate displacement, $D$ is the real- or complex-valued bending rigidity, $I_0$ is the mass per unit length, $\rho_w$ is the density of the fluid, $\Delta^2$ is the bilaplacian operator, $\boldsymbol{n}_\Gamma$ denotes the normal along the plane $\Gamma$, $\jump{v}$ is the jump operator defined as $\jump{v}=v_1-v_2$, and $\mean{v}$ is the mean operator defined as $\mean{v}=\frac{1}{2}(v_1+v_2)$. To ensure that the sign of $\jump{\phi}$ is correct in \eqref{eqn: CS Gamma plate}, on $\Gamma$ we take $v_1$ and $v_2$ in $\jump{v}$ to be the values of $v$ below and above $\Gamma$, respectively. Finally, $\phi^{\rm sc}=\phi - \phi^{\rm in}$ is the scattered wave potential where 
\begin{equation}\label{eqn:plane_inc}
\phi^{\rm in}(x,z)=\beta\mathrm{e}^{\upi k_0 x}\frac{\cosh k_0(z+H)}{\cosh kH}
\end{equation}
denotes the velocity potential of a plane incident wave with scaling $\beta=\frac{-igA}{\omega}$ where $A$ is the amplitude; and $k_0$ is the real and positive solution of the dispersion equation
\begin{equation}
    k_0\tanh(k_0 H)=\frac{\omega^2}{g}.
\end{equation}
We refer the reader to \citet{Wegert_Wilks_Thamwattana_Challis_Koley_Meylan_2026} for a detailed discussion of the above model and its assumptions, and \citet{Linton2001} for a discussion of linear water wave theory.

\subsubsection{Three-dimensional channel model}\label{sec: three dim channel}\noindent
In Figure~\ref{fig:schematic_3d}, we show an example of a three-dimensional channel with a rectangular plate centred in the fluid with a plane incident wave.
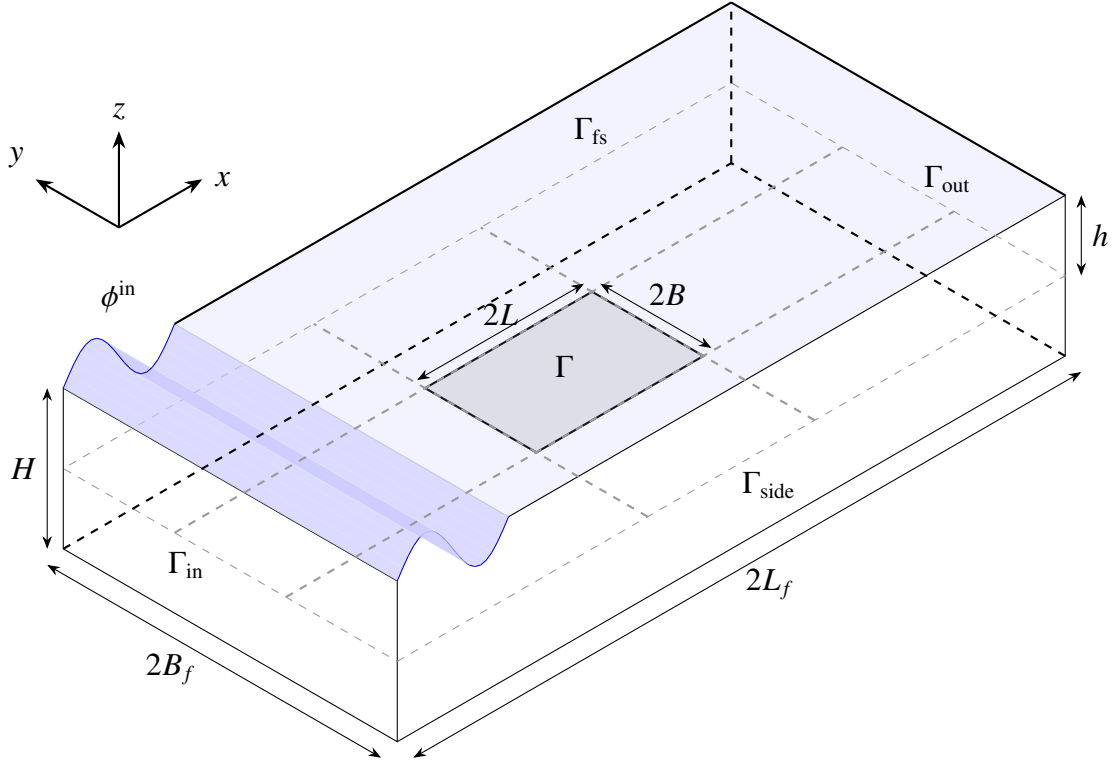
\begin{figure}[t]
    \begin{center}
    \begin{tikzpicture}[
    x={(0.866cm, 0.5cm)}, 
    y={(-0.866cm, 0.5cm)}, 
    z={(0cm, 1cm)},
    scale=0.85
]

\def\L{12} 
\def\W{6} 
\def\H{2.5} 
\def\midZ{1.25} 

\def\Sx{3} 
\def\Sy{2} 
\def\offsetX{\L/2 - \Sx/2}
\def\offsetY{\W/2 - \Sy/2}
\def\waveEnd{\L/6}

\draw (0,0,0) -- (\L,0,0) -- (\L,0,\H) -- (\waveEnd,0,\H); 
\draw (0,0,0) -- (0,\W,0) -- (0,\W,\H) -- (0,0,\H) -- cycle; 

\filldraw[fill=gray!20, draw=black, thick] 
    (\offsetX, \offsetY, \midZ) -- 
    (\offsetX + \Sx, \offsetY, \midZ) -- 
    (\offsetX + \Sx, \offsetY + \Sy, \midZ) -- 
    (\offsetX, \offsetY + \Sy, \midZ) -- cycle;

\draw[dashed, draw=black!40] (0,0,\midZ) -- (\L,0,\midZ) -- (\L,\W,\midZ) -- (0,\W,\midZ) -- cycle;

\def\waveAmp{0.5}
\def\freq{180} 

\foreach \x in {0, 0.1, ..., 2} {
    \pgfmathsetmacro{\xnext}{\x + 0.1}
    \ifdim \xnext pt > 2pt \def\xnext{2} \fi 
    
    \pgfmathsetmacro{\zcurr}{\H + \waveAmp*sin(\x*\freq)}
    \pgfmathsetmacro{\znext}{\H + \waveAmp*sin(\xnext*\freq)}
    
    \fill[blue!20, opacity=0.7] 
        (\x, 0, \zcurr) -- (\xnext, 0, \znext) -- 
        (\xnext, \W, \znext) -- (\x, \W, \zcurr) -- cycle;
        
    \draw[blue!80!black, thin] (\x, 0, \zcurr) -- (\xnext, 0, \znext);
    \draw[blue!80!black, thin] (\x, \W, \zcurr) -- (\xnext, \W, \znext);
}

\fill[blue!30, opacity=0.15, draw=black] 
    (\waveEnd, 0, \H) -- (\L, 0, \H) -- (\L, \W, \H) -- (\waveEnd, \W, \H) -- cycle;

\draw[thick] (\waveEnd, \W, \H) -- (\L, \W, \H);
\draw[thick] (\L, 0, \H) -- (\L, \W, \H);

\draw[dashed, thick] (\L,0,0) -- (\L,\W,0);
\draw[dashed, thick] (0,\W,0) -- (\L,\W,0);
\draw[dashed, thick] (\L,\W,0) -- (\L,\W,\H);

\draw[dashed, draw=black!40, thick] (0,\W/2+\Sy/2,\midZ) -- (\L,\W/2+\Sy/2,\midZ);
\draw[dashed, draw=black!40, thick] (0,\W/2-\Sy/2,\midZ) -- (\L,\W/2-\Sy/2,\midZ);
\draw[dashed, draw=black!40, thick] (\L/2-\Sx/2,0,\midZ) -- (\L/2-\Sx/2,\W,\midZ);
\draw[dashed, draw=black!40, thick] (\L/2+\Sx/2,0,\midZ) -- (\L/2+\Sx/2,\W,\midZ);

\node at (0, \W/2, \H/2) [left, xshift=-0.2cm] {$\Gamma_{\rm in}$};
\node at (\L - 2, \W, \H-1) [left] {$\Gamma_{\rm fs}$};
\node at (\L, \W/2, \H/2) [right, xshift=0.2cm] {$\Gamma_{\rm out}$};
\node at (\L/2, 0, \H/2) [below right, yshift=0.1cm] {$\Gamma_{\rm side}$};
\node at (1, \W, \H + \waveAmp) [above] {$\phi^{\rm in}$};
\node at (\offsetX + \Sx/2 - 0.25, \offsetY + \Sy/2-0.25, \midZ) [above] {$\Gamma$};

\draw[<->, >=Stealth, yshift=-0.25cm, xshift = 0.25cm] (0,0,0) -- (\L,0,0) node[midway, below,yshift=0.1cm, xshift = 0.3cm] { $2L_f$};
\draw[<->, >=Stealth, xshift=-0.25cm, yshift=-0.25cm] (0,0,0) -- (0,\W,0) node[midway, left, xshift=-0.1cm, yshift=-0.1cm] { $2B_f$};
\draw[<->, >=Stealth, xshift=0.25cm] (\L,0,\H) -- (\L,0,\H/2) node[midway, right] { $h$};
\draw[<->, >=Stealth, xshift=-0.25cm] (0,\W,\H) -- (0,\W,0) node[midway, left] { $H$};
\draw[<->, >=Stealth, yshift=0.125cm, xshift=-0.125cm] (\offsetX, \offsetY + \Sy, \midZ) -- (\offsetX + \Sx, \offsetY + \Sy, \midZ) node[midway, above,yshift=-0.05cm] { $2L$};
\draw[<->, >=Stealth, xshift=0.125cm, yshift=0.125cm] (\offsetX + \Sx, \offsetY, \midZ) -- (\offsetX + \Sx, \offsetY + \Sy, \midZ) node[midway, above, xshift=0.15cm] { $2B$};

\begin{scope}[shift={(3, \W+2, \H)}]
    \draw[->, >=Stealth, thick] (0,0,0) -- (1.5,0,0) node[right] {$x$};
    \draw[->, >=Stealth, thick] (0,0,0) -- (0,1.5,0) node[above left] {$y$};
    \draw[->, >=Stealth, thick] (0,0,0) -- (0,0,1.5) node[above] {$z$};
\end{scope}

\end{tikzpicture}
    \end{center}
    \caption{Visualisation of an extension of the submerged plate hydroelasticity problem to a three-dimensional channel with a plane incident wave. Note that our formulation readily allows for a non-plane incident wave.}
    \label{fig:schematic_3d}
\end{figure}
For these cases, the submerged plate can be modelled as a thin two-dimensional Poisson–Kirchhoff plate \citep{OriolHydroelasticity}. In this case, \eqref{eqn: CS Gamma plate} becomes
    \begin{equation}
        \bnabla^2:(\Drig:\bnabla^2\w) - I_0 \omega^2\w - \upi\omega\rho_w\jump{\phi} =\, 0\quad \text{ on }\Gamma,\label{eqn: CS Gamma plate 2d}
    \end{equation}
    where $\bnabla^2$ is the Hessian operator and $\Drig$ is a symmetric rigidity tensor of order 4. Note that when the normal of the plate is aligned with the $z$-coordinate direction, all $z$ components of $\Drig$ vanish. In this case, typical boundary conditions on $\w$ take the following form:
\begin{equation}\label{eqn:free 3d}
(\Drig:\bnabla^2\w)\cdot\boldsymbol{n}_{\Gamma} = (\bnabla\cdot(\Drig:\bnabla^2\w))\cdot\boldsymbol{n}_\Gamma = 0\quad\text{on }\partial\Gamma
\end{equation}
for free boundary conditions;
\begin{equation}\label{eqn:clamped 3d}
 \w = \partial_{\boldsymbol{n}_{\Gamma}}\w = 0\quad\text{on }\partial\Gamma
\end{equation}
for clamped boundary conditions; and
\begin{equation}\label{eqn:simply-supported 3d}
 \w = (\Drig:\bnabla^2\w)\cdot\boldsymbol{n}_{\Gamma}= 0\quad\text{on }\partial\Gamma
\end{equation}
for simply-supported boundary conditions. 
For the sides of the channel, $\Gamma_{\rm side}$, we also impose slip boundary conditions on the fluid, i.e.,
\begin{equation}
    \partialn\phi=0\quad\text{on }\Gamma_{\rm side}.
\end{equation}
Finally, the incident wave becomes
\begin{equation}
    \phi^{\rm in}(x,y,z)=\beta\e^{ik_{0p}x}\frac{\cosh k_0(z+H)}{\cosh kH}\cos\left(\frac{p\pi (y-B_f)}{2B_f}\right),
\end{equation}
where $p$ is the frequency in the $y$-direction and $k_{0p}$ is given by
\begin{equation}
    k_{0p} = \sqrt{k_0^2 - \left(\frac{p\pi}{2B_f}\right)^2}.
\end{equation}
Note that in the case of a plane wave ($p=0$), $\phi^{\rm in}$ reduces to \eqref{eqn:plane_inc}.

\section{Finite element formulation}\label{sec: FE formulation}\noindent
In the following, we derive the finite element formulation for the submerged plate hydroelastic model.

\subsection{Notation}\noindent
Suppose that we partition $\Omega\subset\mathbb{R}^d$ into patches $\Omega_e$ and denote the dimension $d-1$ boundaries of $\Omega_e$ by $\mathcal{E}$. We define $\Gamma=\cup_e\Gamma_e$ where $\Gamma_e\in\mathcal{E}$. Next, we define the skeleton of $\Omega$ and $\Gamma$ as $\Lambda = (\bigcup_e \partial\Omega_e) \setminus \partial\Omega$ and $\Sigma = (\bigcup_e \partial\Gamma_e) \setminus \partial\Gamma$, respectively. On $\Lambda$, we define the operators $\jump{v}=v_1-v_2$ and $\mean{v}=\frac{1}{2}(v_1+v_2)$ where $v_1$ and $v_2$ are the restriction of $v\in V$ to cells $\Omega_{e_1}$ and $\Omega_{e_2}$ that share an interior facet $F\in\Lambda$, similarly on $\Sigma$. Note that in $\Gamma$ we take $v_1$ and $v_2$ in $\jump{v}$ to be the values of $v$ below and above $\Gamma$, respectively.

We also define the notation $(u,v)_\Omega=\int_{\Omega}u\overline{v}~\mathrm{d}\Omega$ and $(u,v)_\Gamma=\int_{\Gamma}u\overline{v}~\mathrm{d}\Gamma$ with associated norms $\lVert u\rVert_{L^2(\Omega;\mathbb{C})}\equiv\lVert u\rVert_{\Omega}=(u,u)_{\Omega}^{1/2}$ and $\lVert u\rVert_{L^2(\Gamma;\mathbb{C})}\equiv\lVert u\rVert_{\Gamma}=(u,u)_{\Gamma}^{1/2}$, respectively. Finally, we let ${\mathcal{V}}=\{v\in L^2(\Omega;\mathbb{C}):v\vert_{\Omega_e}\in H^1(\Omega_e;\mathbb{C})~\forall \Omega_e\in\Omega\}$ and ${\mathcal{V}}_\Gamma=H^2(\Gamma;\mathbb{C})$ be function spaces. Note that $V$ is a so-called \textit{broken space} \citep{oden_discontinuoushpfinite_1998}.

\subsection{Fluid}\noindent
In the following, we derive the weak formulation for the hydroelastic problem given by \eqref{eqn: CS Omega}-\eqref{eqn: CS Sommerfeld}. We proceed by multiplying \eqref{eqn: CS Omega} by a test function $v$ and integrating by parts over each partition of $\Omega$. This gives
\begin{align}
    0 &= (\bnabla\phi,\bnabla v)_{\Omega} - \sum_{\Omega_e\in\Omega}(\partialn\phi,v)_{\partial\Omega_e} \nonumber\\  
    &= (\bnabla\phi,\bnabla v)_{\Omega} - (v_1,\partial_{\boldsymbol{n}_1}\phi_1)_{\Lambda} - (v_2,\partial_{\boldsymbol{n}_2}\phi_2)_{\Lambda} - (\partialn\phi,v)_{\partial\Omega}.
\end{align}
It can be shown by definition of the jump and mean operators that $\jump{ab}=\jump{a}\mean{b} + \mean{a}\jump{b} $. Using this identity and separating $\Lambda$ into $\Gamma$ and $\Lambda\setminus\Gamma$ gives
\begin{equation}
    0 = (\bnabla\phi,\bnabla v)_{\Omega} - (\partialn\phi,v)_{\partial\Omega} - (\mean{\partialn\phi},\jump{v})_{\Lambda\setminus\Gamma} - (\jump{\partialn\phi},\mean{v})_{\Lambda\setminus\Gamma} - (v_1,\partial_{\boldsymbol{n}}\phi_1)_{\Gamma} + (v_2,\partial_{\boldsymbol{n}}\phi_2)_{\Gamma},
\end{equation}
where we have used the normal $\boldsymbol{n}$ on $\Gamma$ as shown in Figure~\ref{fig:schematic}. Next, applying \eqref{eqn: CS Gamma} and using the definition of the jump on $\Gamma$, we can write
\begin{equation}
    0 = (\bnabla\phi,\bnabla v)_{\Omega} - (\partialn\phi,v)_{\partial\Omega} - (\mean{\partialn\phi},\jump{v})_{\Lambda\setminus\Gamma} - (\jump{\partialn\phi},\mean{v})_{\Lambda\setminus\Gamma} + (i\omega\w,\jump{v})_{\Gamma},\label{eqn: wf intermediate 1}
\end{equation}
Requiring that $\partialn\phi$ be continuous, we have that
\begin{align}
    (\jump{\partialn\phi},\mean{v})_{\Lambda\setminus\Gamma}&=0.\label{eqn: DG term vanish}
\end{align}
As a result of \eqref{eqn: DG term vanish}, the second-to-last term in \eqref{eqn: wf intermediate 1} vanishes. Finally, we enforce the continuity of $\phi$ using a symmetric interior penalty method \citep{doi:10.1137/0719052}. This method relies on adding a symmetry term $(\mean{\partialn v},\jump{\phi})_{\Lambda\setminus\Gamma}$ and a penalty term $(\mu_1\jump{\phi},\jump{v})_{\Lambda\setminus\Gamma}$ to the variational form, where $\mu_1$ is a penalty parameter. These are consistent additions in the sense that $\jump{\phi}=0$ in the case of $\phi$ being a solution to \eqref{eqn: CS Omega}--\eqref{eqn: CS Sommerfeld}. This results in:
\begin{equation}
    0 = (\bnabla\phi,\bnabla v)_{\Omega} - (\partialn\phi,v)_{\partial\Omega} - (\mean{\partialn\phi},\jump{v})_{\Lambda\setminus\Gamma} - (\mean{\partialn v},\jump{\phi})_{\Lambda\setminus\Gamma} + (\mu_1\jump{\phi},\jump{v})_{\Lambda\setminus\Gamma} + (i\omega\w,\jump{v})_{\Gamma}.
\end{equation}
Using the boundary conditions given by \eqref{eqn: CS Gamma_b} and \eqref{eqn: CS Gamma_fs}, we can write the above as
\begin{equation}\label{eqn: intermediate pre SRC}
        0 = (\bnabla\phi,\bnabla v)_{\Omega} - (\alpha\phi,v)_{\Gamma_{\rm fs}} - (\partial_{\boldsymbol{n}}\phi,v)_{\Gamma_{\rm in}\cup\Gamma_{\rm out}} - (\mean{\partialn\phi},\jump{v})_{\Lambda\setminus\Gamma} - (\mean{\partialn v},\jump{\phi})_{\Lambda\setminus\Gamma} + (\mu_1\jump{\phi},\jump{v})_{\Lambda\setminus\Gamma} + (i\omega\w,\jump{v})_{\Gamma}.
\end{equation}
We now consider a method to impose the Sommerfeld radiation condition.

\subsubsection{Imposing Sommerfeld radiation conditions in two dimensions}\label{sec: sommerfeld 2d}\noindent
To apply the Sommerfeld radiation conditions in a finite domain, we take inspiration from the Dirichlet-to-Neumann map approach and use the solution in the semi-infinite domains $x\in(-\infty,-L_f)$ and $x\in(L_f,\infty)$ to construct the necessary boundary conditions at $x=\pm L_f$. This introduces a non-local integral operator into the formulation. In the following, we formulate the boundary conditions in two dimensions.

We start with the eigenfunction expansion solution for $|x|>L_f$. This and its derivative in $x$ are given by 
\begin{equation}\label{eqn: phi solution outside}
    \phi(x,z) =\begin{cases}
        \displaystyle\beta\psi_0(z)e^{\upi k_0x} +\sum_{m=0}^{\infty} R_m\psi_m(z)e^{-\upi k_mx},& x\leq -L_f,\\
        \displaystyle\sum_{m=0}^{\infty} T_m\psi_m(z)e^{\upi k_mx},& x \geq L_f,
    \end{cases} 
\end{equation}
\begin{equation}\label{eqn: D_xphi solution outside}
    \partial_x\phi(x,z) =\begin{cases}
        \displaystyle\upi k_0\beta\psi_0(z)e^{\upi k_0x} -\sum_{m=0}^{\infty} \upi k_mR_m\psi_m(z)e^{-\upi k_mx},& x\leq -L_f,\\
        \displaystyle\sum_{m=0}^{\infty} \upi k_mT_m\psi_m(z)e^{\upi k_mx},& x \geq L_f,
    \end{cases} 
\end{equation}
where $k_m$ are solutions of the dispersion relation $\alpha=k_m\tanh(k_mH)$ ordered with increasing magnitude and with $k_m$ being positive imaginary for $m>0$, $\psi_m(z)=\frac{\cosh{k_m(H+z)}}{\cosh{k_mH}}$ are the eigenfunctions, and $R_m$ and $T_m$ are the eigenfunction coefficients. Note that $\psi_m$ satisfies the orthogonality condition
\begin{equation}\label{eqn: orthog psi}
    \int_{-H}^0\psi_m(z)\psi_n(z)~\mathrm{d}z=\delta_{mn}C_{m}.
\end{equation}
At $x=-L_f$ and $x=L_f$, multiplying \eqref{eqn: phi solution outside} by $\psi_n$, integrating, applying orthogonality, and rearranging gives
\begin{equation}\label{eqn: reflection 2d}
    R_n=\begin{cases}
        \displaystyle-\beta e^{-2\upi k_0L_f} + \frac{e^{-\upi k_0L_f}}{C_0}\int_{-H}^0\phi(-L_f,z)\psi_0(z)~\mathrm{d}z,&n=0,\\
        \displaystyle\frac{e^{-\upi k_nL_f}}{C_n}\int_{-H}^0\phi(-L_f,z)\psi_n(z)~\mathrm{d}z,&n>0,
    \end{cases}
\end{equation}
and
\begin{equation}\label{eqn: transmission 2d}
    T_n=\frac{e^{-\upi k_nL}}{C_n}\int_{-H}^0\phi(L_f,z)\psi_n(z)~\mathrm{d}z.
\end{equation}
As a result, \eqref{eqn: D_xphi solution outside} in terms of $\Gamma_{\rm in}$ and $\Gamma_{\rm out}$ can be written as
\begin{align}
    \partial_{\boldsymbol{n}}\phi &= -2\upi k_0\beta\psi_0(z)e^{-\upi k_0L_f} + \sum_{m=0}^{\infty} \frac{\upi k_m\psi_m(z)}{C_m}\int_{\Gamma_{\rm in}}\phi(x,z)\psi_m(z)~\mathrm{d}s,&&\text{on }\Gamma_{\rm in},\\
    \partial_{\boldsymbol{n}}\phi &= \sum_{m=0}^{\infty} \frac{\upi k_m\psi_m(z)}{C_m}\int_{\Gamma_{\rm out}}\phi(x,z)\psi_m(z)~\mathrm{d}s,&&\text{on }\Gamma_{\rm out},
\end{align}
where $\boldsymbol{n}$ is the outward normal and $\Gamma_{\rm in}$ and $\Gamma_{\rm out}$ are as shown in Figure~\ref{fig:schematic}. Truncating the sums at $M$, we may use these in \eqref{eqn: intermediate pre SRC} to construct the operators
\begin{equation}\label{eqn: bilinear form fluid b}
    \begin{aligned}
        B_f([\phi,\w],[v,r]) =&~ (\bnabla\phi,\bnabla v)_{\Omega} - (\alpha\phi,v)_{\Gamma_{\rm fs}} - K(\phi,v) \\&- (\mean{\partialn\phi},\jump{v})_{\Lambda\setminus\Gamma} - (\mean{\partialn v},\jump{\phi})_{\Lambda\setminus\Gamma} + (\mu_1\jump{\phi},\jump{v})_{\Lambda\setminus\Gamma} + (i\omega\w,\jump{v})_{\Gamma},
    \end{aligned}
\end{equation}
and
\begin{equation}\label{eqn: bilinear form fluid L}
    L(v) = (-2\upi k_0\beta\psi_0(z)e^{-\upi k_0L},v)_{\Gamma_{\rm in}}=(-2\upi k_0\phi^{\rm in},v)_{\Gamma_{\rm in}},
\end{equation}
where $K(\phi,v)$ is given by
\begin{equation}\label{eqn: K op 2d}
    K(\phi,v) = \sum_{m=0}^{M}\frac{\upi k_m}{C_m}\left[(v,\psi_m)_{\Gamma_{\rm in}}(\phi,\psi_m)_{\Gamma_{\rm in}} + (v,\psi_m)_{\Gamma_{\rm out}}(\phi,\psi_m)_{\Gamma_{\rm out}}\right].
\end{equation}
It is important to note that $K$ is a non-local operator. Once the problem is posed in a finite function space, the term can be discretised by recognising that the resulting sparse matrix can be written as the outer product of two sparse vectors. Note that we use the term sparse loosely --- the vectors and matrix are dense but only for degrees of freedom on the boundaries $\Gamma_{\rm in}$ and $\Gamma_{\rm out}$. This number of degrees of freedom is usually much less than the total number of degrees of freedom in the problem. Should the discretisation of $K$ require a large amount of memory, matrix-free methods can be utilised to avoid constructing $K$ entirely.

\subsubsection{Imposing Sommerfeld radiation conditions for the three-dimensional channel}\label{sec: sommerfeld 3d}\noindent
We now consider imposing the Sommerfeld radiation conditions as integral operators for the case of a three-dimensional channel. 

The eigenfunction expansion solution for outside $x\in(-L_f,L_f)$ is given by 
\begin{equation}\label{eqn: phi solution outside 3d}
    \phi(x,y,z) =
    \begin{cases}
        \displaystyle\beta\psi_0(z)e^{ik_{0p}x}\chi_p(y)+\sum_{n=0}^{\infty}\sum_{m=0}^{\infty} R_{mn}\psi_m(z)\chi_n(y)e^{-ik_{mn}x},& x\leq -L_f,\\
        \displaystyle\sum_{n=0}^{\infty}\sum_{m=0}^{\infty} T_{mn}\psi_m(z)\chi_n(y)e^{ik_{mn}x},& x \geq L_f,
    \end{cases}
\end{equation}
where $p$ specifies the frequency of the incident wave in $y$ ($p=0$ is a plane wave), $k_{mn}$ are given by
\begin{equation}\label{eqn: kmn 3d}
    k_{mn} = \sqrt{k_m^2 - \left(\frac{n\pi}{2B_f}\right)^2},
\end{equation}
with $k_m$ satisfying the dispersion relation $\alpha=k_m\tanh(k_mH)$ ordered with increasing magnitude and with $k_m$ being positive imaginary for $m>0$, $\psi_m(z)=\frac{\cosh{k_m(H+z)}}{\cosh{k_mH}}$ are the eigenfunctions in $z$, $\chi_n(y)=\cos\left(\frac{n\pi (y-B_f)}{2B_f}\right)$ are the eigenfunctions in $y$, and $R_{mn}$ and $T_{mn}$ are eigen-expansion coefficients. 
As previously, $\psi_m$ satisfies the orthogonality condition in \eqref{eqn: orthog psi}. Furthermore, the eigenfunctions $\chi_n$ satisfy
\begin{equation}
    \int_{-B_f}^{B_f} \chi_m(y)\chi_n(y)~\mathrm{d} y=\delta_{mn}B_m.
\end{equation}
At $x=-L_f$ and $x=L_f$, multiplying \eqref{eqn: phi solution outside 3d} by $\psi_i\chi_j$, integrating in $y$ and $z$, applying orthogonality, and rearranging gives
\begin{equation}\label{eqn: Rij 3d}
    R_{ij}=\begin{cases}
        \displaystyle-\beta e^{-2\upi k_{0p}L_f} + \frac{e^{-\upi k_{0p}L_f}}{C_0B_p}\int_{B_f}^{B_f}\int_{-H}^0\phi(-L_f,y,z)\psi_0(z)\chi_p(y)~\mathrm{d}z~\mathrm{d}y,&i=0,j=p,\\
        \displaystyle\frac{e^{-\upi k_{ij}L_f}}{C_iB_j}\int_{B_f}^{B_f}\int_{-H}^0\phi(-L_f,y,z)\psi_i(z)\chi_j(y)~\mathrm{d}z~\mathrm{d}y,&\rm otherwise,
    \end{cases}
\end{equation}
and
\begin{equation}\label{eqn: Tij 3d}
    T_{ij}=\frac{e^{-\upi k_{ij}L_f}}{C_iB_j}\int_{B_f}^{B_f}\int_{-H}^0\phi(L_f,y,z)\psi_i(z)\chi_j(y)~\mathrm{d}z~\mathrm{d}y.
\end{equation}
Taking a normal derivative of \eqref{eqn: phi solution outside 3d} on $\Gamma_{\rm in}$ and $\Gamma_{\rm out}$ and substituting the expressions for $R_{mn}$ and $T_{mn}$ gives
\begin{align}
    \partial_{\boldsymbol{n}}\phi =\begin{cases} -2\upi k_{0p}\beta\psi_0(z)e^{-\upi k_{0p}L}\chi_p(y) + \sum_{n=0}^\infty\sum_{m=0}^{\infty} \frac{\upi k_{mn}\psi_m(z)\chi_n(y)}{C_mB_n}\int_{\Gamma_{\rm in}}\phi(x,y,z)\psi_m(z)\chi_n(y)~\mathrm{d}S,&\text{on }\Gamma_{\rm in},\\
     \sum_{n=0}^\infty\sum_{m=0}^{\infty} \frac{\upi k_{mn}\psi_m(z)\chi_n(y)}{C_mB_n}\int_{\Gamma_{\rm out}}\phi(x,y,z)\psi_m(z)\chi_n(y)~\mathrm{d}S,&\text{on }\Gamma_{\rm out},\end{cases}
\end{align}
where $\Gamma_{\rm in}$ and $\Gamma_{\rm out}$ are as shown in Figure~ \ref{fig:schematic_3d}. Truncating the series in $n$ at $N$ and $m$ at $M$ and substituting them into \eqref{eqn: intermediate pre SRC}, we obtain the operators $B_f$ and $L$ as given in \eqref{eqn: bilinear form fluid b} and \eqref{eqn: bilinear form fluid L}, respectively, while the integral operator $K(\phi,v)$ becomes
\begin{equation}\label{eqn: K op 3d}
    K(\phi,v) = \sum_{n=0}^{N}\sum_{m=0}^{M}\frac{\upi k_{mn}}{C_mB_n}\left[(v,\psi_m\chi_n)_{\Gamma_{\rm in}}(\phi,\psi_m\chi_n)_{\Gamma_{\rm in}} + (v,\psi_m\chi_n)_{\Gamma_{\rm out}}(\phi,\psi_m\chi_n)_{\Gamma_{\rm out}}\right].
\end{equation}

\subsection{Plate}\noindent
Next, we derive the bilinear form for the plate. For the sake of generality, we consider the three-dimensional case discussed in Section~\ref{sec: three dim channel}. Note that the two-dimensional case is analogous. We avoid using $\mathcal{C}^1$ finite elements to resolve the 4$^{\rm th}$ order operator via a continuous/discontinuous Galerkin (C/DG) formulation. This is similar to the approach discussed by \citet{OriolHydroelasticity}. We begin the derivation by multiplying \eqref{eqn: CS Gamma plate 2d} by a test function $r$ and integrating over $\Gamma$:
\begin{equation}
    0 = B_p([\phi,\w],[v,r]) = (\bnabla^2:(\Drig:\bnabla^2\w),r)_{\Gamma} - (I_0 \omega^2\w,r)_\Gamma - (\upi\omega\rho_w\jump{\phi},r)_\Gamma.
\end{equation}
Similarly to the DG formulation for the fluid, we can proceed with integration by parts on each partition $\Gamma_e$:
\begin{equation}
    (\bnabla^2:(\Drig:\bnabla^2\w),r)_{\Gamma_e} = (\Drig:\bnabla^2\w,\bnabla^2r)_{\Gamma_e} - ((\Drig:\bnabla^2\w)\cdot\boldsymbol{n}_\Gamma,\bnabla r)_{\partial\Gamma_e} + (\bnabla\cdot(\Drig:\bnabla^2\w)\cdot\boldsymbol{n}_\Gamma,r)_{\partial\Gamma_e}.
\end{equation}
$B_p$ can then be written as
\begin{equation}
\begin{aligned}
    B_p([\phi,\w],[v,r]) =&~ (\Drig:\bnabla^2\w,\bnabla^2r)_{\Gamma} - (I_0 \omega^2\w,r)_\Gamma - (\upi\omega\rho_w\jump{\phi},r)_\Gamma\\& - ((\Drig:\bnabla^2\w)\cdot\boldsymbol{n}_\Gamma,\bnabla r)_{\partial\Gamma} + (\bnabla\cdot(\Drig:\bnabla^2\w)\cdot\boldsymbol{n}_\Gamma,r)_{\partial\Gamma}\\& - (\mean{\Drig:\bnabla^2\w},\jump{\bnabla r\otimes\boldsymbol{n}_\Gamma})_{\Sigma} - (\jump{\Drig:\bnabla^2\w},\mean{\bnabla r\otimes\boldsymbol{n}_\Gamma})_{\Sigma}\\& + (\mean{\bnabla\cdot(\Drig:\bnabla^2\w)\cdot\boldsymbol{n}_\Gamma},\jump{r})_{\Sigma} + (\jump{\bnabla\cdot(\Drig:\bnabla^2\w)\cdot\boldsymbol{n}_\Gamma},\mean{r})_{\Sigma}.
\end{aligned}
\end{equation}
The last three terms vanish by requiring continuity of moments, shear, and the test function. Next, we add the symmetric DG term and penalty term
\begin{equation}
    0 = (\mean{\Drig:\bnabla^2r},\jump{\bnabla\w\otimes\boldsymbol{n}_\Gamma})_{\Sigma},~\&~0 = (\mu_2\jump{\bnabla r\otimes\boldsymbol{n}_\Gamma},\jump{\bnabla\w\otimes\boldsymbol{n}_\Gamma})_{\Sigma},
\end{equation}
where $\mu_2$ is a penalty parameter. For continuous $\bnabla\w$, these terms are consistent. We then have
\begin{equation}
\begin{aligned}
    B_p(&[\phi,\w],[v,r]) = (\Drig:\bnabla^2\w,\bnabla^2r)_{\Gamma} - (I_0 \omega^2\w,r)_\Gamma - (\upi\omega\rho_w\jump{\phi},r)_\Gamma\\& - ((\Drig:\bnabla^2\w)\cdot\boldsymbol{n}_\Gamma,\bnabla r)_{\partial\Gamma} + (\bnabla\cdot(\Drig:\bnabla^2\w)\cdot\boldsymbol{n}_\Gamma,r)_{\partial\Gamma}\\& - (\mean{\Drig:\bnabla^2\w},\jump{\bnabla r\otimes\boldsymbol{n}_\Gamma})_{\Sigma} - (\mean{\Drig:\bnabla^2r},\jump{\bnabla\w\otimes\boldsymbol{n}_\Gamma})_{\Sigma}+(\mu_2\jump{\bnabla r\otimes\boldsymbol{n}_\Gamma},\jump{\bnabla\w\otimes\boldsymbol{n}_\Gamma})_{\Sigma}.
\end{aligned}
\end{equation}
Finally, the terms in $\partial\Gamma$ depend on the chosen boundary conditions as follows.
\begin{itemize}
    \item \textbf{Free BCs}: Both terms on $\partial\Gamma$ vanish by \eqref{eqn:free 3d}, and we are left with
\begin{equation}\label{eqn: Bp f/s}
\begin{aligned}
    B_{p,\rm f/s}(&[\phi,\w],[v,r]) =~ (\Drig:\bnabla^2\w,\bnabla^2r)_{\Gamma} - (I_0 \omega^2\w,r)_\Gamma - (\upi\omega\rho_w\jump{\phi},r)_\Gamma\\& - (\mean{\Drig:\bnabla^2\w},\jump{\bnabla r\otimes\boldsymbol{n}_\Gamma})_{\Sigma} - (\mean{\Drig:\bnabla^2r},\jump{\bnabla\w\otimes\boldsymbol{n}_\Gamma})_{\Sigma}+(\mu_2\jump{\bnabla r\otimes\boldsymbol{n}_\Gamma},\jump{\bnabla\w\otimes\boldsymbol{n}_\Gamma})_{\Sigma}.
\end{aligned}
\end{equation}
    \item \textbf{Clamped BCs}: Take $\w$ and $r$ to be in $V_{\Gamma,0}=\{v\in V_\Gamma:v\rvert_{\partial\Gamma}=0\}$, then the second contribution on $\partial\Gamma$ vanishes by the support of $r$. In addition, following \citet{engel_continuousdiscontinuous_2002}, we weakly enforce $\partial_{\boldsymbol{n}_{\Gamma}}\w = 0$ on $\partial\Gamma$ using a symmetric interior penalty method. We then have
\begin{equation}\label{eqn: Bp c}
\begin{aligned}
    B_{p,\rm c}(&[\phi,\w],[v,r]) =~ (\Drig:\bnabla^2\w,\bnabla^2r)_{\Gamma} - (I_0 \omega^2\w,r)_\Gamma - (\upi\omega\rho_w\jump{\phi},r)_\Gamma\\& - ((\Drig:\bnabla^2\w)\cdot\boldsymbol{n}_\Gamma,\bnabla r)_{\partial\Gamma} - ((\Drig:\bnabla^2 r)\cdot\boldsymbol{n}_\Gamma,\bnabla w)_{\partial\Gamma} + (\mu_2(\bnabla r\otimes\boldsymbol{n}_\Gamma),(\bnabla\w\otimes\boldsymbol{n}_\Gamma))_{\partial\Gamma}\\& - (\mean{\Drig:\bnabla^2\w},\jump{\bnabla r\otimes\boldsymbol{n}_\Gamma})_{\Sigma} - (\mean{\Drig:\bnabla^2r},\jump{\bnabla\w\otimes\boldsymbol{n}_\Gamma})_{\Sigma}+(\mu_2\jump{\bnabla r\otimes\boldsymbol{n}_\Gamma},\jump{\bnabla\w\otimes\boldsymbol{n}_\Gamma})_{\Sigma}.
\end{aligned}
\end{equation} 
    \item \textbf{Simply-supported BCs}: Take $\w$ and $r$ to be in $V_{\Gamma,0}$. Then the first contribution on $\partial\Gamma$ vanishes by \eqref{eqn:simply-supported 3d} while the second vanishes by the support of $r$, giving the same operator as $B_{p,\rm f/s}([\phi,\w],[v,r])$ in \eqref{eqn: Bp f/s}.
\end{itemize}

\begin{remark}
The two-dimensional case in \eqref{eqn: CS Gamma plate} is analogous to the above. In this case, the above operators become
\begin{equation}\label{eqn: Bp f/s (2d)}
\begin{aligned}
    B_{p,\rm f/s}([\phi,\w],[v,r]) =&~ (D\Delta\w,\Delta r)_{\Gamma} - (I_0 \omega^2\w,r)_\Gamma - (\upi\omega\rho_w\jump{\phi},r)_\Gamma\\& - (\mean{D\Delta\w},\jump{\boldsymbol{n}_\Gamma\cdot\bnabla r})_{\Sigma} - (\mean{D\Delta r},\jump{\boldsymbol{n}_\Gamma\cdot\bnabla\w})_{\Sigma}+(\mu_2\jump{\bnabla r\cdot\boldsymbol{n}_\Gamma},\jump{\bnabla\w\cdot\boldsymbol{n}_\Gamma})_{\Sigma},
\end{aligned}
\end{equation}
    for the case of free/simply-supported BCs and
\begin{equation}\label{eqn: Bp c (2d)}
\begin{aligned}
    B_{p,\rm c}([\phi,\w],[v,r]) =&~ (D\Delta\w,\Delta r)_{\Gamma} - (I_0 \omega^2\w,r)_\Gamma - (\upi\omega\rho_w\jump{\phi},r)_\Gamma\\& - (D\Delta\w,\boldsymbol{n}_{\Gamma}\cdot\bnabla r)_{\partial\Gamma}- (D\Delta r,\boldsymbol{n}_{\Gamma}\cdot\bnabla w)_{\partial\Gamma}+ (\mu_2(\bnabla r\cdot\boldsymbol{n}_\Gamma),(\bnabla\w\cdot\boldsymbol{n}_\Gamma))_{\partial\Gamma}\\& - (\mean{D\Delta\w},\jump{\boldsymbol{n}_\Gamma\cdot\bnabla r})_{\Sigma} - (\mean{D\Delta r},\jump{\boldsymbol{n}_\Gamma\cdot\bnabla\w})_{\Sigma}+(\mu_2\jump{\bnabla r\cdot\boldsymbol{n}_\Gamma},\jump{\bnabla\w\cdot\boldsymbol{n}_\Gamma})_{\Sigma},
\end{aligned}
\end{equation}
    for clamped BCs.
\end{remark}

\subsection{Discrete formulation}\label{sec: discrete formulation}\noindent 
Let $\Omega_h$ be the finite element partition of $\Omega$ with partitions $\Omega_e$. Denote the dimension $d-1$ facets of $\Omega_h$ by $\mathcal{E}_h$ and the dimension $d-2$ subfacets by $\mathcal{P}_h$. Let $\Gamma_h\subset\mathcal{E}_h$ be the finite element partition of $\Gamma$ with partitions $\Gamma_e\in\mathcal{E}_h$. Finally, let $\Lambda_h$ and $\Sigma_h$ denote the discrete skeletons.

The DG + C/DG formulation for the problem is: find $(\phi_h,\w_h)\in\hat{\mathcal{V}}_h\times\hat{\mathcal{V}}_{\Gamma,h}$ such that
\begin{equation}
    B_f([\phi_h,\w_h],[v_h,r_h]) + \mu_p B_{p,\cdot}([\phi_h,\w_h],[v_h,r_h])= L(v_h),~\forall(v_h,r_h)\in\hat{\mathcal{V}}_h\times\hat{\mathcal{V}}_{\Gamma,h},
\end{equation}
where $\mu_p$ is introduced for stability and dimensional consistency, and the operator $B_{p,\cdot}$ corresponds to either $B_{p,\rm f/s}$ or $B_{p,\rm c}$ depending on the boundary conditions acting on the plate. In addition, the discrete function spaces $\hat{\mathcal{V}}$ and $\hat{\mathcal{V}}_\Gamma$ in the above are given by
\begin{align}
    \hat{\mathcal{V}}_h&=\{v\in L^2(\Omega;\mathbb{C}):v\vert_{\Omega_e}\in \mathbb{P}_r(\Omega_e;\mathbb{C})~\forall \Omega_e\in\Omega_h\},\\
    \hat{\mathcal{V}}_{\Gamma,h}&=\{v\in \mathcal{C}^0(\Gamma;\mathbb{C}):v\vert_{\Gamma_e}\in \mathbb{P}_r(\Gamma_e;\mathbb{C})~\forall \Gamma_e\in\Gamma_h\},
\end{align}
where $\mathbb{P}_r(\Omega_e;\mathbb{C})$ is the space of Lagrange polynomials of degree $r \geq 2$ in an element $\Omega_e$ with complex coefficients. Note that in the case of clamped or simply-supported BCs, the latter is replaced with
\begin{equation}
    \hat{\mathcal{V}}_{\Gamma,h}^0=\{v\in \hat{\mathcal{V}}_{\Gamma,h}:v\lvert_{\partial\Gamma}=0\}.
\end{equation}

We note that in a discrete setting the integration is done element by element, but we will use the same notation as in the continuous forms, i.e. $(u_h,v_h)_{\Omega_h}=\sum_K(u_h,v_h)_K$. Finally, with abuse of notation, we omit the subscript $h$ from all domains and treat these as discrete from this point onwards.

\section{Numerical analysis}\label{sec: analysis}\noindent
\subsection{Preliminary arguments}\label{sec:preliminary}\noindent
In the following, we prove various results for the proposed discrete formulation. To simplify the notation, we define $\xvec_h=[\phi_h,w_h]$ and $\yvec_h=[v_h,r_h]$, and the combined functional space $\X_h=\hat{\mathcal{V}}_h\times\hat{\mathcal{V}}_{\Gamma,h}$, with $\phi_h,v_h\in\hat{\mathcal{V}}_h$ and $w_h,r_h\in\hat{\mathcal{V}}_{\Gamma,h}$. Let us define the energy norm of the discrete problem for any $\xvec_h\in\X_h$ as
\begin{equation}
    \label{eq:energy_norm}
    \|\mathbf x_h\|_{\mathcal X_h}^2 = \|\bnabla\phi_h\|_{L^2(\Omega)}^2 + \|\phi_h\|_{L^2(\Omega)}^2 + J_f(\phi_h)^2
    + \mu_p\ \Big(\|\bnabla^2 w_h\|_{L^2(\Gamma)}^2 + \|w_h\|_{L^2(\Gamma)}^2 + J_p(w_h)^2\Big),
\end{equation}
with the jump terms given by
\[
J_f(v_h)^2 := \sum_{e\in\Lambda\setminus\Gamma} h_e^{-1}\big\|\jump{v_h}\big\|_{L^2(e)}^2,
\qquad
J_p(w_h)^2 := \sum_{p\in\Sigma} h_p^{-1}\bigl\|\jump{\bnabla w_h\otimes\boldsymbol n_{\mathcal P}}\bigr\|_{L^2(p)}^2.
\]

Note that throughout the analysis, we employ the notation $\bnabla(\cdot)$ and $\bnabla^2(\cdot)$ for both the classical and broken gradient and Hessian operators. The precise meaning is determined by the function space under consideration. For globally smooth functions these operators are understood in the classical sense, whereas for discontinuous finite element functions they are interpreted element-wise, i.e., as the broken gradient and broken Hessian.

For the sake of completeness, we introduce new notation that will be used in the analysis of the discrete formulation. Let us consider a triangulation $\mathcal{T}_h$, from which we construct the finite element partition $\Omega_h$ (denoted also as $\Omega$ in Section~\ref{sec: analysis} for notation simplicity), and a given element of the triangulation $K\in\mathcal{T}_h$ of diameter $h_K$ and a facet of this element $\mathcal{E}_K\in\partial K$ with outwards-pointing normal $\mathbf{n}_{\mathcal{E}_K}$. Note the abuse of notation in $K$, which will be used hereinafter to denote both an element in the triangulation $\mathcal{T}_h$ and the global radiation operator \eqref{eqn: K op 2d}/\eqref{eqn: K op 3d}. The conforming and quasi-uniform triangulation corresponding to the structure finite element partition will be denoted by $\mathcal{T}_h^\Gamma$. We denote by $h$ the domain characteristic element size, given by $h=\max_{K\in\mathcal{T}_h}h_K$, and quasi-uniformity constant $c_0\in(0,1]$ such
that $h_e\ge c_0h$ for every facet $e$.

In the analysis that follows, we consider the sesquilinear weak form for the case of free plate boundary conditions, equation~\eqref{eqn: Bp f/s}, and radiation boundary conditions, given by
\begin{equation}
\label{eq:bilinear_form_analysis}
\begin{aligned}
B([\phi,\w],[v,r]) = &~ (\bnabla\phi,\bnabla v)_{\Omega} - (\alpha\phi,v)_{\Gamma_{\rm fs}} - K(\phi,v) \\&- (\mean{\partialn\phi},\jump{v})_{\Lambda\setminus\Gamma} - (\mean{\partialn v},\jump{\phi})_{\Lambda\setminus\Gamma} + (\mu_1\jump{\phi},\jump{v})_{\Lambda\setminus\Gamma} + (\upi\omega\w,\jump{v})_{\Gamma}\\ & 
+\mu_p {\bigg{[}}(\Drig:\bnabla^2\w,\bnabla^2r)_{\Gamma} - (I_0 \omega^2\w,r)_\Gamma - (\upi\omega\rho_w\jump{\phi},r)_\Gamma\\& - (\mean{\Drig:\bnabla^2\w},\jump{\bnabla r\otimes\boldsymbol{n}_\Gamma})_{\Sigma} - (\mean{\Drig:\bnabla^2r},\jump{\bnabla\w\otimes\boldsymbol{n}_\Gamma})_{\Sigma}\\
&~+(\mu_2\jump{\bnabla r\otimes\boldsymbol{n}_\Gamma},\jump{\bnabla\w\otimes\boldsymbol{n}_\Gamma})_{\Sigma}{\bigg{]}}.
\end{aligned}
\end{equation}
Note that similar results can be obtained for the case of clamped or simply-supported boundary conditions. For the sake of brevity, we will not include them in the derivations that follow.

\begin{remark}
    We will use the same bilinear form definition, $B(\cdot,\cdot)$ given by equation~\eqref{eq:bilinear_form_analysis}, for both, continuous and discrete finite element spaces. Note that for the case of continuous spaces, i.e. $\xvec\in\X$, the terms involving jumps on $\Lambda\setminus\Gamma$ and $\Sigma$ vanish.
\end{remark}

We consider a set of assumptions that will be required throughout the analaysis section to establish the well-posedness of the discrete solution. Note that these are reasonable assumptions that do not enforce strong constraints to the numerical solution of most of the problems of interest.

\begin{assumption}[Bending rigidity tensor]\label{ass:rigidity}
We assume that the bending rigidity tensor $\Drig$ is bounded and uniformly elliptic,
\[
\mathbf x\cdot\Drig\cdot\mathbf x \ge \gamma_{\min}\|\mathbf x\|^2 \quad\text{a.e.\ on }\Gamma, \qquad
\|\Drig\|_{L^\infty(\Gamma)}<\infty.
\]
\end{assumption}

\begin{assumption}[Modal data of the global radiation operator]\label{ass:modal}
$\Gamma_{\rm in}$ and $\Gamma_{\rm out}$ share a common cross section, on which $\{\psi_m\}_{m=0}^M$ is a fixed,
finite family, mutually orthogonal in $L^2(\Omega)$ of that cross section, with $C_m:=\|\psi_m\|^2_{L^2(\Omega)}>0$. The axial
wavenumbers $k_m\in\mathbb C$, $k_m^2=k^2-\lambda_m$ (with $\lambda_m$ the transverse eigenvalue associated with $\psi_m$, $k=\omega/c$), are fixed by the outgoing/decaying branch convention $\mathrm{Im}(k_m)\ge0$, with
$k^2\ne\lambda_m$ for every $m\le M$.
\end{assumption}

\begin{assumption}[Interface splitting]\label{ass:gamma_splitting}
The interface $\Gamma$ is a Lipschitz hypersurface such that $\Omega\setminus\Gamma=\Omega^+\cup\Omega^-$
with $\Omega^+,\Omega^-$ disjoint, bounded, Lipschitz domains satisfying
$\overline{\Omega^+}\cap\overline{\Omega^-}=\overline\Gamma$. The triangulation $\mathcal T_h$ is compatible
with this splitting: $\mathcal T_h=\mathcal T_h^+\cup\mathcal T_h^-$, with $\mathcal T_h^\pm$ a shape-regular,
conforming triangulation of $\Omega^\pm$ satisfying Buffa--Ortner's Assumption~1
\citep{buffa2009compact}, and $\Lambda\setminus\Gamma=\Lambda^+\cup\Lambda^-$ where $\Lambda^\pm$ is the
interior skeleton of $\mathcal T_h^\pm$.
\end{assumption}

\begin{remark}
\autoref{ass:gamma_splitting} holds automatically whenever $\partial\Gamma\subset\partial\Omega$, that is, when the plate spans the fluid domain. If $\Gamma$ has a free interior
endpoint, $\Omega^+$ or $\Omega^-$ develops a non-Lipschitz cusp there and \autoref{ass:gamma_splitting}
fails locally. We consider the numerical analysis of this case out of the scope of this work, see Remark~\ref{rem:tips}.
\end{remark}

\begin{assumption}[Regularity of the exact solution]\label{ass:regularity}
In addition to \autoref{ass:gamma_splitting}, the exact solution $\xvec=[\phi,w]$ of the continuous
variational problem underlying~\eqref{eq:bilinear_form_analysis} satisfies
\[
\phi\in H^1(\Omega\setminus\Gamma), \qquad \phi\big|_K\in H^2(K)\ \ \forall K\in\mathcal T_h, \qquad
w\in C^1(\Gamma), \qquad w\big|_E\in H^3(E)\ \ \forall E\in\mathcal T_h^\Gamma,
\]
and solves the strong form of~\eqref{eqn: CS Omega}--\eqref{eqn:free} in each
element, with $\partial_n\phi$ and $\Drig:\nabla^2w$ single-valued across every interior facet of
$\mathcal T_h$ and $\mathcal T_h^\Gamma$ respectively; in particular $\jump\phi=0$ on $\Lambda\setminus\Gamma$
and $\jump{\nabla w\otimes\boldsymbol n_\Gamma}=0$ on $\Sigma$.
\end{assumption}

\begin{remark}[Free interior plate tips]\label{rem:tips}
If $\Gamma$ has a free interior endpoint, $\Omega^+$ or $\Omega^-$ of \autoref{ass:gamma_splitting} develops a non-Lipschitz slit-tip. Then, $\phi,w$
generically fail to be $H^2/H^3$ near the tip. In that case, \autoref{ass:regularity} and the quasi-optimal estimate
\eqref{eq:cea_final} do not apply directly. A rigorous treatment of the estimate requires further extension following, for example, the work in \cite{chaumont2020wavenumber}, which is left for future work. Note also
that the stability results of
Sections~\ref{sec:discrete_continuity}--\ref{sec:garding} and~\ref{sec:wellposedness}(i)--(ii) require no
regularity of the exact solution and are unaffected.
\end{remark}

\begin{assumption}[Continuous non-resonance]\label{ass:continuous_nonresonance}
Let $\X:=H^1(\Omega\setminus\Gamma)\times H^2(\Gamma)$ and let
$B:\X\times \X\to\mathbb C$ denote the continuous sesquilinear form
associated with the coupled hydroelastic problem and the truncated
radiation operator (Eqs.~\eqref{eqn: bilinear form fluid b},
\eqref{eqn: Bp f/s}/\eqref{eqn: Bp c}). For the frequency $\omega$
under consideration, the homogeneous problem $B(\xvec,\yvec)=0\ \forall \yvec\in \X$
admits only the trivial solution $\xvec=0$.
\end{assumption}

In addition to the previous assumptions, the following standard theorems and related lemmas will be used in the subsequent subsections. We refer the reader to \cite{arnold2002unified,van2005discontinuous}, and references therein, for further details on the analysis of similar problems.

\begin{theorem}[Local trace theorem]\label{theorem:trace}
    For every element $K\in\mathcal T_h$ of diameter $h_K$, facet $\mathcal E_K\subset\partial K$, and $\phi_h\in H^1(K)$, there is $C_{\mathrm{tr}}>0$, depending only on the shape regularity of $\mathcal T_h$, such that
    \begin{equation}
        \label{eq:trace_theorem}
        \left\|\phi_h\right\|^2_{L^2(\mathcal{E}_K)}\leq \Ctr\left(h_K^{-1}\left\|\phi_h\right\|^2_{L^2(K)}+h_K\left\|\bnabla\phi_h\right\|^2_{L^2(K)}\right).
    \end{equation}
\end{theorem}
\begin{theorem}[Inverse inequality for fluxes]\label{theorem:inverse_inequality}
    For any $\phi_h\in\hat{\mathcal{V}}_h$, its gradient $\bnabla\phi_h$ is a polynomial on $K$. Then, there exists a constant $C_{\rm inv}$ such that the following inverse inequality holds
    \begin{equation}
        \label{eq:inverse_inequality}\left\|\bnabla\phi_h\cdot\mathbf{n}_{\mathcal{E}_K}\right\|^2_{L^2(\mathcal{E}_K)}\leq C_{\rm inv}h_K^{-1}\left\|\bnabla\phi_h\right\|^2_{L^2(K)}.
    \end{equation}
\end{theorem}
\begin{lemma}[Inverse trace inequality for the Hessian]\label{lemma:inverse_sigma}
For $w_h\in\hat{\mathcal V}_{\Gamma,h}$ and $p=\partial E_1\cap\partial E_2\in\Sigma$ with $E_1,E_2\in\mathcal{T}_h^\Gamma$, there is
$C_{\mathrm{inv}}^\Sigma>0$, in general different from $C_{\mathrm{inv}}$ of \autoref{theorem:inverse_inequality}, such that
\begin{equation}
\big\|\mean{\Drig:\bnabla^2 w_h}\big\|_{L^2(p)} \le \big(C_{\mathrm{inv}}^\Sigma\big)^{1/2} h_p^{-1/2}\,
\|\Drig\|_{L^\infty(\Gamma)}\,\|\bnabla^2 w_h\|_{L^2(E_1\cup E_2)}.
\end{equation}
\end{lemma}
\begin{proof}
Standard polynomial inverse estimate applied to $\left\|\bnabla^2 w_h\right\|_{L^2(p)}$ on each $p=E_1\cap E_2$, combined with the
 inequality $\|\mean{\Drig:\bnabla^2 w_h}\big\|_{L^2(p)} \le \|\Drig\|_{L^\infty(\Gamma)}\,\|\bnabla^2 w_h\|_{L^2(p)}$.
\end{proof}

\begin{lemma}[Fluid recovery operator]\label{lem:fluid_recovery}
There exists $C_R>0$, independent of $h$, and a linear operator
$R_h:\hat{\mathcal V}_h\to H^1(\Omega\setminus\Gamma)$, constructed side-wise on $\Omega^+,\Omega^-$ as per \autoref{ass:gamma_splitting}, following the reconstruction operator as proposed in 
\citet[Eq.~(3.6)]{buffa2009compact}, such that
\begin{equation}
\label{eq:fluid_recovery}
\begin{aligned}
\|v_h - R_hv_h\|_{L^2(\Omega)} \;&\le\; C_R\,h\,\big(\|\bnabla_h v_h\|_{L^2(\Omega)}+J_f(v_h)\big),\\
\|\bnabla(v_h-R_hv_h)\|_{L^2(\Omega)} \;&\le\; C_R\,\big(\|\bnabla_h v_h\|_{L^2(\Omega)}+J_f(v_h)\big).
\end{aligned}
\end{equation}
\end{lemma}
Using \autoref{ass:gamma_splitting}, since $\mathcal T_h$ is a shape-regular and conforming triangulation of the bounded Lipschitz domain $\Omega^\pm$, \autoref{lem:fluid_recovery} follows from \cite{buffa2009compact}. The full description of the proof is omitted here for brevity.

\begin{lemma}[Plate recovery operator]\label{lem:plate_recovery}
There exists $C_R^\Gamma>0$, independent of $h$, and a linear operator $S_h:\hat{\mathcal V}_{\Gamma,h}\to
H^2(\Gamma)$ such
that, for every $w_h\in\hat{\mathcal V}_{\Gamma,h}$,
\begin{equation}
\label{eq:plate_recovery}
\begin{aligned}
\|\bnabla w_h-\bnabla S_hw_h\|_{L^2(\Gamma)}
\;&\le\;
C_R^\Gamma\,h\,\big(\|\bnabla^2 w_h\|_{L^2(\Gamma)}+J_p(w_h)\big),\\
\|\bnabla^2 w_h-\bnabla^2 S_hw_h\|_{L^2(\Gamma)}
\;&\le\;
C_R^\Gamma\,\big(\|\bnabla^2 w_h\|_{L^2(\Gamma)}+J_p(w_h)\big),
\end{aligned}
\end{equation}
where $\bnabla^2$ denotes the broken Hessian.
\end{lemma}

\autoref{lem:plate_recovery} follows from \cite{buffa2009compact}, applied component-wise to the broken gradient field $\bnabla w_h$.

\begin{lemma}[DG compactness]\label{lem:DGcompactness}
Let \autoref{ass:gamma_splitting} hold, and let $\{(\phi_h,w_h)\}_{h>0}\subset\X_h$ be a sequence
satisfying
\[
\sup_h\Big(
\|\bnabla\phi_h\|_{L^2(\Omega)}+
J_f(\phi_h)+
\|\bnabla^2w_h\|_{L^2(\Gamma)}+
J_p(w_h)
\Big)<\infty.
\]
Then there exist $\phi\in H^1(\Omega\setminus\Gamma)$ with no
continuity required across $\Gamma$, and $w\in H^2(\Gamma)$, and a not relabelled subsequence such that
\[
R_h\phi_h \rightharpoonup \phi \ \text{in } H^1(\Omega\setminus\Gamma),
\qquad
R_h\phi_h \to \phi \ \text{in } L^2(\Omega),
\]
\[
S_hw_h \rightharpoonup w\ \text{in } H^2(\Gamma),
\qquad
S_hw_h \to w\ \text{in } H^1(\Gamma).
\]
Moreover $\phi_h-R_h\phi_h \to 0$ in $L^2(\Omega)$ and $\bnabla w_h-\bnabla S_hw_h \to0$ in $L^2(\Gamma)$. Consequently, $\phi_h \to \phi$  strongly in $L^2(\Omega)$ and $\nabla w_h \to \nabla w$ strongly in $L^2(\Gamma)$.
\end{lemma}

\begin{proof}
By Lemmas~\ref{lem:fluid_recovery} and \ref{lem:plate_recovery}, the recovery errors vanish strongly for $h\to0$. That is $\|\phi_h-R_h\phi_h\|_{L^2(\Omega)}\le
C_Rh(\|\bnabla\phi_h\|_{L^2(\Omega)}+J_f(\phi_h))\to0$ since $h\to0$ and the terms in the parenthesis are bounded, with identical argument for $S_hw_h$. The assumed bounds imply that $\{R_h\phi_h\}$ is bounded in $H^1(\Omega\setminus\Gamma)$, as long as \autoref{ass:gamma_splitting} is valid, and $\{S_hw_h\}$ is bounded in $H^2(\Gamma)$. The compact embeddings
\[
H^1(\Omega\setminus\Gamma)\hookrightarrow L^2(\Omega),
\qquad
H^2(\Gamma)\hookrightarrow H^1(\Gamma),
\]
therefore yield the stated convergent subsequences. See also \citet[Theorem 5.2]{buffa2009compact} for an in-depth analysis of an equivalent result. 
\end{proof}

\subsection{Discrete continuity}\label{sec:discrete_continuity}\noindent
Using the trace theorem, \autoref{theorem:trace}, the inverse inequality for fluxes, \autoref{theorem:inverse_inequality}, the inverse inequality for the Hessian, \autoref{lemma:inverse_sigma}, and the fluid and plate recovery operators, \autoref{lem:fluid_recovery} and \autoref{lem:plate_recovery}, we can proof that the sesquilinear form $B(\xvec_h,\yvec_h)$ is continuous as demonstrated in the following theorem.

\begin{theorem}[Discrete continuity]\label{theorem:discrete_continuity}
    Let $\mathcal{T}_h$ be a shape-regular mesh of a bounded Lipschitz domain $\Omega \subset \mathbb{R}^n$ constituting the triangulation of the finite element partition $\Omega_h$. Let $\X_h=\hat{\mathcal{V}}_h\times\hat{\mathcal{V}}_{\Gamma,h}$ be the space equipped with the discrete norm $\|\cdot\|_{\X_h}$. Then, the sesquilinear form $B(\cdot, \cdot): \X_h \times \X_h \to \mathbb{C}$ is continuous. That is, there exists a mesh-independent constant $M > 0$ such that:$$|B(\xvec_h, \yvec_h)| \leq M \|\xvec_h\|_{\X_h} \|\yvec_h\|_{\X_h} \quad \forall \xvec_h, \yvec_h \in \X_h.$$
\end{theorem}

\begin{proof}
To establish the continuity of the coupled sesquilinear form $B(\xvec_h, \yvec_h)$, we apply the triangle inequality to decompose the global expression into different integration domains.
\begin{equation*}
    |B(\xvec_h, \yvec_h)| \leq |B_{\Omega}(\phi_h, v_h)| + |B_{\partial\Omega}(\phi_h, v_h)| + |B_{\Lambda}(\phi_h, v_h)| + |B_{\Gamma}(\phi_h, \w, v_h, r_h)| + |B_{\Sigma}(\w, r_h)|, 
\end{equation*}
with 
\begin{align*}
    B_{\Omega}(\phi_h, v_h) =&\   (\bnabla\phi_h,\bnabla v_h)_{\Omega},\\
    B_{\partial\Omega}(\phi_h, v_h) =&\ (\alpha\phi_h,v_h)_{\Gamma_{\rm fs}} + K(\phi_h,v_h),\\
    B_{\Lambda}(\phi_h, v_h) =&\ - (\mean{\partialn\phi_h},\jump{v_h})_{\Lambda\setminus\Gamma} - (\mean{\partialn v_h},\jump{\phi_h})_{\Lambda\setminus\Gamma} + (\mu_1\jump{\phi_h},\jump{v_h})_{\Lambda\setminus\Gamma},\\
    B_{\Gamma}(\phi_h, \w, v_h, r_h) =&\ (i\omega\w,\jump{v_h})_{\Gamma} +
\mu_p (\Drig:\bnabla^2\w_h,\bnabla^2r_h)_{\Gamma} - \mu_p(I_0 \omega^2\w_h,r_h)_\Gamma - \mu_p(\upi\omega\rho_w\jump{\phi_h},r_h)_\Gamma,\\ 
    B_{\Sigma}(\w_h, r_h) =&\ - (\mean{\Drig:\bnabla^2\w_h},\jump{\bnabla r_h\otimes\boldsymbol{n}_\Gamma})_{\Sigma} + (\mean{\Drig:\bnabla^2r_h},\jump{\bnabla\w_h\otimes\boldsymbol{n}_\Gamma})_{\Sigma}.
\end{align*}
We bound each constituent term individually as follows:

\medskip
\noindent\textbf{Step 1 (Fluid interior term, $B_{\Omega}$).}
Let us start with the term corresponding to the interior integral, recalling that $|B_{\Omega}(\phi_h, v_h)| = |(\bnabla\phi_h,\bnabla v_h)_{\Omega}|=\left|\sum_{K\in\mathcal{T}_h}(\bnabla\phi_h,\bnabla v_h)_{K}\right|$. Applying the Cauchy-Schwarz inequality, we have the sharp bound
\begin{equation}
    \label{eq:proof_continuity_1}
    |(\bnabla\phi_h,\bnabla v_h)_{\Omega}| \leq \|\bnabla\phi_h\|_{L^2(\Omega)} \|\bnabla v_h\|_{L^2(\Omega)} \leq \|\xvec_h\|_{\X_h} \|\yvec_h\|_{\X_h}.
\end{equation}

\medskip
\noindent\textbf{Step 2 (Fluid boundary terms, $B_{\partial\Omega}$).}
We first bound the global operator $K(\phi_h,v_h)$. For a given boundary $\Gamma_{\rm in}$ and for $m=0,\dots,M$, we define 
$a_m:=\big(|k_m|/C_m\big)^{1/2}\big|(v_h,\psi_m)_{\Gamma_{\rm in}}\big|$ and
$b_m:=\big(|k_m|/C_m\big)^{1/2}\big|(\phi_h,\psi_m)_{\Gamma_{\rm in}}\big|$. By the triangle inequality and the discrete Cauchy--Schwarz inequality applied to the finite sum over $m$,
\[
\bigg|\sum_{m=0}^M\frac{k_m}{C_m}(v_h,\psi_m)_{\Gamma_{\rm in}}(\phi_h,\psi_m)_{\Gamma_{\rm in}}\bigg| \le
\sum_{m=0}^M a_mb_m \le \Big(\sum_{m=0}^M a_m^2\Big)^{1/2}\Big(\sum_{m=0}^M b_m^2\Big)^{1/2}.
\]
Let us define $\gamma_K:=\max_{0\le m\le M}|k_m|$. Then, $\sum_m a_m^2 = \sum_m\frac{|k_m|}{C_m}\big|(v_h,\psi_m)_{\Gamma_{\rm
in}}\big|^2 \le \gamma_K\sum_m\frac{\big|(v_h,\psi_m)_{\Gamma_{\rm in}}\big|^2}{C_m} \le
\gamma_K\|v_h\|^2_{L^2(\Gamma_{\rm in})}$, the last step by Bessel's inequality for the orthonormal system $\{\psi_m/\sqrt{C_m}\}_{m=0}^M$. Similarly, the same bound can be obtained for the sum on $b_m$, i.e. $\sum_m b_m^2\le
\gamma_K\|\phi_h\|^2_{L^2(\Gamma_{\rm in})}$. Hence
\[
\Big|\sum_{m=0}^M\frac{k_m}{C_m}(v_h,\psi_m)_{\Gamma_{\rm in}}(\phi_h,\psi_m)_{\Gamma_{\rm in}}\Big| \le \gamma_K\,
\|\phi_h\|_{L^2(\Gamma_{\rm in})}\|v_h\|_{L^2(\Gamma_{\rm in})}.
\]
Using \autoref{ass:modal}, we can reuse the same $\{\psi_m\},\{C_m\}$ for the outlet boundary $\Gamma_{\rm out}$, leading to the global bound
\[
|K(\phi_h,v_h)| \le \gamma_K\Big(\|\phi_h\|_{L^2(\Gamma_{\rm in})}\|v_h\|_{L^2(\Gamma_{\rm in})} +
\|\phi_h\|_{L^2(\Gamma_{\rm out})}\|v_h\|_{L^2(\Gamma_{\rm out})}\Big) \le \gamma_K\,\|\phi_h\|_{L^2(\Gamma_{\rm
in}\cup\Gamma_{\rm out})}\|v_h\|_{L^2(\Gamma_{\rm in}\cup\Gamma_{\rm out})}.
\]
Note that if the inlet and outlet are not aligned, the coefficient on the norms in the final inequality becomes $\max(\gamma_{K_{\rm in}},\gamma_{K_{\rm out}})$ where $\gamma_{K_{\rm in}}$ and $\gamma_{K_{\rm out}}$ are the same as $\gamma_K$, except $k_m$ is calculated using the relevant height of the inlet and outlet, respectively.

Applying Cauchy-Schwarz on the boundary facets belonging to $\Gamma_{\rm fs}$ together with the bound above on
$\Gamma_{\rm in}\cup\Gamma_{\rm out}$, we have that
\begin{align*}
    |(\alpha\phi_h,v_h)_{\Gamma_{\rm fs}} + K(\phi_h,v_h)| \leq&\ |(\alpha\phi_h,v_h)_{\Gamma_{\rm fs}}| + |K(\phi_h,v_h)|\\
    \leq&\ |\alpha|\, \|\phi_h\|_{L^2(\Gamma_{\rm fs})} \|v_h\|_{L^2(\Gamma_{\rm fs})} + \gamma_K\,
    \|\phi_h\|_{L^2(\Gamma_{\rm in}\cup\Gamma_{\rm out})} \|v_h\|_{L^2(\Gamma_{\rm in}\cup\Gamma_{\rm out})}\\
    \leq&\ \gamma\, \|\phi_h\|_{L^2(\partial\Omega)} \|v_h\|_{L^2(\partial\Omega)},
\end{align*}
with $\gamma = \max\left\{|\alpha|,\gamma_K\right\}$. 

Let us now consider the identity $\phi_h = R_h\phi_h + (\phi_h-R_h\phi_h)$
on $\partial\Omega$. Since $\phi_h-R_h\phi_h$ is a polynomial
on each element, the local trace theorem summed over
$K\cap\partial\Omega\ne\emptyset$, together with
\autoref{lem:fluid_recovery}, gives
\[
\begin{aligned}
  \|\phi_h-R_h\phi_h\|^2_{L^2(\partial\Omega)}
  &\le C_{\rm tr}\Big(h^{-1}\|\phi_h-R_h\phi_h\|^2_{L^2(\Omega)}
  + h\|\bnabla(\phi_h-R_h\phi_h)\|^2_{L^2(\Omega)}\Big)\\
  &\le 2C_{\rm tr}C_R^2\,h\,\big(\|\bnabla\phi_h\|_{L^2(\Omega)}+J_f(\phi_h)\big)^2
  \le 4C_{\rm tr}C_R^2\,h\,\big(\|\bnabla\phi_h\|^2_{L^2(\Omega)}+J_f(\phi_h)^2\big),
\end{aligned}
\]
using $\|\bnabla_h\phi_h\|_{L^2(\Omega)}=\|\bnabla\phi_h\|_{L^2(\Omega)}$ for the broken polynomial $\phi_h$.

Since $R_h\phi_h\in H^1(\Omega\setminus\Gamma)$
and $\partial\Omega$ is disjoint from the internal interface $\Gamma$, the
\emph{continuous} trace inequality applies with an $h$-independent
constant $C_{\rm tr}^{\partial\Omega}$: for any fixed $\delta>0$,
\[
  \|R_h\phi_h\|^2_{L^2(\partial\Omega)}
  \le \delta\|\bnabla R_h\phi_h\|^2_{L^2(\Omega)} + C_\delta\|R_h\phi_h\|^2_{L^2(\Omega)}.
\]
By the triangle inequality applied to $\|\bnabla R_h\phi_h+\bnabla \phi_h-\bnabla \phi_h\|_{L^2(\Omega)}$ and $\|R_h\phi_h+\phi_h-\phi_h\|_{L^2(\Omega)}$ and applying \autoref{lem:fluid_recovery} (using $h\le1$),
\[
\|\bnabla R_h\phi_h\|_{L^2(\Omega)} \le (1+C_R)\|\bnabla\phi_h\|_{L^2(\Omega)} + C_RJ_f(\phi_h),\qquad\|R_h\phi_h\|_{L^2(\Omega)} \le \|\phi_h\|_{L^2(\Omega)} + C_R\|\bnabla\phi_h\|_{L^2(\Omega)} + C_RJ_f(\phi_h).
\]
Squaring these bounds via $(a+b)^2\le2a^2+2b^2$ gives
\[
  \|\bnabla R_h\phi_h\|^2_{L^2(\Omega)} \le 2(1+C_R)^2\|\bnabla\phi_h\|^2_{L^2(\Omega)}+2C_R^2J_f(\phi_h)^2,
\]
\[
  \|R_h\phi_h\|^2_{L^2(\Omega)} \le 2\|\phi_h\|^2_{L^2(\Omega)}+4C_R^2\|\bnabla\phi_h\|^2_{L^2(\Omega)}+4C_R^2J_f(\phi_h)^2.
\]
Substituting these into the continuous trace inequality with $\delta:=1$,
\[
  \|R_h\phi_h\|^2_{L^2(\partial\Omega)}
  \le \big[2(1+C_R)^2+4C_\delta C_R^2\big]\|\bnabla\phi_h\|^2_{L^2(\Omega)}
  + 2C_\delta\|\phi_h\|^2_{L^2(\Omega)}
  + 2C_R^2(1+2C_\delta)\,J_f(\phi_h)^2,
\]
with all constants $h$-independent.

Combining the two bounds, $\|R_h\phi_h\|^2_{L^2(\partial\Omega)}$ and $\|\phi_h-R_h\phi_h\|^2_{L^2(\partial\Omega)}$, again via
$(a+b)^2\le2a^2+2b^2$ and using $h\le1$, we have
\[
  \|\phi_h\|^2_{L^2(\partial\Omega)} \le \widehat C_1^2\,\|x_h\|^2_{\X_h},
  \qquad
  \widehat C_1^2 := \max\Big(4(1+C_R)^2+8(C_\delta+C_{\rm tr})C_R^2,\ \ 4C_\delta\Big).
\]
The identical argument applied to $v_h$ gives
$\|v_h\|^2_{L^2(\partial\Omega)}\le\widehat C_1^2\|y_h\|^2_{X_h}$.
Substituting into the inequality for $|B_{\partial\Omega}(\phi_h,v_h)|$, we have
\[
  |B_{\partial\Omega}(\phi_h,v_h)| \le \bar C_1\,\|x_h\|_{X_h}\|y_h\|_{X_h},
  \qquad \bar C_1 := \gamma\widehat C_1^2 = O(1).
\]

\medskip
\noindent\textbf{Step 3 (Fluid skeleton terms, $B_{\Lambda}$).}
Regarding the fluid skeleton terms on $\Lambda\setminus\Gamma$, we again apply the Cauchy-Schwarz inequality face-by-face, giving 
\begin{equation*}
    |(\mean{\partialn\phi_h},\jump{v_h})_{\Lambda\setminus\Gamma}| \leq \left(\sum_{e \in \Lambda\setminus\Gamma} h_e \|\mean{\partialn\phi_h}\|_{L^2(e)}^2\right)^{1/2} \left(\sum_{e \in \Lambda\setminus\Gamma} h_e^{-1} \left\|\jump{v_h}\right\|_{L^2(e)}^2\right)^{1/2}.
\end{equation*}
Applying the inverse inequality for fluxes (\autoref{theorem:inverse_inequality}) to the first term, we control the normal gradient by the element volume gradient, i.e.,  $h_e \|\mean{\partialn\phi_h}\|_{L^2(e)}^2 \leq C_{\rm inv} \|\bnabla\phi_h\|_{L^2(K_1\cup K_2)}^2$. Using the fact that each element $K$ contributes to at most $N_f$ faces (a shape-regularity constant, so that $\sum_e\|\bnabla\phi_h\|^2_{K_1\cup
K_2}\le N_f\|\bnabla\phi_h\|^2_{\Omega}$), we have that
\begin{equation*}
    |(\mean{\partialn\phi_h},\jump{v_h})_{\Lambda\setminus\Gamma}| \leq \big(C_{\mathrm{inv}}N_f\big)^{1/2}\|\bnabla\phi_h\|_{\Omega}\,J_f(v_h) \le (C_{\mathrm{inv}}N_f)^{1/2}\|\xvec_h\|_{\X_h}\|\yvec_h\|_{\X_h}.
\end{equation*}
Using the same arguments, the symmetric term $(\mean{\partialn v},\jump{\phi})_{\Lambda\setminus\Gamma}$ has the identical upper bound.

For the penalty term, discrete Cauchy--Schwarz and $\mu_1=\sigma_1h^{-1}\le\sigma_1h_e^{-1}$ (since
$h_e\le h$) give
\[
\big|(\mu_1\jump{\phi_h},\jump{v_h})_{\Lambda\setminus\Gamma}\big| \le \sigma_1\sum_eh_e^{-1}\left\|\jump{\phi_h}\right\|_{L^2(e)}
\left\|\jump{v_h}\right\|_{L^2(e)} 
\le \sigma_1\|\mathbf x_h\|_{\mathcal X_h}\|\mathbf y_h\|_{\mathcal X_h}.
\]
Collecting the three pieces, with $C_2:=2(C_{\mathrm{inv}}N_f)^{1/2}$,
\[
|B_\Lambda| \le \big(C_2 + \sigma_1\big)\|\mathbf x_h\|_{\mathcal X_h}\|\mathbf y_h\|_{\mathcal X_h} .
\]

\medskip
\noindent\textbf{Step 4 (Fluid--structure coupling terms).}
For the fluid-structure coupling terms acting on the plate boundary $\Gamma$, we apply a similar argument as in the free surface boundary. Let $\Omega^+,\Omega^-$ denote the two sides of the interface $\Gamma$. \autoref{ass:regularity} permits $\phi\in H^1(\Omega\setminus\Gamma)$,
i.e.\ a genuine jump across $\Gamma$, therefore, $R_h\phi_h$, valued in $H^1(\Omega\setminus\Gamma)$, is not assumed single-valued at $\Gamma$. By the continuous trace theorem applied separately on
$\Omega^+$ and $\Omega^-$, with constants $C_{\rm tr}^{\Omega^\pm}$, both
$h$-independent, and the triangle inequality,
\[
  \|\jump{R_hv_h}\|_{L^2(\Gamma)}
  \le \|R_hv_h|_{\Omega^+}\|_{L^2(\Gamma)} + \|R_hv_h|_{\Omega^-}\|_{L^2(\Gamma)}
  \le 2C_{\rm tr}^\Gamma\|R_hv_h\|_{H^1(\Omega)},
  \qquad C_{\rm tr}^\Gamma:=\max(C_{\rm tr}^{\Omega^+},C_{\rm tr}^{\Omega^-}).
\]
By \autoref{lem:fluid_recovery}, squaring the bounds
$\|\nabla R_hv_h\|_{L^2(\Omega)}\le(1+C_R)\|\nabla v_h\|_{L^2(\Omega)}+C_RJ_f(v_h)$ and
$\|R_hv_h\|_{L^2(\Omega)}\le\|v_h\|_{L^2(\Omega)}+C_R\|\nabla v_h\|_{L^2(\Omega)}+C_RJ_f(v_h)$
via $(a+b)^2\le2a^2+2b^2$ exactly as in Step~2, we obtain
\[
\|R_hv_h\|^2_{H^1(\Omega)} \;\le\; 2\|v_h\|^2_{L^2(\Omega)}
+\big[2(1+C_R)^2+4C_R^2\big]\|\nabla v_h\|^2_{L^2(\Omega)}+6C_R^2J_f(v_h)^2
\;\le\; C_{R,1}\,\|y_h\|^2_{X_h},
\]
with $C_{R,1}:=\max\big(2,\,2(1+C_R)^2+4C_R^2,\,6C_R^2\big)$, hence
\[
  \|\jump{R_hv_h}\|_{L^2(\Gamma)} \le \widetilde C\,\|y_h\|_{X_h},
  \qquad \widetilde C := 2C_{\rm tr}^\Gamma\sqrt{C_{R,1}}.
\]
For the non-conforming remainder, the same trace argument as in Step~2 gives, using again \autoref{lem:fluid_recovery} and $h\le1$,
\[
\|\jump{v_h-R_hv_h}\|^2_{L^2(\Gamma)} \le 4C_{\rm tr}C_R^2\,h\,\big(\|\nabla v_h\|^2_{L^2(\Omega)}+J_f(v_h)^2\big)
\le 4C_{\rm tr}C_R^2\,\|y_h\|^2_{X_h},
\]
therefore
\[
\|\jump{v_h-R_hv_h}\|_{L^2(\Gamma)} \le 2\sqrt{C_{\rm tr}}\,C_R\,\|y_h\|_{X_h}.
\]
By the triangle
inequality, $\jump{v_h}=\jump{R_hv_h}+\jump{v_h-R_hv_h}$, then
\[
  \|\jump{v_h}\|_{L^2(\Gamma)} \le \widetilde C'\,\|y_h\|_{X_h},
  \qquad \widetilde C' := \widetilde C + 2\sqrt{C_{\rm tr}}\,C_R = O(1).
\]
The identical argument (with $\phi_h$ in place of $v_h$) gives
$\|\jump{\phi_h}\|_{L^2(\Gamma)}\le\widetilde C'\|x_h\|_{X_h}$. Hence
\[
  |(i\omega w_h,\jump{v_h})_\Gamma|
  \le |\omega|\,\|w_h\|_{L^2(\Gamma)}\|\jump{v_h}\|_{L^2(\Gamma)}
  \le |\omega|\widetilde C'\,\|x_h\|_{X_h}\|y_h\|_{X_h},
\]
\[
  |\mu_p(i\omega\rho_w\jump{\phi_h},r_h)_\Gamma|
  \le |\omega\rho_w\mu_p|\,\widetilde C'\,\|x_h\|_{X_h}\|y_h\|_{X_h}.
\]
Grouping the two coupling terms, with
$\bar C_3 := |\omega|\widetilde C'\max(1,|\mu_p\rho_w|)$,
\[
  |(i\omega w_h,\jump{v_h})_\Gamma - \mu_p(i\omega\rho_w\jump{\phi_h},r_h)_\Gamma|
  \le \bar C_3\,\|x_h\|_{X_h}\|y_h\|_{X_h}. 
\]

\medskip
\noindent\textbf{Step 5 (Structure terms).}
The structural bilinear form components on $\Gamma$ satisfy standard continuity via Cauchy-Schwarz bounds, that is
\begin{equation*}
    |\mu_p(\Drig:\bnabla^2\w_h,\bnabla^2r_h)_{\Gamma}| \leq |\mu_p|\, \|\Drig\|_{L^{\infty}(\Gamma)} \|\bnabla_h^2\w_h\|_{L^2(\Gamma)} \|\bnabla_h^2 r_h\|_{L^2(\Gamma)} \leq |\mu_p|\, \|\Drig\|_{L^{\infty}(\Gamma)} \|\xvec_h\|_{\X_h} \|\yvec_h\|_{\X_h},
\end{equation*}
\begin{equation*}
    |\mu_p(I_0 \omega^2\w_h,r_h)_\Gamma| \leq |\mu_p I_0 \omega^2|\, \|\w_h\|_{L^2(\Gamma)} \|r_h\|_{L^2(\Gamma)} \leq |\omega^2|\,|\mu_p I_0|\, \|\xvec_h\|_{\X_h} \|\yvec_h\|_{\X_h}.
\end{equation*}
Grouping the two terms together we have that:
\begin{equation*}
    |B_{\Gamma}(\phi_h, \w_h, v_h, r_h)| \leq  \left[C_3|\omega|(1+|\mu_p  \rho_w|) h^{-1/2} + |\mu_p|\, \|\Drig\|_{L^{\infty}(\Gamma)} + |\omega^2|\,|\mu_p I_0|\right]\|\xvec_h\|_{\X_h} \|\yvec_h\|_{\X_h}.
\end{equation*}

\medskip
\noindent\textbf{Step 6 (Structure skeleton terms, $B_\Sigma$).}
Finally, we analyse the structural internal skeleton terms supported on $\Sigma$. Let $p = \partial E_1 \cap \partial E_2 \in \Sigma$ denote an arbitrary interior face shared by the neighbouring structure elements $E_1, E_2 \in \mathcal{T}_h^\Gamma$. Decomposing the skeletal integral into local facet contributions and applying Cauchy-Schwarz inequality we have 
\begin{align*}
    \left|\mu_p(\mean{\Drig:\bnabla^2\w_h},\jump{\bnabla r_h\otimes\boldsymbol{n}_\Gamma})_{\Sigma}\right| \leq |\mu_p| \sum_{p \in \Sigma} \|\mean{\Drig:\bnabla^2\w_h}\|_{L^2(p)} \|\jump{\bnabla r_h\otimes\boldsymbol{n}_\Gamma}\|_{L^2(p)}.
\end{align*}
Since $\w_h$ and $r_h$ belong to the discrete polynomial space $\hat{\mathcal{V}}_{\Gamma,h}$, their derivatives inherit localized polynomial structures. Using \autoref{lemma:inverse_sigma} we have
\begin{equation*}
    \|\mean{\Drig:\bnabla^2\w_h}\|_{L^2(p)} \leq C_{\rm inv}^{1/2} h_p^{-1/2} \|\Drig\|_{L^\infty(\Gamma)} \|\bnabla^2\w_h\|_{L^2(E_1 \cup E_2)}.
\end{equation*}
Using the fact that each facet $E$ contributes to at most $N_f^\Sigma$ faces, we have that the $\sum_e\|\bnabla^2\w_h\|^2_{E_1\cup E_2}\le N_f^\Sigma\|\bnabla^2\w_h\|^2_{\Gamma}$. Therefore, summing over the skeleton $\Sigma$ and applying the discrete Cauchy-Schwarz inequality for sums yields 
\begin{equation*}
    \begin{aligned}
    \left|\mu_p\big(\mean{\Drig:\bnabla^2 w_h},\jump{\bnabla r_h\otimes\boldsymbol n_\Gamma}\big)_\Sigma\right| \le
\big(C_{\mathrm{inv}}^\Sigma N_f^\Sigma\big)^{1/2}\|\Drig\|_{L^\infty(\Gamma)}\,\|\bnabla^2
w_h\|_{\Gamma}\|\yvec_h\|_{\X_h}.
    \end{aligned}
\end{equation*}
Doing the same for the transposed term, adding both and using $\mu_p=\mu_p^{1/2}\cdot\mu_p^{1/2}$, with $C_6:=2(C_{\mathrm{inv}}^\Sigma N_f^\Sigma)^{1/2}\|\Drig\|_{L^\infty(\Gamma)}$, we have that
\[
\Big|\mu_p\big(\mean{\Drig:\bnabla^2 w_h},\jump{\bnabla r_h\otimes\boldsymbol n_\Gamma}\big)_\Sigma\Big| +
\Big|\mu_p\big(\mean{\Drig:\bnabla^2 r_h},\jump{\bnabla w_h\otimes\boldsymbol n_\Gamma}\big)_\Sigma\Big| \le
C_6|\mu_p|\,\|\mathbf x_h\|_{\mathcal X_h}\|\mathbf y_h\|_{\mathcal X_h}.
\]
Finally, the penalty term on the skeleton $\Sigma$ is bounded by
\[
\Big|\big(\mu_2\jump{\bnabla w_h\otimes\boldsymbol n_\Gamma},\jump{\bnabla r_h\otimes\boldsymbol
n_\Gamma}\big)_\Sigma\Big| \le \sigma_2\|\mathbf
x_h\|_{\mathcal X_h}\|\mathbf y_h\|_{\mathcal X_h}.
\]
Hence $|B_\Sigma(\w_h, r_h)|\le(C_6+\sigma_2)|\mu_p|\|\mathbf x_h\|_{\mathcal X_h}\|\mathbf y_h\|_{\mathcal X_h}$.

\medskip
\noindent\textbf{Step 7 (Assembly).}
Summing all independent contributions via the triangle inequality yields the global upper bound:
\begin{equation*}
    |B(\xvec_h, \yvec_h)| \leq \left( 1 + \bar C_1 + \tilde{C}_2  + \bar{C}_3 + \tilde{C}_4 + \tilde{C}_5 + \tilde{C}_6 \right) \|\xvec_h\|_{\X_h} \|\yvec_h\|_{\X_h}.
\end{equation*}
With $M = \max(1,\bar C_1, \tilde{C}_2, \bar{C}_3,\tilde{C}_4,\tilde{C}_5,\tilde{C}_6)$, with $\tilde{C}_2=(C_2+\sigma_1)$, $\tilde{C}_4=|\mu_p|\, \|\Drig\|_{L^{\infty}(\Gamma)}$, $\tilde{C}_5=|\omega^2|\,|\mu_p I_0|$ and $\tilde{C}_6=(C_6+\sigma_2)|\mu_p|$, results into the statement: there exists a mesh-independent constant $M = \mathcal{O}(1)$ such that 
\begin{equation*}
    |B(\xvec_h, \yvec_h)| \leq M \|\xvec_h\|_{\X_h} \|\yvec_h\|_{\X_h},
\end{equation*}
which completes the proof.
\end{proof}

\subsection{Discrete G\aa rding inequality}\label{sec:garding}\noindent

\begin{theorem}[Discrete G\aa rding inequality]\label{thm:garding}
    Let $\X_h$ be the discrete space equipped with the norm $\|\cdot\|_{\X_h}$ defined in \eqref{eq:energy_norm}. Then, the coupled sesquilinear form $B(\cdot, \cdot)$ defined in \eqref{eq:bilinear_form_analysis} satisfies a discrete Gårding-type inequality for $\mu_p=1/\rho_w$. That is, there exist constants $C_1 > 0$ and $C_2 > 0$ such that for all $\xvec_h = [\phi_h, w_h] \in \X_h$:
\begin{equation}
    \begin{aligned}
\mathrm{Re}\,B(\mathbf x_h,\mathbf x_h) \ge&\ C_1\left(\|\bnabla\phi_h\|_\Omega^2 + \sum_{e\in\Lambda\setminus\Gamma} h_e^{-1}\big\|\jump{\phi_h}\big\|_{L^2(e)}^2 +
\mu_p\|\bnabla^2w_h\|^2_{\Gamma} + \mu_p \sum_{p\in\Sigma}h_p^{-1}
    \big\|\jump{\bnabla w_h\otimes\boldsymbol n_{\mathcal P}}\big\|_{L^2(p)}^2\right)\\
&\ - C_2\left(\|\phi_h\|_\Omega^2 +
\mu_p\|w_h\|^2_{\Gamma}\right).
\label{eq:garding}
    \end{aligned}
\end{equation}
provided the recovery-operator constant $C_R$ of \autoref{lem:fluid_recovery} satisfies the mesh-shape-regularity condition stated in Remark~\ref{rem:CR_smallness}.

In particular, 
\[
  C_1:=\min\Big(1,\ \sigma_1-\sigma_1^\dagger,\ \tfrac12\gamma_{\min},\ \sigma_2-\sigma_2^*,\ \tfrac14-\kappa_{\rm fs}\Big)>0,
  \qquad
  C_2:=\max\big(C_{2,{\rm fs}},\ |I_0\omega^2|+\epsilon^{-1}|\omega|\rho_w\big),
\]
for $\sigma_1>\sigma_1^{\dagger}$, $\sigma_2>\sigma_2^\ast$, with fixed
$\delta:=1/\big(32\alpha(1+C_R)^2\big)$,
$\epsilon := 1/(8|\omega|\kappa_1)$ and $h\le1$, where
\[
\begin{aligned}
&\sigma_1^{\dagger}:=\sigma_1^\ast+\widetilde C_J, \qquad
\sigma_1^\ast:=\frac{2C_{\mathrm{inv}}N_f}{c_0}, \qquad
\sigma_2^\ast:=\frac{2C_{\mathrm{inv}}^\Sigma N_f^\Sigma
\|\Drig\|^2_{L^\infty(\Gamma)}}{c_0\gamma_{\min}}, \\[2pt]
&\kappa_1 := 16(C_{\rm tr}^\Gamma)^2\big[2(1+C_R)^2+4C_R^2\big]+8C_{\rm tr}C_R^2,
\qquad
\kappa_{\rm fs} := 8\alpha C_\delta C_R^2, \\[2pt]
&C_{2,\mathrm{fs}}:=4\alpha C_\delta+32\epsilon|\omega|(C_{\rm tr}^\Gamma)^2, \qquad
\widetilde C_J := C_{J,{\rm fs}} + \epsilon|\omega|C_J',
\end{aligned}
\]
with $C_{J,{\rm fs}}$, $C_J'$ as defined in Step~3 below.
\end{theorem}

\begin{proof}Let us evaluate the sesquilinear form by choosing the test functions to be equal to the trial functions, setting $\yvec_h = \xvec_h$. Substituting these into the global expression for $B(\xvec_h, \xvec_h)$ yields 
\begin{equation}
\begin{aligned}
B(\mathbf x_h,\mathbf x_h) = &\ \|\bnabla\phi_h\|^2_\Omega - (\alpha\phi_h,\phi_h)_{\Gamma_{\rm fs}} - K(\phi_h,\phi_h)
- 2\big(\mean{\partial_n\phi_h},\jump{\phi_h}\big)_{\Lambda\setminus\Gamma} + \big(\mu_1\jump{\phi_h},\jump{\phi_h}
\big)_{\Lambda\setminus\Gamma} \\
&+ (\upi\omega w_h,\jump{\phi_h})_\Gamma + \mu_p\Big[(\Drig:\bnabla^2w_h,\bnabla^2w_h)_\Gamma -
(I_0\omega^2w_h,w_h)_\Gamma - (\upi\omega\rho_w\jump{\phi_h},w_h)_\Gamma \\
&- 2\big(\mean{\Drig:\bnabla^2w_h},\jump{\bnabla w_h\otimes\boldsymbol n_\Gamma}\big)_\Sigma + \big(\mu_2\jump{\bnabla
w_h\otimes\boldsymbol n_\Gamma},\jump{\bnabla w_h\otimes\boldsymbol n_\Gamma}\big)_\Sigma\Big].
\end{aligned}
\label{eq:step1}
\end{equation}
We proceed by bounding $\mathrm{Re}\,B(\mathbf x_h,\mathbf x_h)$ from below in four passes over \eqref{eq:step1}: the radiation
term, the fluid-skeleton block, the Robin/coupling boundary block, and the structure block.

\medskip
\noindent\textbf{Step 1 (Radiation term).}
Writing out $K(\phi_h,\phi_h)=\sum_{m=0}^M\frac{\upi k_m}{C_m}\big(|(\phi_h,\psi_m)_{\Gamma_{\rm in}}|^2+
|(\phi_h,\psi_m)_{\Gamma_{\rm out}}|^2\big)$, every bracketed quantity is real and nonnegative and every $C_m>0$ is
real, so with $k_m=\mathrm{Re}(k_m)+\upi\,\mathrm{Im}(k_m)$,
\[
\mathrm{Re}(\upi k_m) = -\mathrm{Im}(k_m) \implies \mathrm{Re}\,K(\phi_h,\phi_h) = -\sum_{m=0}^M
\frac{\mathrm{Im}(k_m)}{C_m}\Big(|(\phi_h,\psi_m)_{\Gamma_{\rm in}}|^2+|(\phi_h,\psi_m)_{\Gamma_{\rm out}}|^2\Big)
\le 0,
\]
since $\mathrm{Im}(k_m)\ge0$ for every $m$ by the outgoing/decaying branch convention. Hence
$\mathrm{Re}(-K(\phi_h,\phi_h))\ge0$, and this term contributes a nonnegative amount to $\mathrm{Re}\,B(\mathbf
x_h,\mathbf x_h)$.

\medskip
\noindent\textbf{Step 2 (Fluid terms).}
Given a skeleton facet $e\in\Lambda\setminus\Gamma$ shared by two  $K_1,K_2$, using Young's inequality, we have that for any $\delta_1>0$,
\[
2\big|\big(\mean{\partial_n\phi_h},\jump{\phi_h}\big)_e\big| \le 2\,h_e^{1/2}\|\mean{\partial_n\phi_h}\|_{L^2(e)}
h_e^{-1/2}\|\jump{\phi_h}\|_{L^2(e)} \le \delta_1 h_e\|\mean{\partial_n\phi_h}\|^2_{L^2(e)} +
\delta_1^{-1}h_e^{-1}\left\|\jump{\phi_h}\right\|^2_{L^2(e)}.
\]
Since $\mean{\partial_n\phi_h}=\tfrac12(\partial_n\phi_h|_{K_1}+\partial_n\phi_h|_{K_2})$, the inverse inequality
for fluxes (\autoref{theorem:inverse_inequality}) applied on each side and shape-regularity ($h_e\sim h_{K_1}\sim h_{K_2}$) give
$h_e\|\mean{\partial_n\phi_h}\|^2_{L^2(e)}\le C_{\mathrm{inv}}\|\bnabla\phi_h\|^2_{L^2(K_1\cup K_2)}$. Summing over
$e\in\Lambda\setminus\Gamma$ and using that each element $K$ is shared by at most $N_f$ facets,
\[
\sum_{e\in\Lambda\setminus\Gamma}h_e\|\mean{\partial_n\phi_h}\|^2_{L^2(e)} \le C_{\mathrm{inv}}N_f
\|\bnabla\phi_h\|^2_\Omega,
\]
so that, summing the Young bound over all facets,
\[
2\big|\big(\mean{\partial_n\phi_h},\jump{\phi_h}\big)_{\Lambda\setminus\Gamma}\big| \le \delta_1C_{\mathrm{inv}}N_f
\|\bnabla\phi_h\|^2_\Omega + \delta_1^{-1}\sum_{e\in\Lambda\setminus\Gamma} h_e^{-1}\big\|\jump{\phi_h}\big\|_{L^2(e)}^2.
\]
Recalling that $\mu_1|_e=\sigma_1/h_e$, so that $(\mu_1\jump{\phi_h},\jump{\phi_h})_{\Lambda\setminus\Gamma}=
\sigma_1\sum_{e\in\Lambda\setminus\Gamma} h_e^{-1}\big\|\jump{\phi_h}\big\|_{L^2(e)}^2$, and combining with the first term $\|\bnabla\phi_h\|^2_\Omega$, we have that
\begin{align*}
\mathrm{Re}\Big[\|\bnabla\phi_h\|^2_\Omega - 2\big(\mean{\partial_n\phi_h},\jump{\phi_h}\big)_{\Lambda\setminus\Gamma} +
\big(\mu_1\jump{\phi_h},\jump{\phi_h}\big)_{\Lambda\setminus\Gamma}\Big] \ge&\ (1-\delta_1C_{\mathrm{inv}}N_f)
\|\bnabla\phi_h\|^2_\Omega \\
&\ + (\sigma_1-\delta_1^{-1})\sum_{e\in\Lambda\setminus\Gamma} h_e^{-1}\big\|\jump{\phi_h}\big\|_{L^2(e)}^2.
\end{align*}
Taking $\delta_1=1/(2C_{\mathrm{inv}}N_f)$ and defining  $\sigma_1^\ast=2C_{\mathrm{inv}}N_f/c_0\ge2C_{\mathrm{inv}}N_f$ we have that $\sigma_1>\sigma_1^\ast\ge\delta_1^{-1}$. Therefore, 
\[
\mathrm{Re}\Big[
    \|\bnabla\phi_h\|^2_\Omega - 2\big(\mean{\partial_n\phi_h},\jump{\phi_h}\big)_{\Lambda\setminus\Gamma} +
\big(\mu_1\jump{\phi_h},\jump{\phi_h}\big)_{\Lambda\setminus\Gamma}
\Big]
\ge \tfrac12\|\bnabla\phi_h\|^2_\Omega + (\sigma_1-\sigma_1^\ast)\sum_{e\in\Lambda\setminus\Gamma} h_e^{-1}\big\|\jump{\phi_h}\big\|_{L^2(e)}^2,
\]
with $\sigma_1 - \sigma_1^\ast > 0$ by definition.

\noindent\textbf{Step 3 (Free surface and fluid--structure interaction terms).}
Let us write $\phi_h = R_h\phi_h + (\phi_h-R_h\phi_h)$ and bound the two terms on $\Gamma_{\rm fs}\subset\partial\Omega$ separately. That is, since $\phi_h-R_h\phi_h$ is a polynomial on each $K$, the local trace theorem (\autoref{theorem:trace}), summed over $K\cap\Gamma_{\rm fs}\ne\emptyset$, together with
\autoref{lem:fluid_recovery}, gives (for the recovery constant $C_R>0$ from that lemma)
\[
\begin{aligned}
\|\phi_h-R_h\phi_h\|^2_{L^2(\Gamma_{\rm fs})}
&\le C_{\mathrm{tr}}\Big(h^{-1}\|\phi_h-R_h\phi_h\|^2_{L^2(\Omega)}+h\|\bnabla(\phi_h-R_h\phi_h)\|^2_{L^2(\Omega)}\Big)\\
&\le 2C_{\mathrm{tr}}C_R^2\,h\,\big(\|\bnabla\phi_h\|_{L^2(\Omega)}+J_f(\phi_h)\big)^2
\le 4C_{\mathrm{tr}}C_R^2\,h\,\big(\|\bnabla\phi_h\|^2_{L^2(\Omega)}+J_f(\phi_h)^2\big).
\end{aligned}
\]

Similarly, since $R_h\phi_h\in H^1(\Omega\setminus\Gamma)$ is conforming and $\Gamma_{\rm fs}\subset\partial\Omega$ is disjoint from the interior interface $\Gamma$, the continuous trace inequality applies with an $h$-independent constant $C_\delta$: for any fixed $\delta>0$,
\[
\|R_h\phi_h\|^2_{L^2(\Gamma_{\rm fs})} \le \delta\,\|\bnabla R_h\phi_h\|^2_{L^2(\Omega)} + C_\delta\,\|R_h\phi_h\|^2_{L^2(\Omega)}.
\]
By the triangle inequality and \autoref{lem:fluid_recovery}, we have
\begin{equation*}
\begin{aligned}
\|\bnabla R_h\phi_h\|_{L^2(\Omega)} \le&\ (1+C_R)\|\bnabla\phi_h\|_{L^2(\Omega)} + C_RJ_f(\phi_h), \\
\|R_h\phi_h\|_{L^2(\Omega)} \le&\ \|\phi_h\|_{L^2(\Omega)} + C_R\|\bnabla\phi_h\|_{L^2(\Omega)} + C_RJ_f(\phi_h).
\end{aligned}
\end{equation*}
Squaring via $(a+b)^2\le2a^2+2b^2$ as in Step-2 of \autoref{theorem:discrete_continuity} and grouping terms we have that
\[
\|R_h\phi_h\|^2_{L^2(\Gamma_{\rm fs})}
\le \big[2\delta(1+C_R)^2+4C_\delta C_R^2\big]\|\bnabla\phi_h\|^2_{L^2(\Omega)}
+ 2C_\delta\|\phi_h\|^2_{L^2(\Omega)}
+ 2C_R^2(\delta+2C_\delta)\,J_f(\phi_h)^2.
\]
Combining the two bounds via $(a+b)^2\le2a^2+2b^2$,
\[
\alpha\|\phi_h\|^2_{L^2(\Gamma_{\rm fs})} \le \big[4\alpha\delta(1+C_R)^2+8\alpha C_\delta C_R^2\big]\|\bnabla\phi_h\|^2_{L^2(\Omega)}
+ 4\alpha C_\delta\,\|\phi_h\|^2_{L^2(\Omega)} + C_{J,\mathrm{fs}}(h)\,J_f(\phi_h)^2,
\]
where
\[
C_{J,\mathrm{fs}}(h):=4\alpha C_R^2\big(C_{\mathrm{tr}}h+\delta+2C_\delta\big).
\]
Note that fixing $\delta:=1/\big(32\alpha(1+C_R)^2\big)$, i.e. a constant depending on $\alpha,\Omega$, and the
mesh-shape-regularity constant $C_R$ only with $\alpha>0$. Thus, $4\alpha\delta(1+C_R)^2=\tfrac18$ and the part $8\alpha C_\delta C_R^2$ is then a fixed, $h$-independent number that does not vanish as $\delta$ is refined further. Here, we denote it
\begin{equation}
\label{eq:kappa_fs}
\kappa_{\rm fs} := 8\alpha C_\delta C_R^2.
\end{equation}
Then,
\[
\mathrm{Re}\Big[-(\alpha\phi_h,\phi_h)_{\Gamma_{\rm fs}}\Big] \ge -\Big(\tfrac18+\kappa_{\rm fs}\Big)\|\bnabla\phi_h\|^2_{L^2(\Omega)} -4\alpha C_\delta\,\|\phi_h\|^2_{L^2(\Omega)} - C_{J,\mathrm{fs}}(h)\,J_f(\phi_h)^2.
\]

For the choice $\mu_p=1/\rho_w$, the two coupling contributions appearing in \eqref{eq:step1} are
\[
(\upi\omega w_h,\jump{\phi_h})_\Gamma
-
(\upi\omega\jump{\phi_h},w_h)_\Gamma
=
-2\omega\,\operatorname{Im}(z_h).
\]
Since $\operatorname{Im}(z_h)\in\mathbb R$, the right-hand side is purely real. Consequently,
\[
\operatorname{Re}
\Big[
(\upi\omega w_h,\jump{\phi_h})_\Gamma
-
(\upi\omega\jump{\phi_h},w_h)_\Gamma
\Big]
=
-2\omega\,\operatorname{Im}(w_h,\jump{\phi_h})_\Gamma.
\]

Let us now apply the identical splitting to $\jump{\phi_h}=\jump{R_h\phi_h}+\jump{\phi_h-R_h\phi_h}$ on $\Gamma$. By the continuous
trace theorem applied separately on the two sides $\Omega^+,\Omega^-$ of
$\Gamma$ (constants $C_{\rm tr}^{\Omega^\pm}$, $h$-independent) and the
triangle inequality,
\[
  \|\jump{R_h\phi_h}\|_{L^2(\Gamma)} \le
  2C_{\rm tr}^\Gamma\|R_h\phi_h\|_{H^1(\Omega)},
  \qquad C_{\rm tr}^\Gamma:=\max(C_{\rm tr}^{\Omega^+},C_{\rm tr}^{\Omega^-}).
\]
By the \autoref{lem:fluid_recovery} and the same squaring argument as in Step-4 of \autoref{theorem:discrete_continuity}, 
\[
\|R_h\phi_h\|^2_{H^1(\Omega)} \le 2\|\phi_h\|^2_\Omega+\big[2(1+C_R)^2+4C_R^2\big]\|\bnabla\phi_h\|^2_\Omega+6C_R^2J_f(\phi_h)^2,
\]
therefore, by the continuous trace theorem on $\Omega^\pm$,
\[
\|\jump{R_h\phi_h}\|^2_{L^2(\Gamma)}
\le 8(C_{\rm tr}^\Gamma)^2\Big(\big[2(1+C_R)^2+4C_R^2\big]\|\bnabla\phi_h\|^2_\Omega+2\|\phi_h\|^2_\Omega+6C_R^2J_f(\phi_h)^2\Big).
\]
Together with the remainder bound
$\|\jump{\phi_h-R_h\phi_h}\|^2_{L^2(\Gamma)}\le4C_{\rm tr}C_R^2\,h\,(\|\bnabla\phi_h\|^2_\Omega+J_f(\phi_h)^2)$
and $\jump{\phi_h}=\jump{R_h\phi_h}+\jump{\phi_h-R_h\phi_h}$, we obtain
\[
  \|\jump{\phi_h}\|^2_{L^2(\Gamma)} \le
  \kappa_1\,\|\bnabla\phi_h\|^2_\Omega
  + 32(C_{\rm tr}^\Gamma)^2\|\phi_h\|^2_\Omega
  + C_J'\,J_f(\phi_h)^2,
\]
with
\[
\kappa_1 := 16(C_{\rm tr}^\Gamma)^2\big[2(1+C_R)^2+4C_R^2\big]+8C_{\rm tr}C_R^2,
\qquad
C_J' := 96(C_{\rm tr}^\Gamma)^2C_R^2 + 8C_{\rm tr}C_R^2.
\]

By Young's inequality with parameter $\epsilon>0$,
\[  2|\omega|\|w_h\|_{L^2(\Gamma)}\|\jump{\phi_h}\|_{L^2(\Gamma)}
  \le \epsilon^{-1}|\omega|\|w_h\|^2_{L^2(\Gamma)} + \epsilon|\omega|\|\jump{\phi_h}\|^2_{L^2(\Gamma)}.
\]
Now using $\rho_w\mu_p=1$ to rewrite
$\epsilon^{-1}|\omega|\|w_h\|^2_{L^2(\Gamma)}=\epsilon^{-1}|\omega|\rho_w\mu_p\|w_h\|^2_{\Gamma}$, we obtain
\[
\begin{aligned}
2|\omega|\|w_h\|_{L^2(\Gamma)}\|\jump{\phi_h}\|_{L^2(\Gamma)}
  \le&\ \epsilon|\omega|\kappa_1\|\bnabla\phi_h\|^2_\Omega
  + 32\epsilon|\omega|(C_{\rm tr}^\Gamma)^2\|\phi_h\|^2_\Omega
  + \epsilon^{-1}|\omega|\rho_w\mu_p\|w_h\|^2_{\Gamma}
  + \epsilon|\omega|C_J'J_f(\phi_h)^2.
  \end{aligned}
\]

Let us fix $\epsilon:=1/\big(8|\omega|\kappa_1\big)$, with $\omega>0$, so that $\epsilon|\omega|\kappa_1=\tfrac18$.
Combined with Step~3, the total gradient-term coefficient consumed so far is $\Big(\tfrac18+\kappa_{\rm fs}\Big) + \tfrac18 \;=\; \tfrac14+\kappa_{\rm fs}$. Therefore, adding the contributions,
\[
\begin{aligned}
  \mathrm{Re}\Big[-(\alpha\phi_h,\phi_h)_{\Gamma_{\rm fs}}
  + (i\omega w_h,\phi_h)_\Gamma - \mu_p(i\omega\rho_w\phi_h,w_h)_\Gamma\Big]
  \ge&\ -\Big(\tfrac14+\kappa_{\rm fs}\Big)\|\bnabla\phi_h\|^2_\Omega - C_{2,{\rm fs}}\|\phi_h\|^2_\Omega\\
  &\ - \epsilon^{-1}|\omega|\rho_w\mu_p\|w_h\|^2_{\Gamma}
  - \widetilde C_J\,J_f(\phi_h)^2, 
  \end{aligned}
\]
where $C_{2,{\rm fs}}:=4\alpha C_\delta+32\epsilon|\omega|(C_{\rm tr}^\Gamma)^2$
and $\widetilde C_J:=C_{J,{\rm fs}}+\epsilon|\omega|C_J'$, both
$h$-independent.

\begin{remark}\label{rem:CR_smallness}
Note that $\kappa_{\rm fs} \;<\; \tfrac14$ in order for the gradient coefficient $\tfrac14+\kappa_{\rm fs}$ to remain strictly below $\tfrac12$. This requires a small-enough $C_R$ mesh-shape-regularity constant.
\end{remark}

\noindent\textbf{Step 4 (Structure terms).} Through the use of \autoref{ass:rigidity}, we note that
\[
(\Drig:\bnabla^2w_h,\bnabla^2w_h)_\Gamma \ge \gamma_{\min}
\|\bnabla^2w_h\|^2_{\Gamma},
\]
which is a real, nonnegative quantity. Following the same procedure as in Step~2, for $p\in\Sigma$ shared by $E_1,E_2$, Young's inequality gives,
for any $\delta_2>0$,
\[
2\big|\big(\mean{\Drig:\bnabla^2w_h},\jump{\bnabla w_h\otimes\boldsymbol n_\Gamma}\big)_p\big| \le \delta_2h_p
\|\mean{\Drig:\bnabla^2w_h}\|^2_{L^2(p)} + \delta_2^{-1}h_p^{-1}\|\jump{\bnabla w_h\otimes\boldsymbol n_\Gamma}\|
^2_{L^2(p)},
\]
and, by the inverse Hessian-trace inequality of \autoref{lemma:inverse_sigma} applied on each side and shape-regularity,
$h_p\|\mean{\Drig:\bnabla^2w_h}\|^2_{L^2(p)}\le C_{\mathrm{inv}}^\Sigma\|\Drig\|^2_{L^\infty(\Gamma)}\|\bnabla^2w_h\|
^2_{L^2(E_1\cup E_2)}$. Summing over $p\in\Sigma$, with each element counted at most $N_f^\Sigma$ times,
\[
2\big|\big(\mean{\Drig:\bnabla^2w_h},\jump{\bnabla w_h\otimes\boldsymbol n_\Gamma}\big)_\Sigma\big| \le \delta_2
C_{\mathrm{inv}}^\Sigma N_f^\Sigma\|\Drig\|^2_{L^\infty(\Gamma)}\|\bnabla^2w_h\|^2_{\Gamma} +
\delta_2^{-1}\sum_{p\in\Sigma}h_p^{-1}
    \big\|\jump{\bnabla w_h\otimes\boldsymbol n_{\mathcal P}}\big\|_{L^2(p)}^2.
\]
Recalling that $\mu_2|_p=\sigma_2/h_p$, combining all the rigidity-related terms we obtain
\begin{align*}
(\Drig:\bnabla^2w_h,\bnabla^2w_h)_\Gamma -&\ 2\big(\mean{\Drig:\bnabla^2w_h},\jump{\bnabla w_h\otimes\boldsymbol n_\Gamma}
\big)_\Sigma + \big(\mu_2\jump{\bnabla w_h\otimes\boldsymbol n_\Gamma},\jump{\bnabla w_h\otimes\boldsymbol n_\Gamma}
\big)_\Sigma \\
\ge&\ \left(\gamma_{\min}-\delta_2C_{\mathrm{inv}}^\Sigma N_f^\Sigma\|\Drig\|^2_{L^\infty(\Gamma)}\right)
\|\bnabla^2w_h\|^2_{\Gamma} + \left(\sigma_2-\delta_2^{-1}\right)\sum_{p\in\Sigma}h_p^{-1}
    \big\|\jump{\bnabla w_h\otimes\boldsymbol n_{\mathcal P}}\big\|_{L^2(p)}^2.
\end{align*}
Taking $\delta_2=\gamma_{\min}/(2C_{\mathrm{inv}}^\Sigma N_f^\Sigma\|\Drig\|^2_{L^\infty(\Gamma)})$ and defining $\sigma_2^\ast=2C_{\mathrm{inv}}^\Sigma N_f^\Sigma\|\Drig\|^2_{L^\infty(\Gamma)}/(c_0\gamma_{\min})$, we note that 
$\sigma_2>\sigma_2^\ast\ge\delta_2^{-1}$. Therefore,
\begin{align*}
(\Drig:\bnabla^2w_h,\bnabla^2w_h)_\Gamma -&\ 2\big(\mean{\Drig:\bnabla^2w_h},\jump{\bnabla w_h\otimes\boldsymbol n_\Gamma}
\big)_\Sigma + \big(\mu_2\jump{\bnabla w_h\otimes\boldsymbol n_\Gamma},\jump{\bnabla w_h\otimes\boldsymbol n_\Gamma}
\big)_\Sigma \\
\ge&\ \tfrac12\gamma_{\min}\|\bnabla^2w_h\|^2_{\Gamma} + (\sigma_2-\sigma_2^\ast)\sum_{p\in\Sigma}h_p^{-1}
    \big\|\jump{\bnabla w_h\otimes\boldsymbol n_{\mathcal P}}\big\|_{L^2(p)}^2.
\end{align*}
Finally, multiplying by $\mu_p>0$ and adding the mass term bound
$-\mu_p(I_0\omega^2w_h,w_h)_\Gamma=-\mu_p I_0\omega^2\|w_h\|^2_{\Gamma}\ge-|I_0\omega^2|\,\mu_p\|w_h\|^2_{\Gamma}$,
\begin{align*}
\mathrm{Re}&\left(\mu_p\left[(\Drig:\bnabla^2w_h,\bnabla^2w_h)_\Gamma - (I_0\omega^2w_h,w_h)_\Gamma - 2\big(\mean{\Drig:\bnabla^2w_h},\jump{\bnabla w_h\otimes\boldsymbol n_\Gamma}\big)_\Sigma\right.\right. \\
+&\ \left.\left.\big(\mu_2\jump{\bnabla w_h\otimes\boldsymbol
n_\Gamma},\jump{\bnabla w_h\otimes\boldsymbol n_\Gamma}\big)_\Sigma\right]\right) \\
\ge&\ \mu_p\Big(\tfrac12\gamma_{\min}
\|\bnabla^2w_h\|^2_{\Gamma} + (\sigma_2-\sigma_2^\ast)\sum_{p\in\Sigma}h_p^{-1}
|    \big\|\jump{\bnabla w_h\otimes\boldsymbol n_{\mathcal P}}\big\|_{L^2(p)}^2\Big) - |I_0\omega^2|\,\mu_p\|w_h\|^2_{\Gamma}.
\end{align*}

\noindent\textbf{Step 5 (Assembly).}
The final step is to bring together all the bounds from Steps~1--4. Combining those estimates gives
\[
\begin{aligned}
\mathrm{Re}\,B(\mathbf x_h,\mathbf x_h) \ge\ &\ \Big(\tfrac14-\kappa_{\rm fs}\Big)\|\bnabla\phi_h\|^2_\Omega +
\big(\sigma_1-\sigma_1^\ast-\widetilde C_J\big)\sum_{e\in\Lambda\setminus\Gamma} h_e^{-1}\big\|\jump{\phi_h}\big\|_{L^2(e)}^2 \\
&\ + \mu_p\Big(\tfrac12\gamma_{\min}\|\bnabla^2w_h\|^2_{\Gamma}+(\sigma_2-\sigma_2^\ast)\sum_{p\in\Sigma}h_p^{-1}
    \big\|\jump{\bnabla w_h\otimes\boldsymbol n_{\mathcal P}}\big\|_{L^2(p)}^2\Big) \\
&\ - C_{2,\mathrm{fs}}\|\phi_h\|^2_\Omega -\big(|I_0\omega^2|+|\omega|\rho_w\big)\mu_p\|w_h\|^2_{\Gamma},
\end{aligned}
\]
where $\tfrac14-\kappa_{\rm fs}>0$ by Remark~\ref{rem:CR_smallness}.

Defining $\sigma_1^{\dagger}:=\sigma_1^\ast+\widetilde C_J$, 
for $\sigma_1>\sigma_1^{\dagger}$ and  $\sigma_2>\sigma_2^*$,
\[
\begin{aligned}
  \mathrm{Re}\,B(x_h,x_h) \ge& C_1\Big(\|\bnabla\phi_h\|^2_\Omega
  +\sum_eh_e^{-1}\|\jump{\phi_h}\|^2_{L^2(e)}
  +\mu_p\|\bnabla^2w_h\|^2_{\Gamma}
  +\mu_p\sum_ph_p^{-1}\|\jump{w_h\otimes n_P}\|^2_{L^2(p)}\Big)\\
  &-C_2\big(\|\phi_h\|^2_\Omega+\mu_p\|w_h\|^2_{\Gamma}\big),
  \end{aligned}
\]
with
\[
  C_1:=\min\Big(1,\ \sigma_1-\sigma_1^\dagger,\ \tfrac12\gamma_{\min},\ \sigma_2-\sigma_2^*,\ \tfrac14-\kappa_{\rm fs}\Big)>0,
  \qquad
  C_2:=\max\big(C_{2,{\rm fs}},\ |I_0\omega^2|+\epsilon^{-1}|\omega|\rho_w\big).
\]
This is exactly \eqref{eq:garding}, completing the proof.
\end{proof}

\subsection{Extension to the exact solution}\label{sec:extension}\noindent
\autoref{theorem:discrete_continuity} bounds $B(\cdot,\cdot)$ only on $\X_h\times\X_h$: two of its five blocks,
$B_\Lambda$ and $B_\Sigma$, control the mean-flux and mean-moment terms via \autoref{theorem:inverse_inequality}
and \autoref{lemma:inverse_sigma} respectively, both of which require their argument to be a piecewise
\emph{polynomial}. Since the exact solution $\xvec$ is not, in general, a polynomial, these two bounds do not
extend directly to pairs $(\xvec,\yvec_h)$ or $(\xvec-\mathbf z_h,\yvec_h)$ with $\mathbf z_h\in\X_h$, which is a requirement in further theorems, e.g. \autoref{thm:wellposed}. Note also that the recovery operators $R_h$
and $S_h$ used in \autoref{theorem:discrete_continuity} extend directly to the broken-regularity exact solution, since their construction only requires element-wise
$H^1$/$H^2$ regularity, not polynomiality. In what follows we extend the continuity
bound to the exact solution $\mathbf x\in\mathcal X$. 

For $\mathbf u=[\varphi,\chi]$ with $\varphi|_K\in H^2(K)$, $\chi|_E\in H^3(E)$ for every $K\in\mathcal T_h$,
$E\in\mathcal T_h^\Gamma$, define the regularity seminorm
\begin{equation}
\label{eq:reg_seminorm}
\mathcal R(\mathbf u)^2 := \sum_{K\in\mathcal T_h}\|\bnabla^2\varphi\|^2_{L^2(K)} + \mu_p\sum_{E\in\mathcal
T_h^\Gamma}\|\bnabla^3\chi\|^2_{L^2(E)},
\end{equation}
and the augmented (broken) norm on $\mathcal X(h):=\X_h+\{\mathbf u: \mathcal R(\mathbf u)<\infty\}$,
\begin{equation}
\label{eq:augmented_norm}
\vertiii{\mathbf u}_h^2 := \|\mathbf u\|_{\X_h}^2 + h\,\mathcal R(\mathbf u)^2.
\end{equation}
With this definition, $\vertiii{\mathbf u}_h \ge\|\mathbf u\|_{\X_h}$, with equality up to an $\mathcal O(h^{1/2})$ correction whenever $\mathcal R(\mathbf u)<\infty$.

\begin{lemma}[Extended continuity]\label{lem:cont_ext}
Let $\mathbf u=[\varphi,\chi]\in\mathcal X(h)$. Then, for every $\yvec_h=[v_h,r_h]\in\X_h$,
\begin{equation}
\label{eq:cont_ext}
|B(\mathbf u,\yvec_h)| \le \widetilde M\,\vertiii{\mathbf u}_h\,\|\yvec_h\|_{\X_h},
\end{equation}
with a constant $\widetilde M=\mathcal O(1)$ independent of mesh size.
\end{lemma}

\begin{proof}
The bounds for $B_{\Omega}$,
$B_{\partial\Omega}$, and the non-mean part of $B_\Gamma$ stated in the proof of \autoref{theorem:discrete_continuity} apply directly for $\varphi,\chi$ in place of $\phi_h,w_h$, with the same constants $1,\bar C_1,\bar C_3,\tilde C_4,\tilde C_5$. It only remains to revisit the mean-flux term in $B_\Lambda$ and the mean-moment term in $B_\Sigma$, where \autoref{theorem:inverse_inequality} and \autoref{lemma:inverse_sigma} were used and require polynomial structure, not necessarily present in $\varphi,\chi$.

\smallskip\noindent\textbf{Mean-flux term.} For $e=\partial K_1\cap\partial K_2\in\Lambda\setminus\Gamma$,
since $\varphi|_{K_i}\in H^2(K_i)$, the function $\partial_n\varphi|_{K_i}$ belongs to $H^1(K_i)$, so the trace
\autoref{theorem:trace} applies directly, giving:
\[
\|\partial_n\varphi\|^2_{L^2(e)} \le C_{\mathrm{tr}}\big(h_e^{-1}\|\bnabla\varphi\|^2_{L^2(K_i)} +
h_e\|\bnabla^2\varphi\|^2_{L^2(K_i)}\big),
\]
using $\bnabla(\partial_n\varphi)\le\|\bnabla^2\varphi\|$ pointwise. Multiplying by $h_e$, using
$h_e\le h\le1$ so $h_e^2\le h$, and averaging over $K_1,K_2$ we have
\[
h_e\|\mean{\partial_n\varphi}\|^2_{L^2(e)} \le \tfrac12C_{\mathrm{tr}}\big(\|\bnabla\varphi\|^2_{L^2(K_1\cup K_2)}
+ h\,\|\bnabla^2\varphi\|^2_{L^2(K_1\cup K_2)}\big).
\]
Summing over $e\in\Lambda\setminus\Gamma$ with each element counted at most $N_f$ times, we obtain
\[
\sum_{e}h_e\|\mean{\partial_n\varphi}\|^2_{L^2(e)} \le \tfrac12 C_{\mathrm{tr}}N_f\big(\|\bnabla\varphi\|^2_\Omega +
h\,\mathcal R_\Omega(\varphi)^2\big), \qquad \mathcal R_\Omega(\varphi)^2:=\sum_K\|\bnabla^2\varphi\|^2_{L^2(K)}.
\]
Applying Cauchy--Schwarz to the term $(\mean{\partial_n\varphi},\jump{v_h})_{\Lambda\setminus\Gamma}$ as in the original proof then gives
\[
\begin{aligned}
|(\mean{\partial_n\varphi},\jump{v_h})_{\Lambda\setminus\Gamma}| \le&\ \big(\tfrac12C_{\mathrm{tr}}N_f\big)^{1/2}
\big(\|\bnabla\varphi\|^2_\Omega+h\,\mathcal R_\Omega(\varphi)^2\big)^{1/2}\Big(\sum_e
h_e^{-1}\|\jump{v_h}\|^2_{L^2(e)}\Big)^{1/2} \\
\le&\ C\big(\|\bnabla\varphi\|_\Omega+h^{1/2}\mathcal
R_\Omega(\varphi)\big)\|\yvec_h\|_{\X_h}.
\end{aligned}
\]
Since $h^{1/2}\mathcal R_\Omega(\varphi) \le h^{1/2}\mathcal R(\mathbf u) \le \vertiii{\mathbf u}_h$ and $\|\bnabla\varphi\|_\Omega \le \|\mathbf u\|_{\X_h} \le \vertiii{\mathbf u}_h$, it follows that
\[
|(\mean{\partial_n\varphi},\jump{v_h})_{\Lambda\setminus\Gamma}|
\le
C\,\vertiii{\mathbf u}_h\,\|\yvec_h\|_{\X_h},
\]
with a constant independent of $h$. 

The symmetric term $(\mean{\partial_n v_h},\jump\varphi)_{\Lambda\setminus\Gamma}$ is bounded
identically, using \autoref{theorem:inverse_inequality} on $v_h$ (polynomial, unchanged from the original
proof) and \autoref{theorem:trace} on $\jump\varphi$ (using $\varphi\in H^2$ elementwise); if $\varphi\in
H^1(\Omega\setminus\Gamma)$ globally this jump vanishes identically (\autoref{ass:regularity}), so this term contributes
only through $\mathbf z_h\in\X_h$ when $\mathbf u=\xvec-\mathbf z_h$.

\smallskip\noindent\textbf{Mean-moment term.} The identical argument applies to $B_\Sigma$: for $p=\partial
E_1\cap\partial E_2\in\Sigma$, $\chi|_{E_i}\in H^3(E_i)$ gives $\Drig:\bnabla^2\chi|_{E_i}\in H^1(E_i)$, so
\autoref{theorem:trace} can be used directly to bound
$\|\mean{\Drig:\bnabla^2\chi}\|_{L^2(p)}$ by $h_p^{-1/2}\|\Drig\|_{L^\infty(\Gamma)}\|\bnabla^2\chi\|_{L^2(E_1\cup
E_2)} + h_p^{1/2}\|\Drig\|_{L^\infty(\Gamma)}\|\bnabla^3\chi\|_{L^2(E_1\cup E_2)}$. Summing over $\Sigma$ as in the proof of \autoref{theorem:discrete_continuity} produces the analogous bound $C_6'(\|\bnabla^2\chi\|_{\Gamma}+h^{1/2}\mathcal R_\Gamma(\chi))\|\yvec_h\|_{\X_h}$, with $\mathcal R_\Gamma(\chi)^2:=\sum_E\|\bnabla^3\chi\|^2_{L^2(E)}$. Since $h^{1/2}\mathcal R_\Gamma(\chi) \le h^{1/2}\mathcal R(\mathbf u) \le \vertiii{\mathbf u}_h$, the resulting estimate also satisfies 
\[
|B_\Sigma(\mathbf u,\yvec_h)|
\le
C\,\vertiii{\mathbf u}_h\,\|\yvec_h\|_{\X_h},
\]
with a mesh-independent constant.

Summing all five blocks via the triangle inequality, the discrete-only blocks contribute an $O(1)$ amount (\autoref{theorem:discrete_continuity}), while
the mean-flux and mean-moment terms satisfy the same
mesh-independent bound as the remaining blocks.
Therefore,
\[
|B(\mathbf u,\yvec_h)| \le \widetilde M\, \vertiii{\mathbf u}_h\, \|\yvec_h\|_{\X_h},
\]
with $\widetilde M=O(1)$ independent of $h$. This completes the proof.
\end{proof}

\subsection{Well-posedness}\label{sec:wellposedness}\noindent
Before stating the well-posedness theorem, we recall that the linear forcing functional generated by the incoming Sommerfeld mode is given by
\begin{equation}
\label{eq:rhs}
L(\mathbf y_h) = \big(-2\upi k_0\,\phi^{\rm in},\,v_h\big)_{\Gamma_{\rm in}}, \qquad \mathbf y_h=[v_h,r_h]\in\X_h.
\end{equation}
Note that $L$ depends on $\mathbf y_h$ only through its fluid component $v_h$ and is supported entirely on
$\Gamma_{\rm in}$, and
$k_0=\omega/c\in\mathbb R$.

\begin{lemma}[Boundedness of the linear functional]\label{lem:rhs_bound}
Let $L$ be given by~\eqref{eq:rhs}. Then $L\in\X_h'$, with $\X_h'$ being the dual space of $\X_h$, and
\begin{equation}
\label{eq:rhs_bound}
|L(\mathbf y_h)| \le C_\ell\,\|\mathbf y_h\|_{\X_h} \qquad\forall\,\mathbf y_h\in\X_h, \qquad
C_\ell := 2|k_0|\,|\beta|\,C_0^{1/2}\widehat C_1,
\end{equation}
with $C_0=\|\psi_0\|^2_{L^2}$ as in \autoref{ass:modal} and $\widehat C_1$ as in \autoref{theorem:discrete_continuity}.
\end{lemma}
\begin{proof}
By Cauchy--Schwarz on $\Gamma_{\rm in}$,
\[
|L(\mathbf y_h)| = 2|k_0|\,\big|(\phi^{\rm in},v_h)_{\Gamma_{\rm in}}\big| \le 2|k_0|\,\|\phi^{\rm
in}\|_{L^2(\Gamma_{\rm in})}\,\|v_h\|_{L^2(\Gamma_{\rm in})}.
\]
Since $|e^{-\upi k_0 L_{\rm dom}}|=1$, $\|\phi^{\rm in}\|_{L^2(\Gamma_{\rm in})} = |\beta|\,\|\psi_0\|_{L^2} =
|\beta|\,C_0^{1/2}$. As $\Gamma_{\rm in}\subset\partial\Omega$, we can reuse the same bound $\|\phi_h\|^2_{L^2(\partial\Omega)}\le\widehat
C_1^2\|\mathbf x_h\|^2_{\mathcal X_h}$ established in the proof of \autoref{theorem:discrete_continuity} (Step 2), applied to $v_h$ restricted to the subset $\Gamma_{\rm in}$ of $\partial\Omega$, giving $\|v_h\|_{L^2(\Gamma_{\rm in})}\le
\widehat C_1\|\mathbf y_h\|_{\mathcal X_h}$. Combining the two bounds gives~\eqref{eq:rhs_bound}.
\end{proof}

In particular, since \[
\|L\|_{\mathcal X_h'} := \sup_{\mathbf y_h\in\mathcal X_h\setminus
\{0\}} |L(\mathbf y_h)|/\|\mathbf y_h\|_{\mathcal X_h},
\]
we have that $\|L\|_{\mathcal X_h'}\le C_\ell$, mesh-independent, a direct consequence of the same trace estimate already used to control $B_{\partial\Omega}$ in \autoref{theorem:discrete_continuity}. This result will be used in the stability bound of \autoref{thm:wellposed} below.

\begin{lemma}[Consistency]\label{lem:consistency}
Let \autoref{ass:regularity} hold for $\xvec=[\phi,w]$. Then
\[
B(\xvec,\yvec_h) = L(\yvec_h) \qquad \forall\,\yvec_h=[v_h,r_h]\in\X_h.
\]
\end{lemma}

\begin{proof}
We verify the identity block by block, following the decomposition of $B$ used in the proof of \autoref{theorem:discrete_continuity}. 

\smallskip\noindent\textbf{Fluid terms.} Since $\phi\in H^1(\Omega\setminus\Gamma)$ globally
(\autoref{ass:regularity}), $\jump\phi=0$ on every $e\in\Lambda\setminus\Gamma$, so the penalty term
$(\mu_1\jump\phi,\jump{v_h})_{\Lambda\setminus\Gamma}=0$, i.e. vanishes identically. Elementwise integration by parts on
each $K\in\mathcal T_h$ gives
\[
(\bnabla\phi,\bnabla v_h)_K = -(\Delta\phi,v_h)_K + (\partial_n\phi,v_h)_{\partial K},
\]
noting that $\phi$ solves the strong interior equation on $K$ (Section~\ref{sec: analysis}), i.e. $-\Delta\phi=0$. Summing over $K\in\mathcal T_h$, we have that
\[
\sum_{K}(\partial_n\phi,v_h)_{\partial K} = -(\mean{\partial_n\phi},\jump{v_h})
_{\Lambda\setminus\Gamma} + (\upi\omega w,\jump{v_h})_\Gamma - (\alpha
\phi,v_h)_{\Gamma_{\rm fs}} - K(\phi,v_h),
\]
where we have used the strong fluid--structure interface condition relating $\partial_n\phi$ on $\Gamma$ to $w$ and the strong Robin/Sommerfeld boundary conditions on $\Gamma_{\rm fs}$, $\Gamma_{\rm in}\cup\Gamma_{\rm out}$. Combined with $(\bnabla\phi,\bnabla v_h)_\Omega$, this reproduces exactly the first line of~\eqref{eq:bilinear_form_analysis} together with the $\Lambda\setminus\Gamma$ and $\Gamma$ coupling terms of the second line, with no residual.

\smallskip\noindent\textbf{Structure terms.} Since $\bnabla w$ is continuous across $\Sigma$
(\autoref{ass:regularity}), $\jump{\bnabla w\otimes\boldsymbol n_\Gamma}=0$ on every $p\in\Sigma$, so the
penalty term $(\mu_2\jump{\bnabla r_h\otimes\boldsymbol n_\Gamma},\jump{\bnabla w\otimes\boldsymbol
n_\Gamma})_\Sigma$ vanishes identically. Elementwise integration by parts (twice, by parts on each
$E\in\mathcal T_h^\Gamma$) applied to $(\Drig:\bnabla^2w,\bnabla^2r_h)_E$, using that $w$ solves the strong plate
equation on $E$ and that $\Drig:\bnabla^2w$ is single-valued across every $p\in\Sigma$
(\autoref{ass:regularity}), reduces the sum over $E\in\mathcal T_h^\Gamma$ exactly to $(I_0\omega^2
w,r_h)_\Gamma + (\upi\omega\rho_w\jump\phi,r_h)_\Gamma - (\mean{\Drig:\bnabla^2w},\jump{\bnabla
r_h\otimes\boldsymbol n_\Gamma})_\Sigma$ plus boundary terms on $\partial\Gamma$ that vanish under the assumed
free-plate boundary conditions (equation~\eqref{eqn: Bp f/s}).

Multiplying by $\mu_p$ reproduces exactly the bracketed line of~\eqref{eq:bilinear_form_analysis} (excluding the
$(\mean{\Drig:\bnabla^2r_h},\jump{\bnabla w\otimes\boldsymbol n_\Gamma})_\Sigma$ term, which vanishes identically
since $\jump{\bnabla w\otimes\boldsymbol n_\Gamma}=0$, and the $\mu_2$-penalty term, which vanishes for the same
reason).

Summing all blocks gives $B(\xvec,\yvec_h)=L(\yvec_h)$ for every $\yvec_h\in\X_h$, as claimed.
\end{proof}

\begin{lemma}[Radiation identity]\label{lem:radiation_identity}
For $\xvec=(\phi,w)\in \X$ solving the homogeneous continuous problem, $B(\xvec,\xvec)=0$, at $\mu_p=1/\rho_w$. Then
\[
\operatorname{Im}B(\xvec,\xvec) = -\sum_{m\in\mathcal P} k_m\big(|R_m|^2 C_m + |T_m|^2 C_m\big),
\]
where $\mathcal P$ is the set of propagating modes, with $k_m\in\mathbb R$.
\end{lemma}
\begin{proof}
The fluid--plate coupling terms $(i\omega w,\jump\phi)_\Gamma$ and
$\mu_p(i\omega\rho_w\jump\phi,w)_\Gamma$ are complex conjugates at
$\mu_p=1/\rho_w$ and cancel exactly in $\operatorname{Im}B(\xvec,\xvec)$ (see the proof of the Gårding inequality~\ref{thm:garding}). Moreover, all elastic-plate terms are real. 

Therefore, the only imaginary contribution is $\operatorname{Im}K(\phi,\phi)=\sum_{m\in\mathcal P}k_m(|R_m|^2C_m+|T_m|^2C_m)$. Since $\operatorname{Im}(\upi k_{mn})=\kappa_{mn}$ and by \autoref{ass:modal}
evanescent modes do not contribute to $\operatorname{Im}K(\phi,\phi)$, we have that
\[
\operatorname{Im}K(\phi,\phi)=\sum_{(m,n)\in\mathcal P}\frac{\kappa_{mn}}{C_mB_n}
\Big(\big|(\phi,\psi_m\chi_n)_{\Gamma_{\rm in}}\big|^2
+\big|(\phi,\psi_m\chi_n)_{\Gamma_{\rm out}}\big|^2\Big),
\]
the sum restricted to the propagating set $\mathcal P:=\{(m,n):k_{mn}\in\mathbb R\}$,
where $\kappa_{mn}=k_{mn}$.

\end{proof}

\begin{proposition}[Reduction of injectivity]\label{prop:reduction}
If the homogeneous continuous problem admits a nontrivial solution $\xvec$, then $R_m=T_m=0$ for every propagating mode (\autoref{lem:radiation_identity}), so $\xvec$ carries no radiated energy. Therefore, by unique continuation of the Laplace equation, $\phi\equiv 0$ on $\Omega\setminus\Gamma$, the solution reduces to an eigenfunction of the compactly-supported interior hydroelastic operator. Consequently, loss of uniqueness can occur only at isolated interior resonance frequencies.
\end{proposition}
\begin{remark}
This does not rule out the absence of such resonances. Trapped modes are known to occur for related floating-elastic-plate geometries at isolated frequencies. Nonetheless, Proposition~\ref{prop:reduction} narrows \autoref{ass:continuous_nonresonance} from an open-ended statement about the full exterior problem to a finite-dimensional interior eigenvalue condition, which can be verified numerically. In the following well-posedness theorem, this condition is assumed to be satisfied.
\end{remark}

\begin{theorem}[Discrete well-posedness]\label{thm:wellposed}
Let $\mathcal T_h$, $\X_h$, and $B(\cdot,\cdot)$ be the triangulation, finite element space, and bilinear form as in
Theorems~\ref{theorem:discrete_continuity} and~\ref{thm:garding}, let $h\le h^{\ast\ast}$ so that the $h$-independent constants $C_1>0$ and $C_2$ of \autoref{thm:garding} hold, let $L$ be the excitation linear functional~\eqref{eq:rhs}, and let Assumptions~\ref{ass:rigidity},~\ref{ass:modal}, and \ref{ass:continuous_nonresonance} hold. Then there exists $h_0\in(0,h^{\ast\ast}]$ and a constant $\gamma>0$, both independent of $h$, such that for every $h<h_0$:
\begin{enumerate}
\item[\rm(i)] \textbf{(Discrete inf-sup condition.)}
\begin{equation}
\label{eq:infsup}
\inf_{\xvec_h\in\X_h\setminus\{0\}} \sup_{\yvec_h\in\X_h\setminus\{0\}}\frac{|B(\xvec_h,\yvec_h)|}{\|\xvec_h\|_{\X_h}\|\yvec_h\|_{\X_h}} \ge \gamma.
\end{equation}
\item[\rm(ii)] \textbf{(Existence, uniqueness, and stability.)} The discrete variational problem: \textit{Find }$\xvec_h\in\X_h$ \textit{such that}
\begin{equation}
\label{eq:discrete_problem}
B(\xvec_h,\yvec_h)=L(\yvec_h)\quad\forall\,\yvec_h\in\X_h,
\end{equation}
has a unique solution, which satisfies the stability bound
$\|\xvec_h\|_{\X_h}\le\gamma^{-1}\|L\|_{\X_h'}\le C_\ell\,\gamma^{-1}$, with $C_\ell$ as defined in \autoref{lem:rhs_bound}.
\item[\rm(iii)] \textbf{(Quasi-optimality.)} For every $\yvec_h\in\X_h$,
\begin{equation}
\label{eq:cea}
\|\xvec_h-\yvec_h\|_{\X_h} \le \Big(1+\frac{M}{\gamma}\Big)\inf_{\mathbf z_h\in\X_h}\|\xvec_h-\mathbf z_h\|_{\X_h}.
\end{equation}
Furthermore, under \autoref{ass:regularity}, by \autoref{lem:consistency} (consistency) and \autoref{lem:cont_ext} (extended continuity), the Galerkin solution $\xvec_h$  of~\eqref{eq:discrete_problem} satisfies the quasi-optimal error bound
\begin{equation}
\label{eq:cea_final}
\|\xvec-\xvec_h\|_{\X_h}\le\big(1+\widetilde M/\gamma\big)\inf_{\mathbf z_h\in\X_h}\vertiii{\xvec-\mathbf z_h}_h,
\end{equation}
with $\widetilde M$ and $\vertiii{\cdot}_h$ as defined in \autoref{lem:cont_ext}.
\end{enumerate}
\end{theorem}

\begin{proof}
\smallskip
\noindent\textbf{(i) Discrete inf-sup condition.} We prove the discrete inf-sup condition~\eqref{eq:infsup} by contradiction. Suppose the claim is false, then there exist a sequence of meshes $h_n\le 1/n$ for $n\in\mathbb N$ and a sequence of solution fields $\xvec_{h_n}\in\X_{h_n}$ with $\|\xvec_{h_n}\|_{\X_{h_n}}=1$ such that
\begin{equation}
\label{eq:contra_hyp}
\sup_{\yvec_{h_n}\in\X_{h_n}\setminus\{0\}}\frac{|B(\xvec_{h_n},\yvec_{h_n})|}{\|\yvec_{h_n}\|_{\X_{h_n}}} \;\longrightarrow\; 0 \qquad (\mbox{for }n\to\infty).
\end{equation}
We derive a contradiction with $\|\xvec_{h_n}\|_{\X_{h_n}}=1$.

Selecting $\yvec_{h_n}=\xvec_{h_n}$, the Gårding inequality (\autoref{thm:garding}) gives
\[ 
C_1 A_h(\xvec_{h_n}) \le \operatorname{Re}B(\xvec_{h_n},\xvec_{h_n}) + C_2 E_h(\xvec_{h_n}),
\]
where we have used 
\[ 
A_h(\xvec_h):= \|\bnabla\phi_h\|_\Omega^2 + \sum_{e\in\Lambda\setminus\Gamma} h_e^{-1}\|\jump{\phi_h}\|_{L^2(e)}^2 + \mu_p\|\bnabla^2w_h\|_{\Gamma}^2 + \mu_p \sum_{p\in\Sigma} h_p^{-1} \|\jump{\bnabla w_h\otimes\boldsymbol n_{\mathcal P}}\|_{L^2(p)}^2 ,
\]
and
\[
E_h(\xvec_h):= \|\phi_h\|_{L^2(\Omega)}^2 + \mu_p\|w_h\|_{L^2(\Gamma)}^2 . 
\]
By construction of the energy norm, $\|\xvec_h\|_{\X_h}^2 = A_h(\xvec_h) + E_h(\xvec_h)$. Therefore, for the sequence $\xvec_{h_n}$ with $\|\xvec_{h_n}\|_{\X_{h_n}}=1$, we have $A_{h_n}(\xvec_{h_n}) = 1 - E_{h_n}(\xvec_{h_n}) $. Hence, the Gårding inequality gives
\[
C_1 \le \operatorname{Re}B(\xvec_{h_n},\xvec_{h_n}) + (C_1 + C_2) E_h(\xvec_{h_n}).
\]

Moreover, taking $\yvec_{h_n}=\xvec_{h_n}$ in~\eqref{eq:contra_hyp} gives $|B(\xvec_{h_n},\xvec_{h_n})|\to0$ when $n\to\infty$, hence in particular
$\mathrm{Re}\,B(\xvec_{h_n},\xvec_{h_n})\to0$. Which yields the following infimum limit inequality
\[
\liminf_{n\to\infty} \Big( \|\phi_n\|_{L^2(\Omega)}^2 + \mu_p\|w_n\|_{L^2(\Gamma)}^2 \Big) \ge \frac{C_1}{C_1 + C_2}>0. 
\]
Since $\|\mathbf x_{h_n}\|_{\X_{h_n}}=1$, the normalization also gives $A_{h_n}(\mathbf x_{h_n})\le1$ for every
$n$, thus
\[
\sup_n\Big(\|\nabla\phi_{h_n}\|_{L^2(\Omega)}+J_f(\phi_{h_n})+\mu_p^{1/2}\|\nabla^2w_{h_n}\|_{L^2(\Gamma)}
+\mu_p^{1/2}J_p(w_{h_n})\Big)<\infty.
\]
Invoking \autoref{lem:DGcompactness}, we have that there exist
$\phi\in H^1(\Omega\setminus\Gamma)$ and $w\in H^2(\Gamma)$ such that $\phi_{h_n}-R_h\phi_{h_n}\to0$, $\nabla w_{h_n}-\nabla S_hw_{h_n}\to0$ in $L^2(\Omega)$ and $L^2(\Gamma)$, respectively. Then, in particular, $\phi_{h_n}\to\phi$ in $L^2(\Omega)$ and $\nabla w_{h_n}\to\nabla w$ in $L^2(\Gamma)$.

Let $\xvec:=(\phi,w)\in\X:=H^1(\Omega\setminus\Gamma)\times H^2(\Gamma)$. Fix an arbitrary $\yvec=(v,r)\in\X$
and let $\yvec_{h_n}\in\X_{h_n}$ be a $\X_{h_n}$-conforming interpolant of $\yvec$ with
$\|\yvec-\yvec_{h_n}\|_{\X_{h_n}}\to0$. Since $\yvec$ is a continuous field, the
interpolation additionally gives $J_f(v_{h_n})\to0$, $J_p(r_{h_n})\to0$. Decomposing
$B(\xvec_{h_n},\yvec_{h_n})$ into the five blocks $B_\Omega,B_{\partial\Omega},B_\Lambda,B_\Gamma,B_\Sigma$ of
\autoref{theorem:discrete_continuity}, we have that:
\begin{itemize}
\item[a)] $B_\Omega(\phi_{h_n},v_{h_n})\to(\nabla\phi,\nabla v)_\Omega$ by weak--strong convergence of
$\nabla\phi_{h_n}\rightharpoonup\nabla\phi$ against $\nabla R_hv_{h_n}\to\nabla v$ strongly, together with
$\nabla(\phi_{h_n}-R_h\phi_{h_n})\to0$.
\item[b)] $B_{\partial\Omega}$: the free-surface term converges by the strong $L^2(\partial\Omega)$ convergence
of $R_h\phi_{h_n}\to\phi$ and $R_hv_{h_n}\to v$ (\autoref{lem:DGcompactness}). The radiation term $K(\phi_{h_n},v_{h_n})$ converges term-wise by the same strong trace convergence.
\item[c)] $B_\Lambda$: The terms $(\mean{\partial_n\phi_{h_n}},\jump{v_{h_n}})_{\Lambda\setminus\Gamma}$ and $(\mu_1\jump{\phi_{h_n}},\jump{v_{h_n}})_{\Lambda\setminus\Gamma}$ are bounded, see Step 3 in \autoref{theorem:discrete_continuity}, and vanish since $J_f(v_{h_n})\to0$. For the symmetric term we use $v_{h_n}\to v$ strongly in $H^1(\Omega)$, multiplied against the bounded $J_f(\phi_{h_n})$, thus it also vanishes. Hence $B_\Lambda\to0$.
\item[d)] $B_\Gamma$: converges via strong $L^2(\Gamma)$ convergence of $\jump{\phi_{h_n}}\to\jump\phi$ and $w_{h_n},S_hw_{h_n}\to w$ strongly in $H^1(\Gamma)\hookrightarrow L^2(\Gamma)$.
\item[e)] $B_\Sigma$: analogous Cauchy--Schwarz splitting using \autoref{lemma:inverse_sigma} and $J_p(r_{h_n})\to0$ gives $B_\Sigma\to0$.
\end{itemize}
Collecting the five limits,
\[
B(\xvec_{h_n},\yvec_{h_n})\;\longrightarrow\;B(\xvec,\yvec).
\]
On the other hand, $|B(\xvec_{h_n},\yvec_{h_n})|\le\|\yvec_{h_n}\|_{\X_{h_n}}
\sup_{\yvec_{h_n}\ne0}|B(\xvec_{h_n},\yvec_{h_n})|/\|\yvec_{h_n}\|_{\X_{h_n}}\to0$ by~\eqref{eq:contra_hyp} and
boundedness of $\|\yvec_{h_n}\|_{\X_{h_n}}$. Since $\yvec\in \X$ was arbitrary,
$B(\xvec,\yvec)=0$ for all $\yvec\in \X$.

By \autoref{ass:continuous_nonresonance}, $B(\xvec,\yvec)=0\ \forall \yvec\in \X$
implies $\xvec=0$, i.e.\ $\phi=0$ and $w=0$. In particular
$\|\phi_{h_n}\|_{\Omega_{h_n}}\to\|\phi\|_{L^2(\Omega)}=0$ and $\|w_{h_n}\|_{\Gamma_{h_n}}\to\|w\|_{L^2(\Gamma)}=0$. This contradicts
\[
\liminf_{n\to\infty} \Big( \|\phi_n\|_{L^2(\Omega)}^2 + \mu_p\|w_n\|_{L^2(\Gamma)}^2 \Big) \ge \frac{C_1}{C_1+C_2}>0,
\]
for fixed $C_1>0$ and $C_2>0$. This contradiction shows that no such sequence $(h_n,\xvec_{h_n})$ exists, i.e.\ there is $h_0>0$ and $\gamma>0$ such that~\eqref{eq:infsup} holds for all $h<h_0$.

\smallskip
\noindent\textbf{(ii) Existence, uniqueness, stability.}
Fix $h<h_0$ and a basis of $\X_h$; let $\mathbf B_h$ be the matrix
representation of $B$ and $\mathbf L_h$ the vector representing $L$, so
that~\eqref{eq:discrete_problem} is the square linear system
$\mathbf B_h\mathbf x_h=\mathbf L_h$. Suppose $\mathbf B_h\mathbf x_h=0$ for
some $\xvec_h\in\X_h$; then $B(\xvec_h,\yvec_h)=0$ for every
$\yvec_h\in\X_h$, so the numerator in~\eqref{eq:infsup} vanishes at
$\xvec_h$, and~\eqref{eq:infsup} with $\gamma>0$ forces
$\|\xvec_h\|_{\X_h}=0$, i.e.\ $\xvec_h=0$. Thus $\mathbf B_h$ has trivial
kernel; since $\X_h$ is finite-dimensional, $\mathbf B_h$ is square, and a
square matrix with trivial kernel is invertible. Hence a unique
$\xvec_h\in\X_h$ solving~\eqref{eq:discrete_problem} exists for every $L$.
For the stability bound, apply~\eqref{eq:infsup} to this $\xvec_h$ and use
$B(\xvec_h,\yvec_h)=L(\yvec_h)$:
\[
\gamma\,\|\xvec_h\|_{\X_h} \le \sup_{\yvec_h\ne0}\frac{|B(\xvec_h,\yvec_h)|}{\|\yvec_h\|_{\X_h}} = \sup_{\yvec_h\ne0}\frac{|L(\yvec_h)|}{\|\yvec_h\|_{\X_h}} = \|L\|_{\X_h'} \le C_\ell,
\]
and dividing by $\gamma>0$ gives $\|\xvec_h\|_{\X_h}\le\gamma^{-1}\|L\|_{\X_h'}\le C_\ell\,\gamma^{-1}$.

\smallskip
\noindent\textbf{(iii) Quasi-optimality.}
For an arbitrary
$\mathbf z_h\in\X_h$, since $\xvec_h-\mathbf z_h\in\X_h$,
inequality~\eqref{eq:infsup} gives
\[
\gamma\,\|\xvec_h-\mathbf z_h\|_{\X_h} \le \sup_{\mathbf y_h\ne0}\frac{|B(\xvec_h-\mathbf z_h,\mathbf y_h)|}{\|\mathbf y_h\|_{\X_h}};
\]
under \autoref{ass:regularity}, \autoref{lem:consistency} gives
Galerkin orthogonality $B(\xvec_h-\xvec,\yvec_h)=0\ \forall\yvec_h\in\X_h$,
hence $B(\xvec_h-\mathbf z_h,\yvec_h)=B(\xvec-\mathbf z_h,\yvec_h)$; \autoref{lem:cont_ext}
bounds the right-hand side by
$\widetilde M_h\vertiii{\xvec-\mathbf z_h}_h\|\yvec_h\|_{\X_h}$; dividing by
$\gamma$ gives $\|\xvec_h-\mathbf z_h\|_{\X_h}\le(\widetilde M/\gamma)\vertiii{\xvec-\mathbf z_h}_h$;
the triangle inequality and $\|\xvec-\mathbf z_h\|_{\X_h}\le\vertiii{\xvec-\mathbf z_h}_h$
give $\|\xvec-\xvec_h\|_{\X_h}\le(1+\widetilde M/\gamma)\vertiii{\xvec-\mathbf z_h}_h$;
taking the infimum over $\mathbf z_h\in\X_h$ yields~\eqref{eq:cea_final}.
The unconditional bound~\eqref{eq:cea} (without \autoref{ass:regularity})
follows the same computation with $M,\|\cdot\|_{\X_h}$ in place of
$\widetilde M,\vertiii{\cdot}_h$.
\end{proof}

\section{Numerical results}\label{sec: results}\noindent

\subsection{Implementation}\noindent
We implement the formulation in Section~\ref{sec: discrete formulation} using Gridap \citep{Badia2020,Verdugo2022}, a flexible finite element package written in the Julia programming language. To generate meshes, we use Gmsh \citep{10.1002/nme.2579} and GridapGmsh \citep{gridapgmsh}. Our implementation is open-source, and a link is provided in the \textit{Data availability} section.

As discussed in Section~\ref{sec: sommerfeld 2d} and \ref{sec: sommerfeld 3d}, we implement the Sommerfeld radiation conditions using outer products to discretise $K(\phi,v)$ given by \eqref{eqn: K op 2d} in the case of the two-dimensional problem, or \eqref{eqn: K op 3d} in the case of the three-dimensional channel. For the penalty parameters $\mu_1$ and $\mu_2$ that appear in the formulation, we take $\mu_1=\mu_2=Dr(r+1)/h_e$ where $r$ is the order of the Lagrange polynomials for $\hat{\mathcal{V}}_h$ and $\hat{\mathcal{V}}_{\Gamma,h}$ and $h_e$ is the element size. For all examples, we use 2nd-order Lagrange polynomials, i.e., $r=2$ for both $\hat{\mathcal{V}}_h$ and $\hat{\mathcal{V}}_{\Gamma,h}$. Note that the polynomial order can be increased in our implementation.

\subsection{Validation}\noindent
To validate our formulation and implementation, we consider a submerged horizontal plate as shown in Figure~\ref{fig:schematic} and compare to the solution method of \citet{Wegert_Wilks_Thamwattana_Challis_Koley_Meylan_2026}. In that previous work, the problem is solved by formulating the diffraction and radiation problems as hypersingular boundary integral equations and using a expansion in the plate modes. For the validation, we  take the plate half-length to be $L=10$~m, the fluid depth to be $H=10$~m, the submergence depth to be $h=2$~m, the density of the fluid to be $\rho_w=1025$~kg~m$^{-3}$, the incident wave amplitude to be $A=1$~m, and gravity to be $9.81$~m~s$^{-2}$. For the plate properties, we take the Young's modulus to be $E=3 \times 10^{12}$~Pa, the Poisson's ratio to be $\nu=0.392$, the thickness to be $d=0.01$~m, and the plate density to be $\rho=1780$~kg~m$^{-3}$. This yields a rigidity of roughly $D=2.95 \times 10^5$~Pa~m$^3$ and a mass per unit length of $I_0=17.8$~kg~m$^{-2}$. For the finite element method, we take the computational domain to have half-length $L_f=15$~m and the number of modes in the operator imposing the Sommerfeld radiation conditions to be $M=5$. Finally, we consider free boundary conditions.

\begin{figure}[p]
    \centering
    \includegraphics[width=\linewidth]{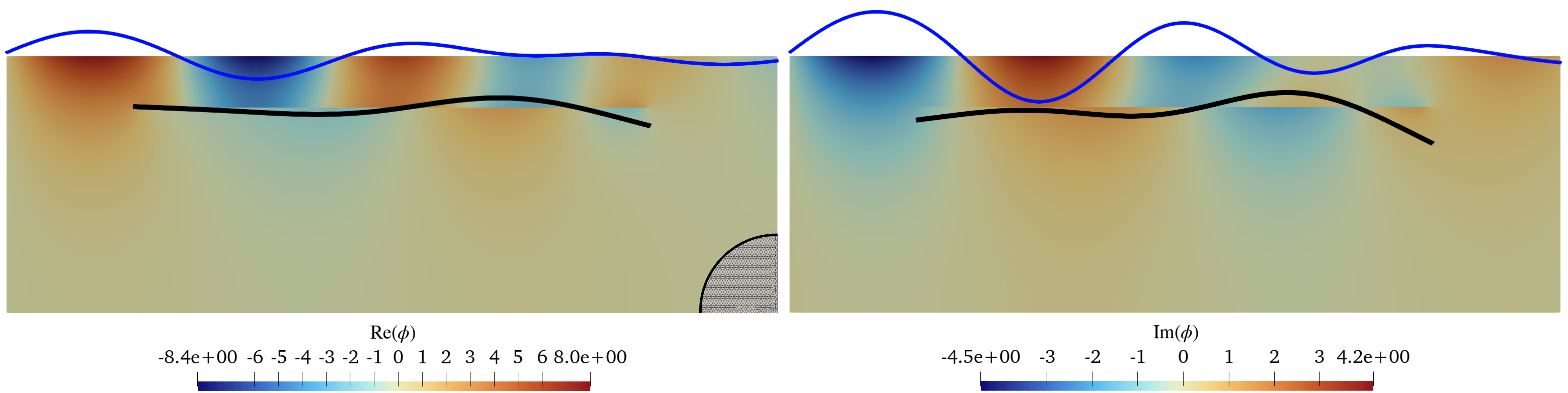}
    \caption{The real (left) and imaginary (right) parts of the finite element solution for the submerged plate hydroelasticity problem for an incident wave of period 3~s. We visualise the potential $\phi$ using a colour map. The free-surface displacement and plate displacement are visualised by a unitary warping of the surface and plate mesh, respectively. The mesh size is approximately $h_{\rm size}=0.1$\,m and is visualised in the lower right part of the real solution for $\phi$.}
    \label{fig:horz_phi_eta_w}
    \centering
    \begin{subfigure}[c]{0.13\textwidth}
        \centering
        \includegraphics[width=1\linewidth]{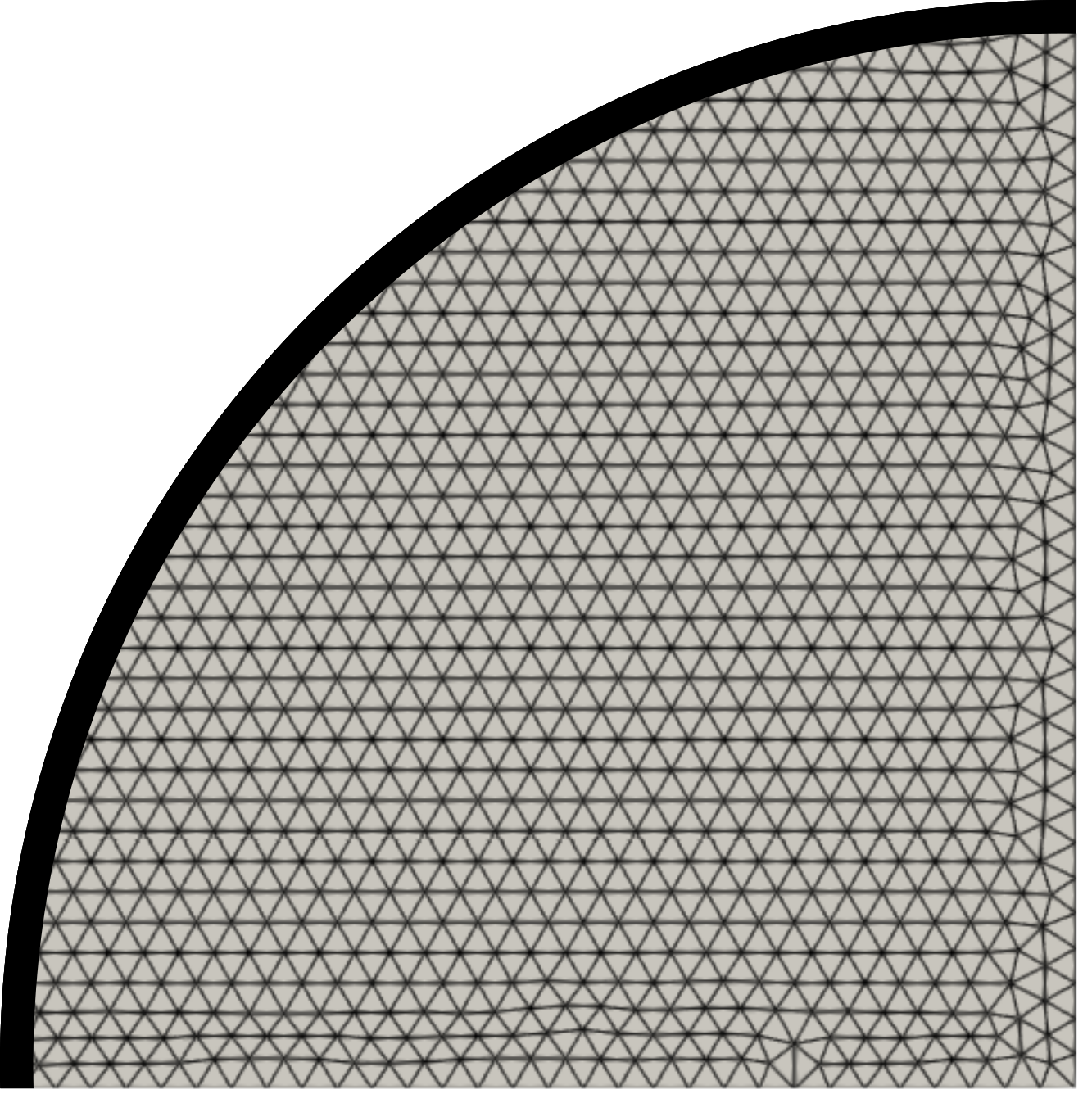}
        $h_{\rm size}=0.1$
    \end{subfigure}
    \begin{subfigure}[c]{0.845\textwidth}
        \centering
        \begin{subfigure}{0.49\textwidth}
            \centering
            \includegraphics[width=0.7\linewidth]{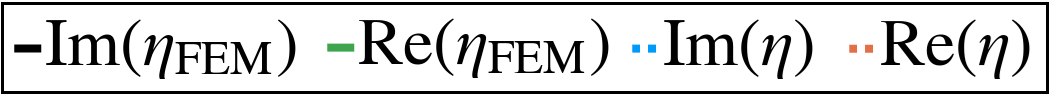}
            \includegraphics[width=1\linewidth]{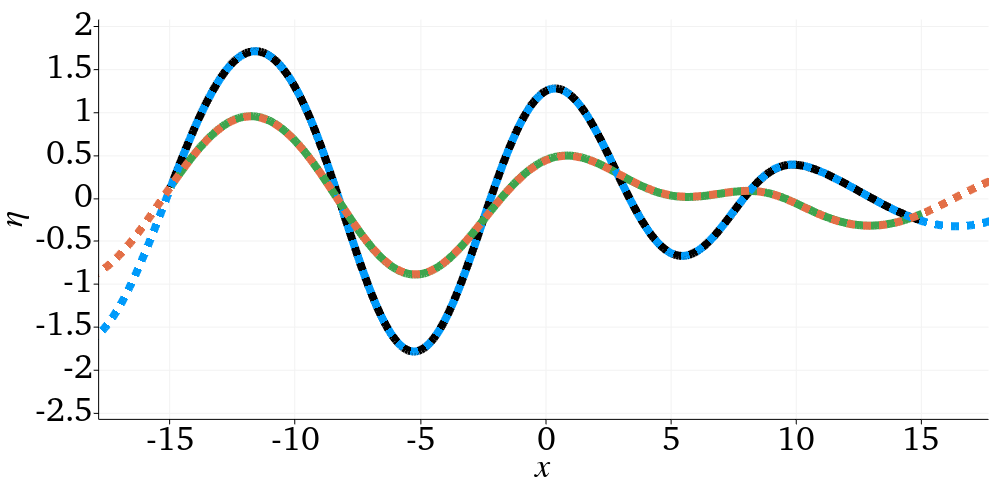}
        \end{subfigure}
        \begin{subfigure}{0.49\textwidth}
            \centering
            \includegraphics[width=0.7\linewidth]{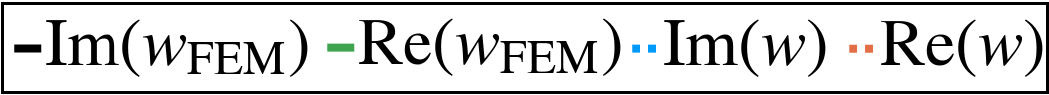}
            \includegraphics[width=1\linewidth]{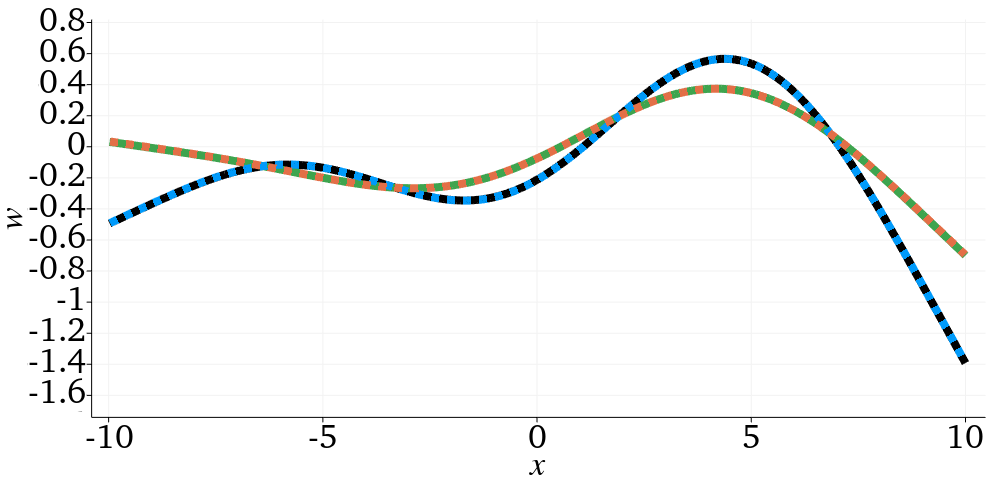}
        \end{subfigure}
    \end{subfigure}
    \begin{subfigure}[c]{0.13\textwidth}
        \centering
        \includegraphics[width=1\linewidth]{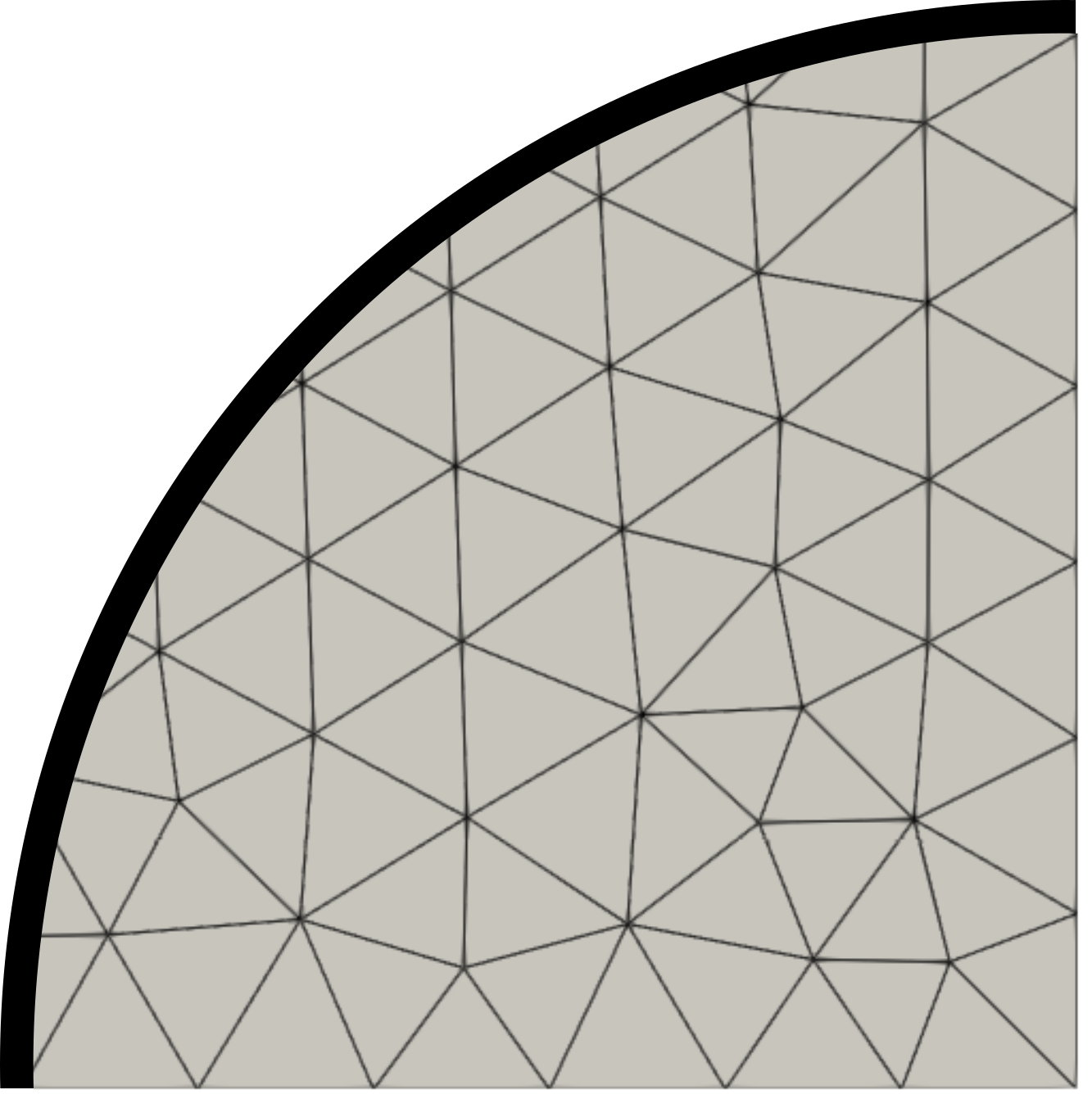}
        $h_{\rm size}=0.5$
    \end{subfigure}
    \begin{subfigure}[c]{0.845\textwidth}
        \centering
            \includegraphics[width=0.49\linewidth]{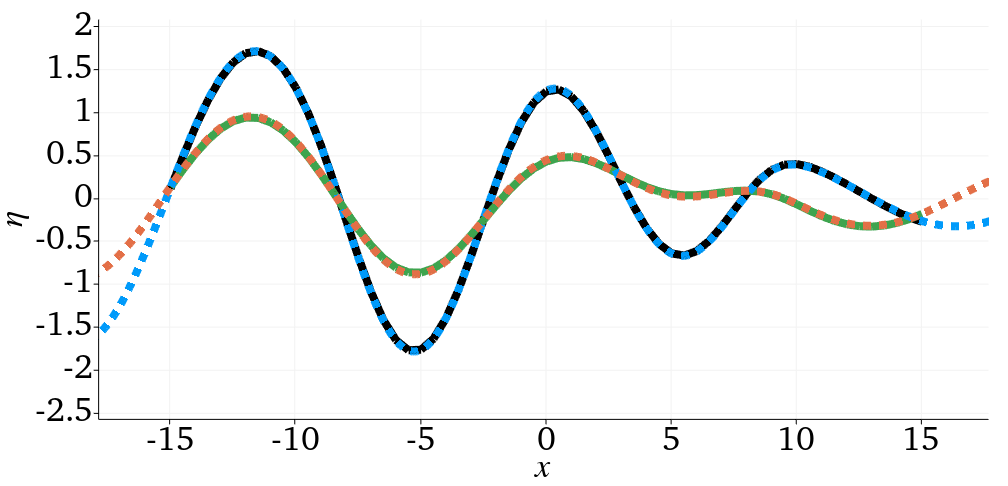}
            \includegraphics[width=0.49\linewidth]{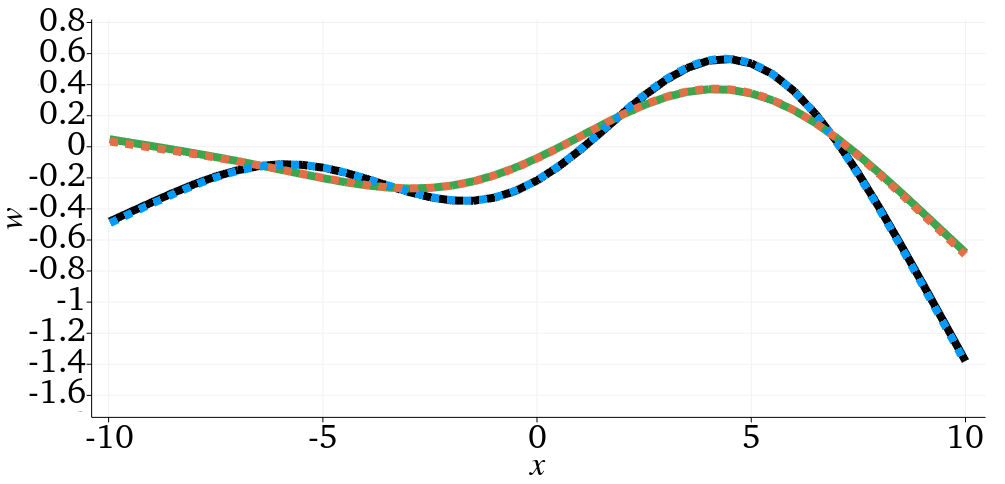}
    \end{subfigure}
    \begin{subfigure}[c]{0.13\textwidth}
        \centering
        \includegraphics[width=1\linewidth]{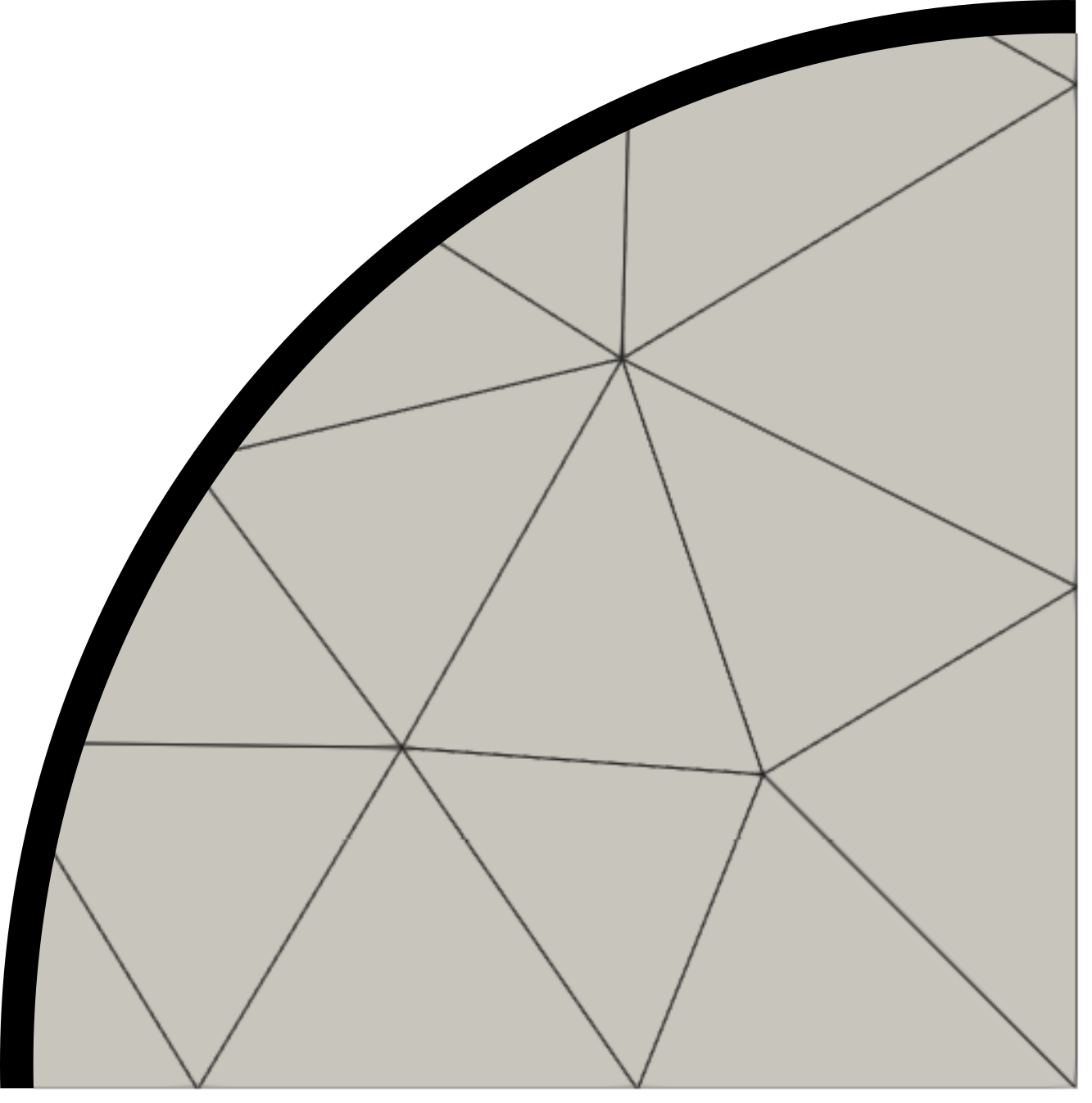}
        $h_{\rm size}=1.5$
    \end{subfigure}
    \begin{subfigure}[c]{0.845\textwidth}
        \centering
        \includegraphics[width=0.49\linewidth]{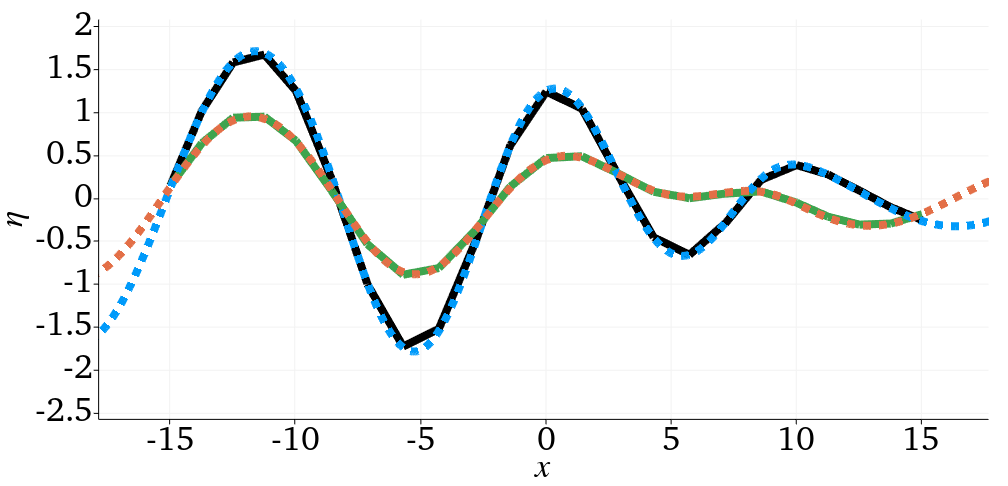}
        \includegraphics[width=0.49\linewidth]{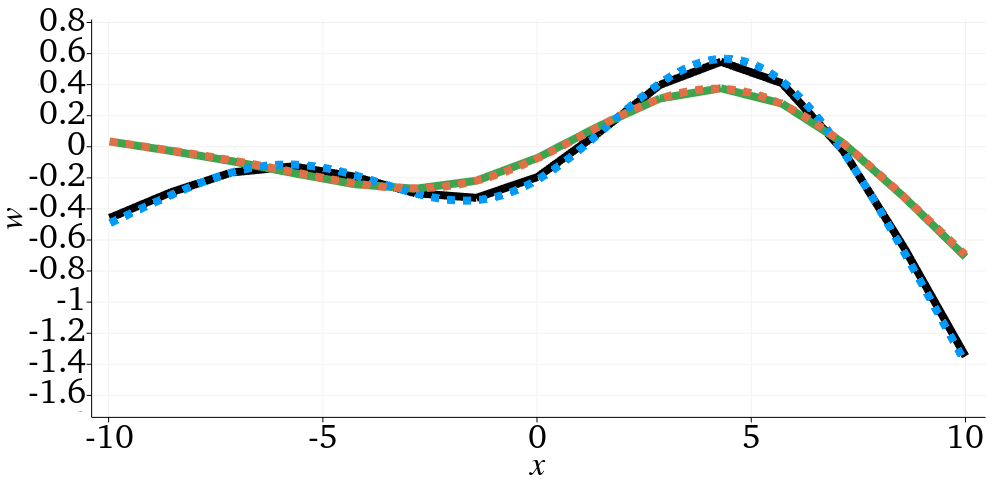}
    \end{subfigure}
    \caption{Comparison between the proposed finite element method and the approach of \citet{Wegert_Wilks_Thamwattana_Challis_Koley_Meylan_2026} for the submerged plate hydroelasticity problem for an incident wave of period 3~s using different mesh sizes.}
    \label{fig:horz_mesh_results}

    \centering
    \captionof{table}{The reflection and transmission coefficients for the submerged plate hydroelasticity problem for an incident wave of period 3~s using various mesh sizes. The second column gives the results using the approach of \citet{Wegert_Wilks_Thamwattana_Challis_Koley_Meylan_2026}.}
    \label{tab:R_and_T_horz}
    \footnotesize
    \begin{tabular}{c|c|c|c|c|c|c}
    Case & \citet{Wegert_Wilks_Thamwattana_Challis_Koley_Meylan_2026} & $h_{\rm size}=0.1$ & $h_{\rm size}=0.25$ & $h_{\rm size}=0.5$ & $h_{\rm size}=1.0$ & $h_{\rm size}=1.5$ \\ \hline
    $\lvert R\rvert$ & 0.941 & 0.942 & 0.940 & 0.938 & 0.937 & 0.943 \\
    $\lvert T\rvert$ & 0.336 & 0.334 & 0.339 & 0.346 & 0.348 & 0.330 \\
    \end{tabular}
\end{figure}

\begin{figure}
    \centering
    \includegraphics[width=0.9\linewidth]{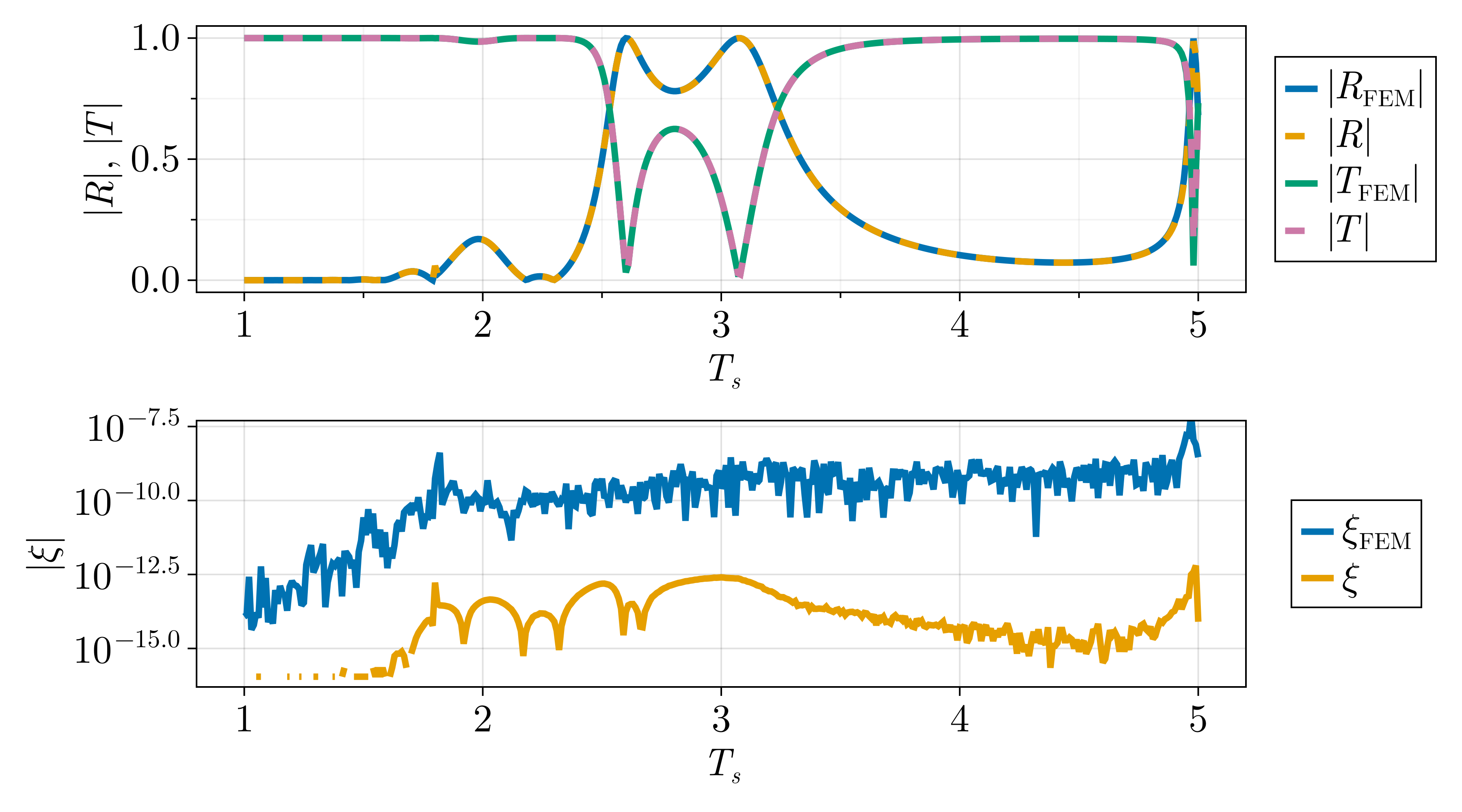}
    \caption{Computed reflection and transmission coefficients for the submerged plate hydroelasticity problem for a range of incident wave periods. The energy conservation $\xi\coloneqq\lvert R\rvert^2+\lvert T\rvert^2-1$ is also shown. }
    \label{fig:ref_trans_2d_compare}
\end{figure}

Figure~\ref{fig:horz_phi_eta_w} shows the real and imaginary parts of the finite element solution for an incident wave of period 3\,s. For this example, we require that the entire mesh have an approximate size of $h_{\rm size}=0.1$\,m, although the mesh could be coarsened near the ocean floor. We visualise the real and imaginary parts of the free-surface displacement and plate displacement by warping the corresponding edges with a scaling of one. The free-surface displacement $\eta$ is calculated as $\eta = \frac{i \omega \phi}{g}$ at $z=0$, which is obtained from the dynamic free surface condition \citep{Wegert_Wilks_Thamwattana_Challis_Koley_Meylan_2026}. In Figure~\ref{fig:horz_mesh_results}, we compare the free-surface displacement and the plate displacement from the proposed finite element scheme with the approach of \citet{Wegert_Wilks_Thamwattana_Challis_Koley_Meylan_2026} for different mesh sizes. This shows excellent agreement, even when the mesh is extremely coarse relative to the size of the problem. In Table~\ref{tab:R_and_T_horz}, we give the reflection and transmission coefficients computed using \eqref{eqn: reflection 2d} and \eqref{eqn: transmission 2d} for $n=0$. Note that $R$ and $T$ are scaled by $1/\beta$ to satisfy energy balance, $\lvert R\rvert^2+\lvert T\rvert^2\approx1$. Again, we see good agreement between the finite element approach and the approach of \citet{Wegert_Wilks_Thamwattana_Challis_Koley_Meylan_2026}. Finally, we consider computing the reflection and transmission coefficients for incident wave periods ranging from 1\,s to 5\,s. For these results, we use an approximate mesh size of $h_{\rm size}=0.1$~m. Figure~\ref{fig:ref_trans_2d_compare} shows the results for the reflection and transmission coefficients, as well as energy conservation $\xi\coloneqq\lvert R\rvert^2+\lvert T\rvert^2-1=0$. We see that both methods are in good agreement over a range of wave periods. 

\subsection{Slanted plates}\noindent
In this section, we consider the case of a slanted plate with end points $(-L,-h_1)$ and $(L,-h_2)$. Extending the implementation to this case requires no adjustments to the code, other than the mesh generation. This is of considerable benefit, as the approach of \citet{Wegert_Wilks_Thamwattana_Challis_Koley_Meylan_2026} would require several modifications.

For this example, we take $L=10$~m, $h_1=0.5$\,m, $h_2=7$\,m, and the mesh size to be approximately $h_{\rm size}=0.25$\,m. The remaining parameters are as in the example shown in Figure~\ref{fig:horz_phi_eta_w}. We take the boundary conditions on the ends of the plate to either be both free or both simply-supported. In Figure~\ref{fig:slanted} we give the results for the potential $\phi$, free-surface displacement $\eta$, and the plate displacement $w$. In addition, we find that the reflection and transmission coefficients are approximately $\lvert R\rvert=0.756$ and $\lvert T\rvert=0.655$ for free boundary conditions and $\lvert R\rvert=0.969$ and $\lvert T\rvert=0.248$ for simply-supported boundary conditions.

\begin{figure}[!t]
    \centering
    \begin{subfigure}{\textwidth}
        \centering
        \includegraphics[width=0.49\linewidth]{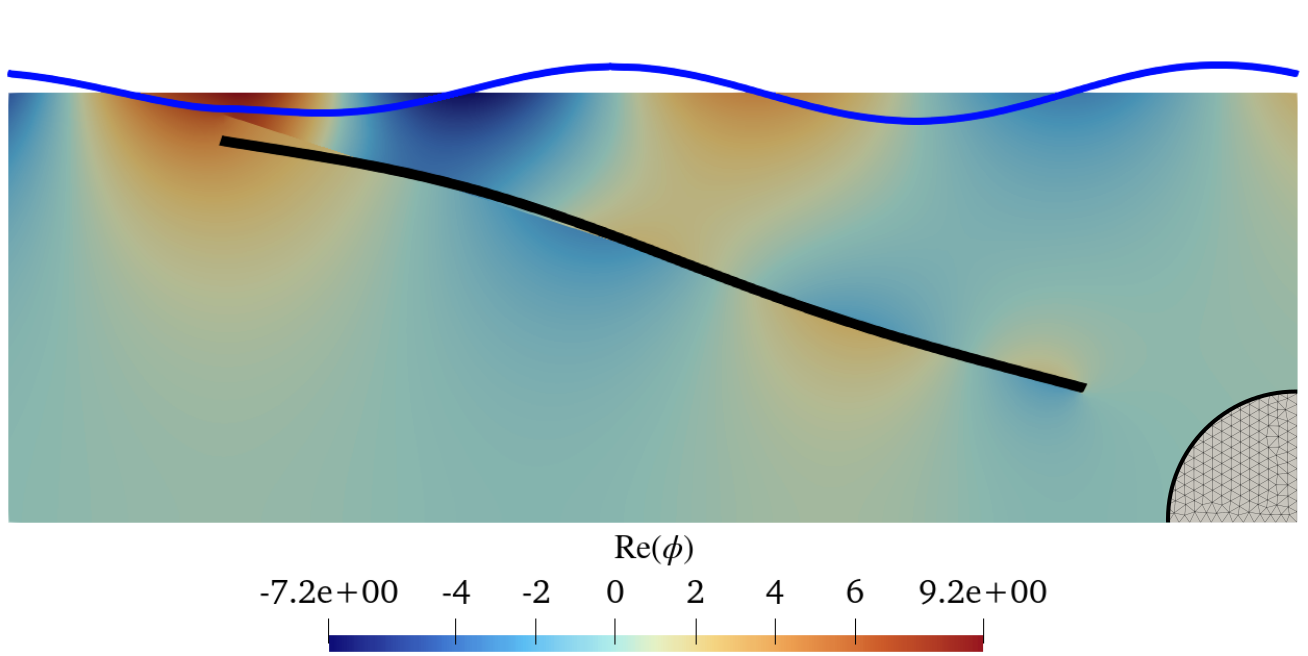}
        \includegraphics[width=0.49\linewidth]{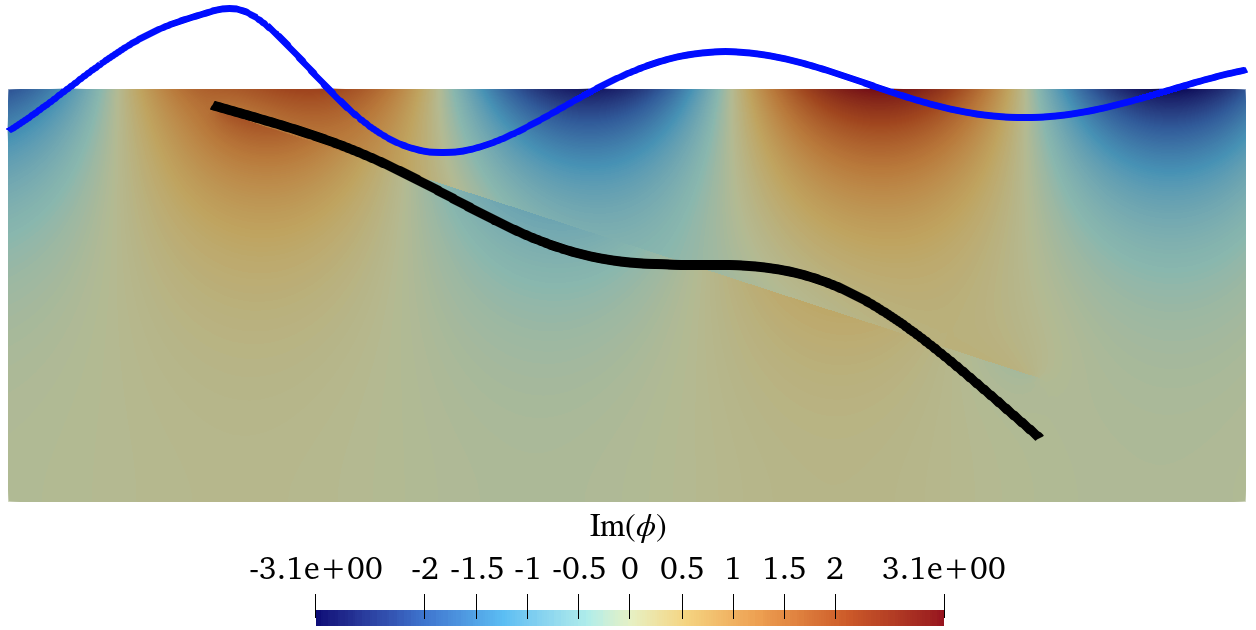}
        \caption{Free boundary conditions}
    \end{subfigure}
    \begin{subfigure}{\textwidth}
        \centering
        \includegraphics[width=0.49\linewidth]{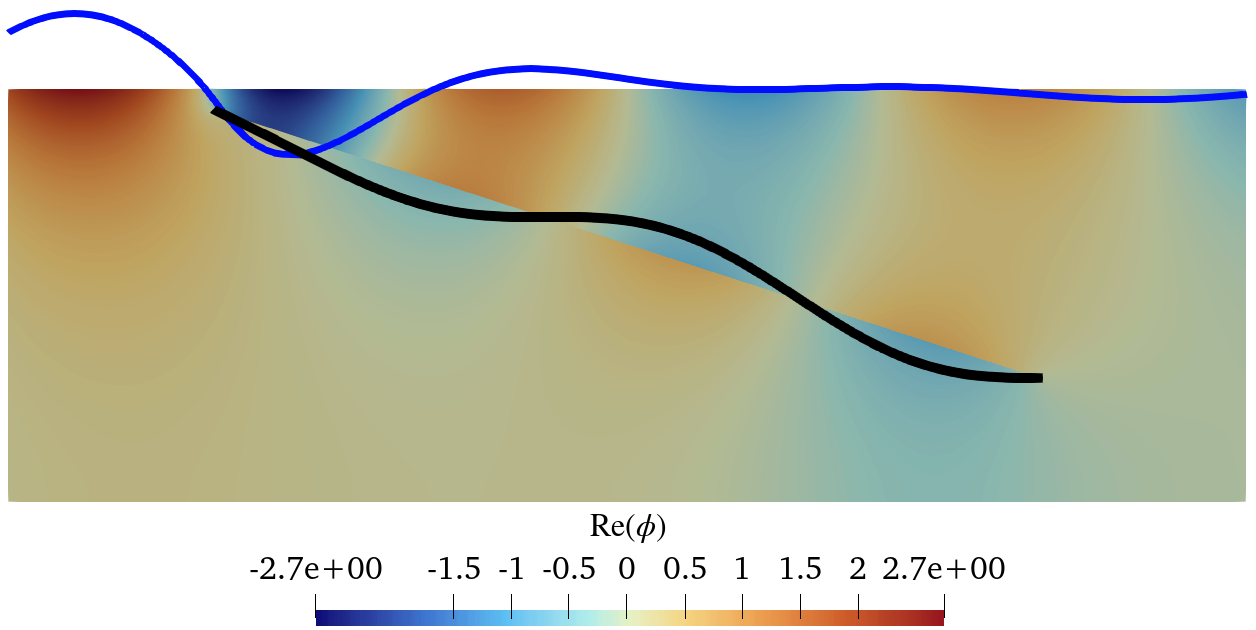}
        \includegraphics[width=0.49\linewidth]{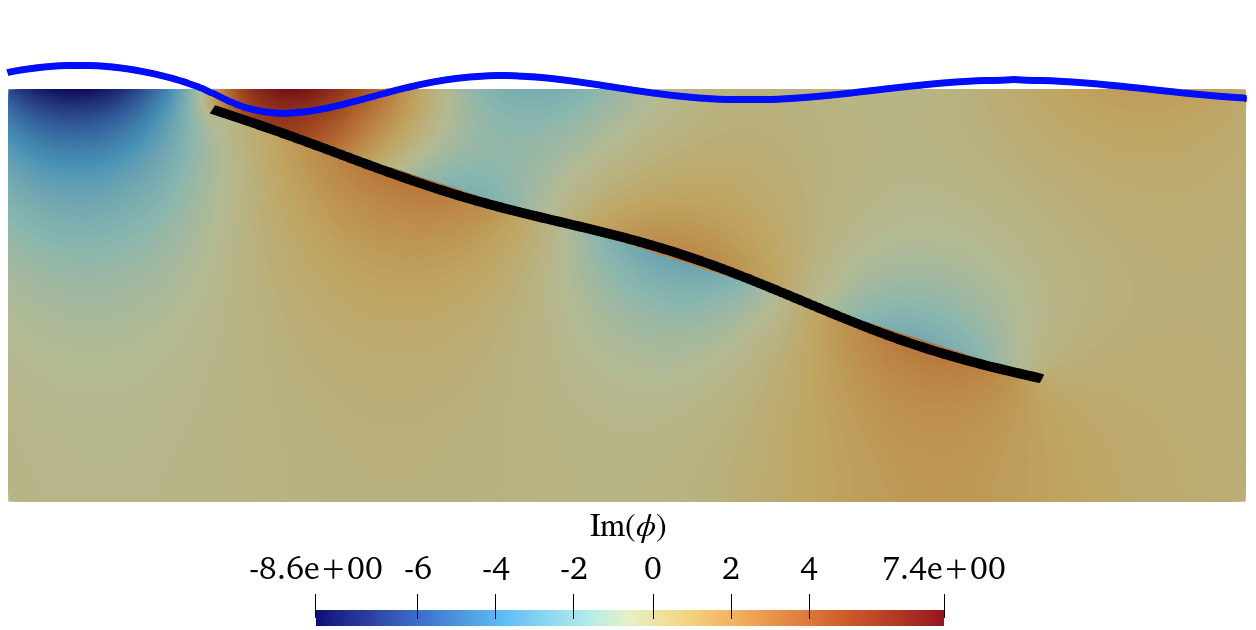}
        \caption{Simply-supported boundary conditions}
    \end{subfigure}
    \caption{Results for the slanted submerged plate for different boundary conditions and an incident wave of period 3\,s. The two cases are (a) free boundary conditions and (b) simply-supported boundary conditions.  }
    \label{fig:slanted}
\end{figure}


\subsection{Three-dimensional channel}\noindent
In this section we consider the case of the three-dimensional channel for three different plate configurations: (i) a centred rectangular plate, (ii) an offset rectangular plate, and (iii) an annulus-shaped plate.

For all cases, we take the half-width of the channel to be $B_f=15$\,m, the half-length of the computational domain in $x$ to be $L_f=12.5$\,m, the submergence depth of the plates to be $h=1$~m, the incident wave to be planar with period 4\,s and amplitude $A=1$\,m, the mesh size at the top of the fluid and on the plate to be approximately 0.25\,m, the mesh size to be 1.5\,m on the channel floor, and the plate to be isotropic with the same Young's modulus and Poisson's ratio as in the two-dimensional examples. In the non-local operator $K$ that imposes the Sommerfeld radiation conditions, we take the number of modes to be $M=5$ and $N=5$. We note that in three dimensions, $N$ should be large enough so that propagating modes are captured in the sum in Equation~\eqref{eqn: K op 3d}. In other words, $N$ should at least be equal to $q$, where $q$ is the largest integer $i$ such that $k_{0i}$ is real. Finally, to solve the linear system we use an iterative induced dimension reduction method \citep{doi:10.1137/070685804} that is preconditioned with an incomplete LU factorisation. Although this solver is sufficient for the examples in this paper, it is not scalable and in the future focus should be devoted to developing scalable linear solvers for these problems. 

For the centred and offset rectangular plates, the plate half-length is taken to be $L=10$\,m and the plate half-width is taken to be $B=5$\,m. For the offset plate, we shift the plate in the $y$ direction by 3\,m. Figure~\ref{fig:3d_meshes rect} shows the mesh used for the centred rectangular plate. Note that the case of the offset plate is similar. 
\begin{figure}[!t]
    \centering
    \begin{subfigure}{0.49\linewidth}
        \centering
        \includegraphics[width=1\linewidth]{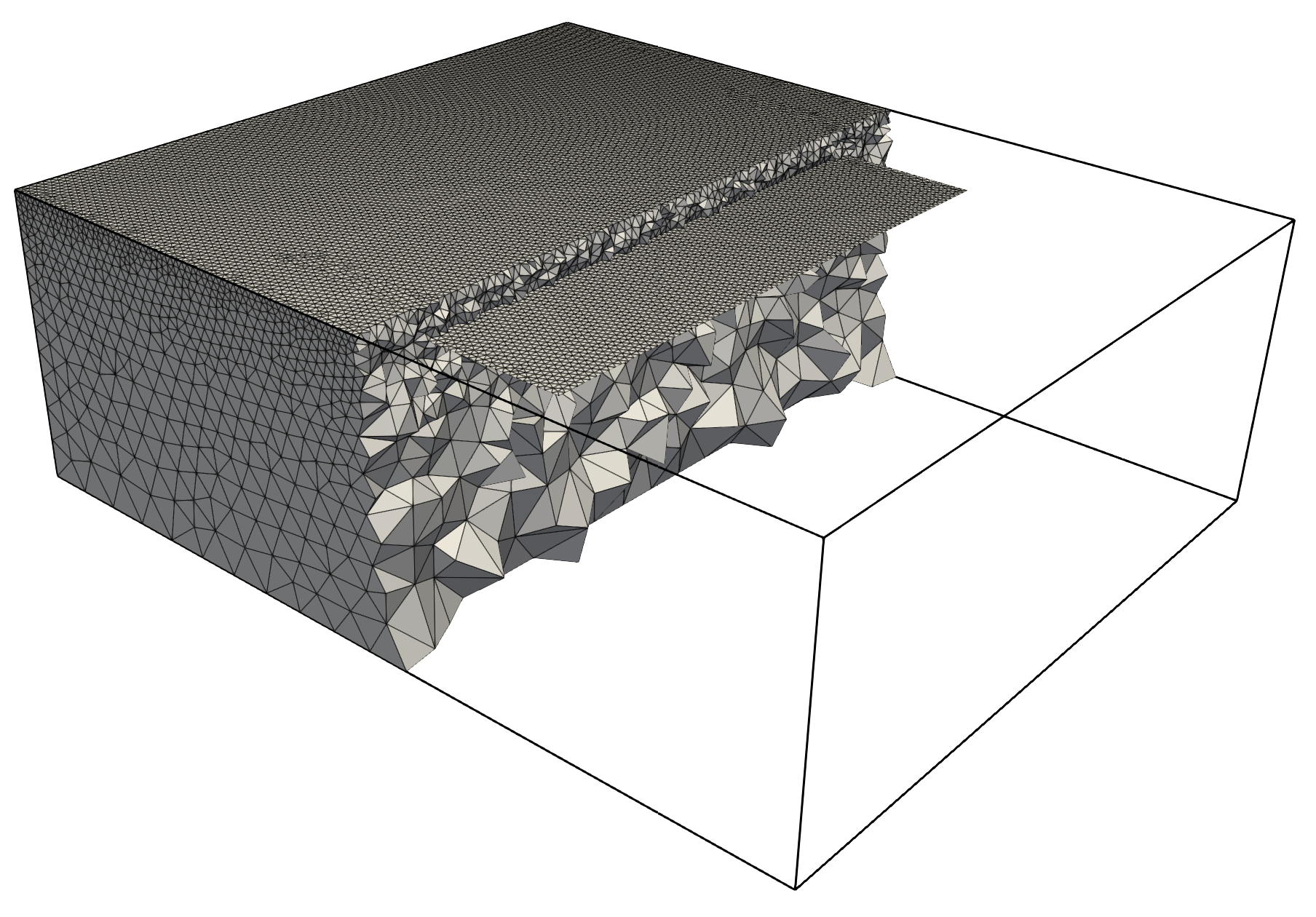}
        \caption{}
        \label{fig:3d_meshes rect}
    \end{subfigure}
    \begin{subfigure}{0.49\linewidth}
        \centering
        \includegraphics[width=1\linewidth]{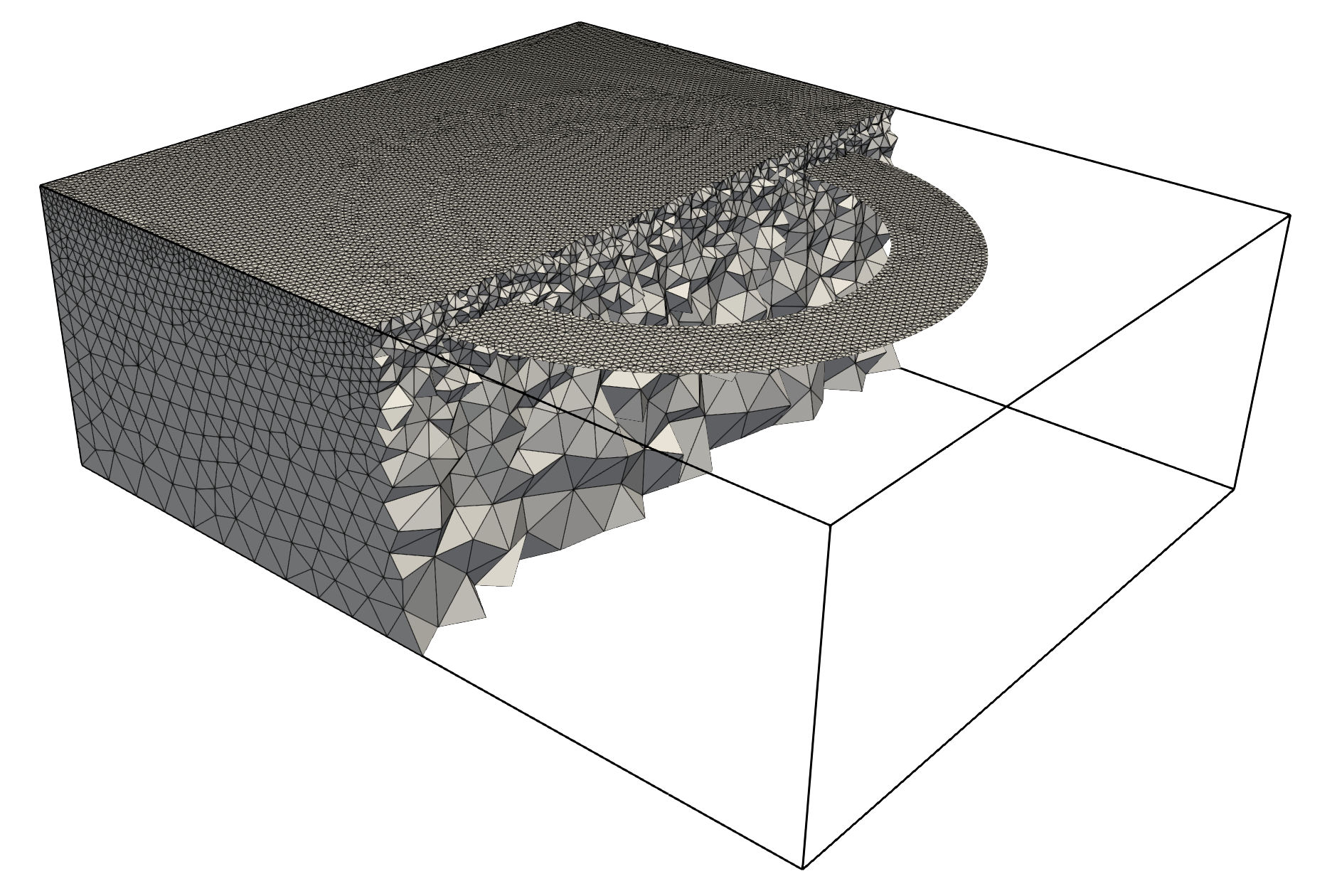}
        \caption{}
        \label{fig:3d_meshes annulus}
    \end{subfigure}
    \caption{The meshes used for the three-dimensional channel examples. The two cases shown are (a) the centred rectangular plate, and (b) the annulus-shaped plate. So that the plate is visible, we only visualise half of the full background mesh. Note that the mesh for the offset rectangular plate is similar to that shown in (a).}
    \label{fig:3d_meshes}
\end{figure}
In Figures~\ref{fig:3d_rect} and \ref{fig:3d_rect_offset} we show the real and imaginary parts of the free-surface displacement and plate displacement for the centred rectangular plate and the offset rectangular plate, respectively. Tables~\ref{tab:3d_rect} and \ref{tab:3d_rect_offset} give the values of the reflection and transmission coefficients for the propagating modes for the centred rectangular plate and offset rectangular plate, respectively. As expected, the reflection and transmission coefficients corresponding to the symmetric mode, $n=2$, become large for the offset plate due to the non-symmetry of the system.
\begin{figure}[!t]
    \centering
        \includegraphics[trim={6.5cm 0 4cm 0},clip,width=0.49\linewidth]{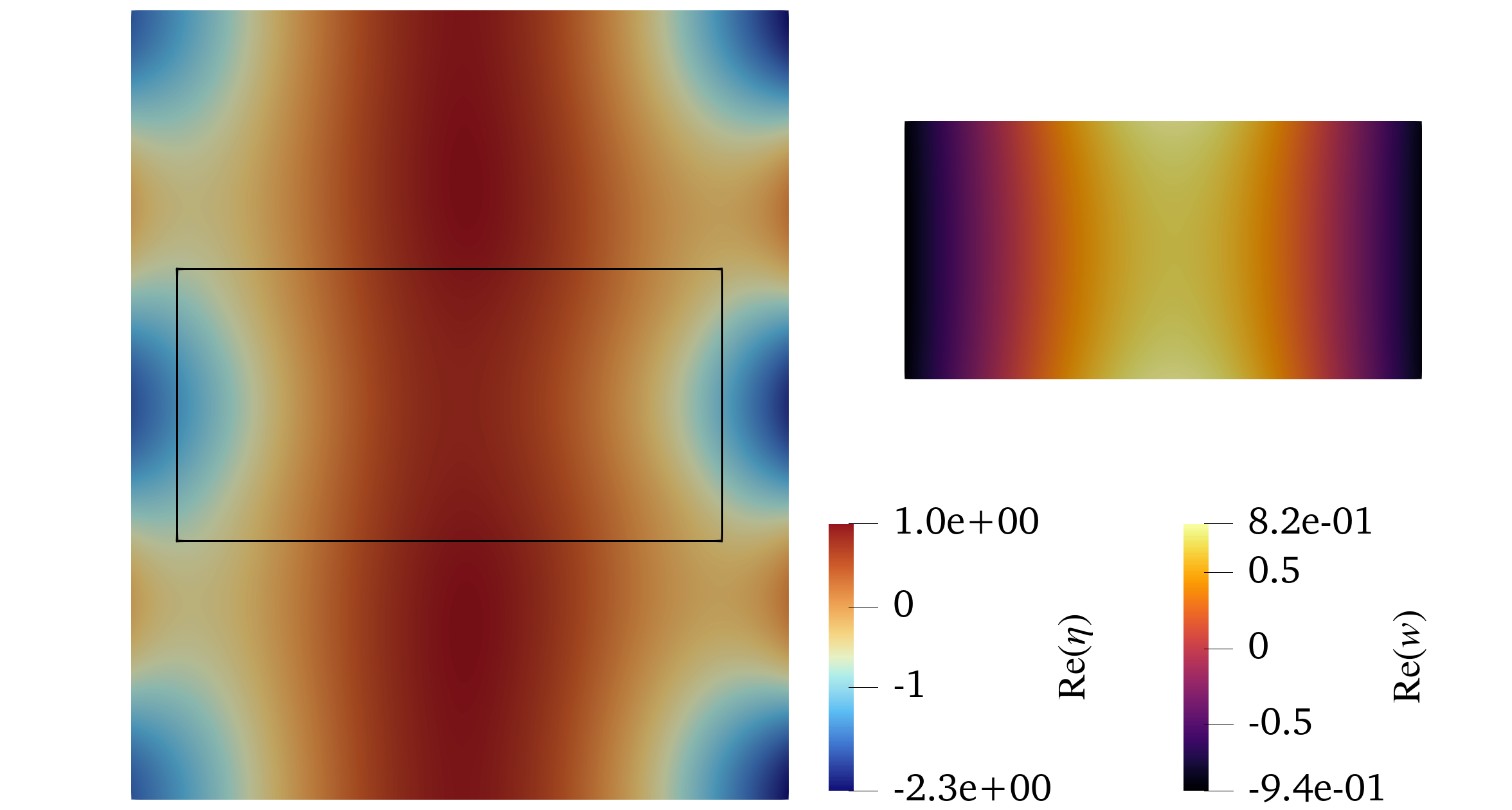}
        \includegraphics[trim={6.5cm 0 3.5cm 0},clip,width=0.49\linewidth]{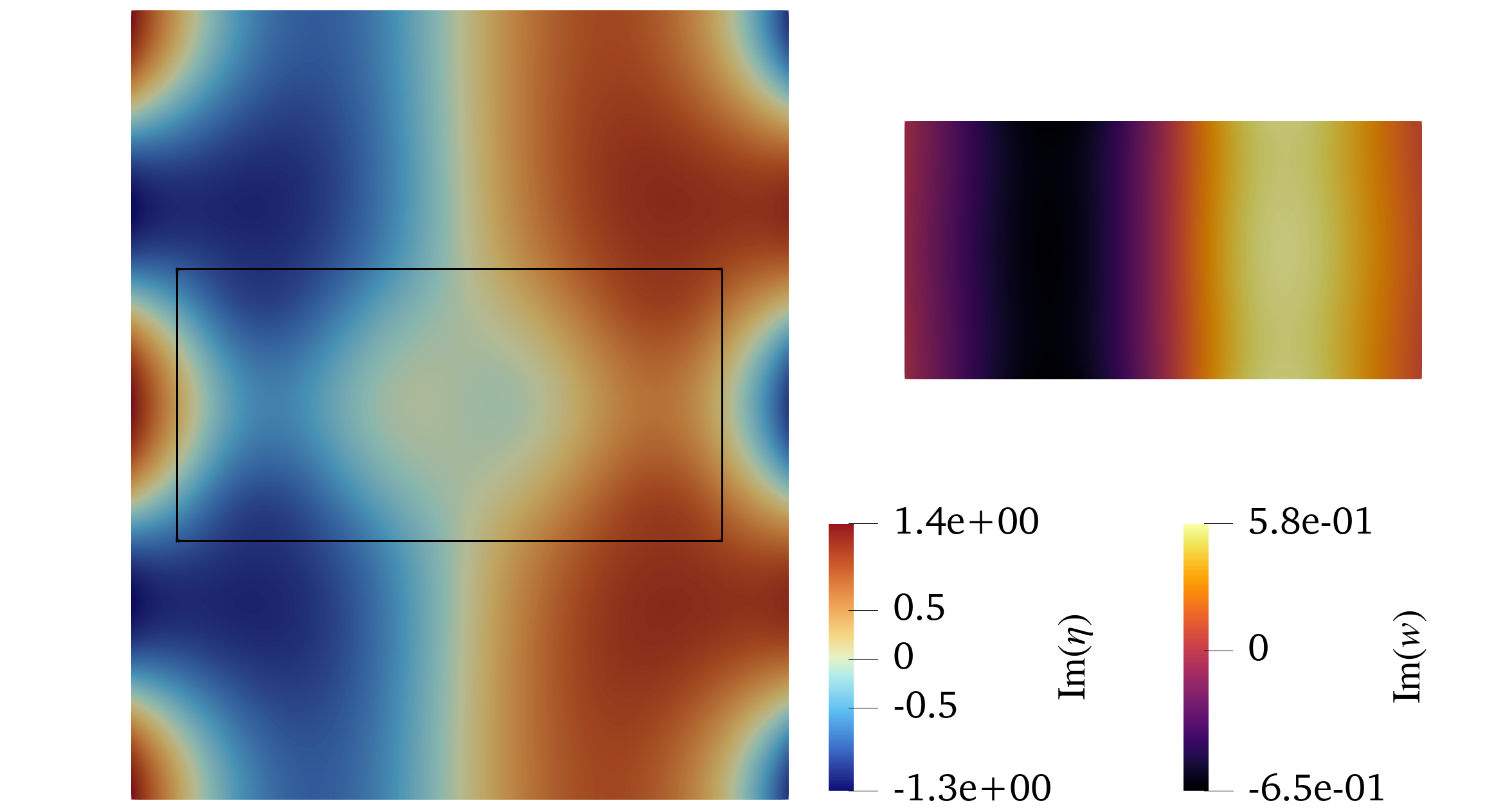}
    \caption{Visualisations of the real and imaginary parts of the free-surface displacement and plate displacement for the three-dimensional channel with a submerged rectangular plate with an incident wave with period 4\,s.}
    \label{fig:3d_rect}
    
    \captionof{table}{The reflection and transmission coefficients of the propagating modes for the case of a centred rectangular plate in Figure~\ref{fig:3d_rect}.}
    \label{tab:3d_rect}
    \begin{tabular}{c|c|c|c}
    $n$ & 1 & 2 & 3 \\\hline
    $\lvert R_{0n}\rvert$ & $2.76\times 10^{-2}$ & $1.54\times 10^{-5}$ & $2.05\times 10^{-2}$ \\
    $\lvert T_{0n}\rvert$ & $9.99\times 10^{-1}$ & $2.17\times 10^{-5}$ & $1.05\times 10^{-1}$ \\
    \end{tabular}

    \centering
    \includegraphics[trim={6.5cm 0 4cm 0},clip,width=0.49\linewidth]{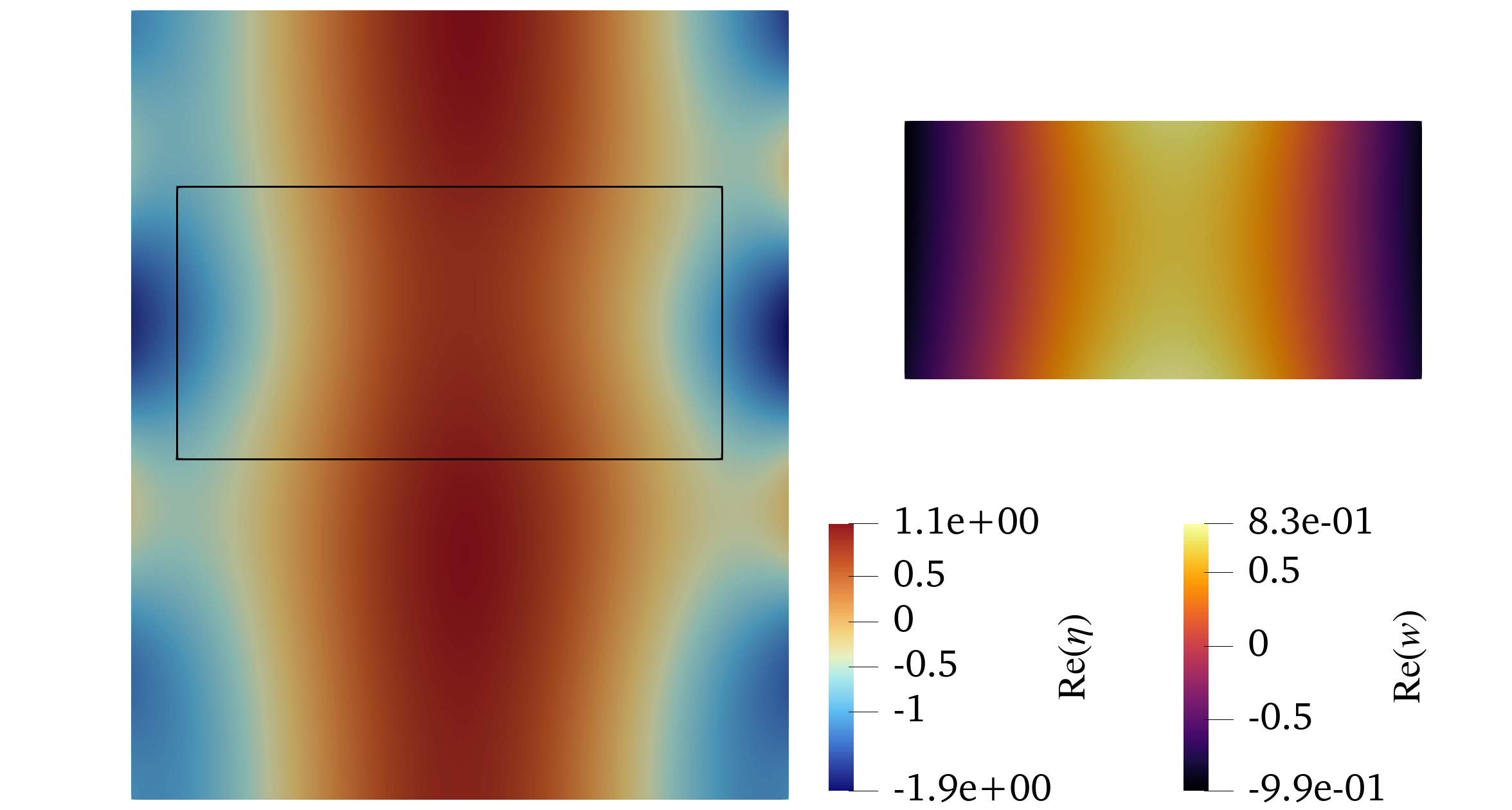}
    \includegraphics[trim={6.5cm 0 4cm 0},clip,width=0.49\linewidth]{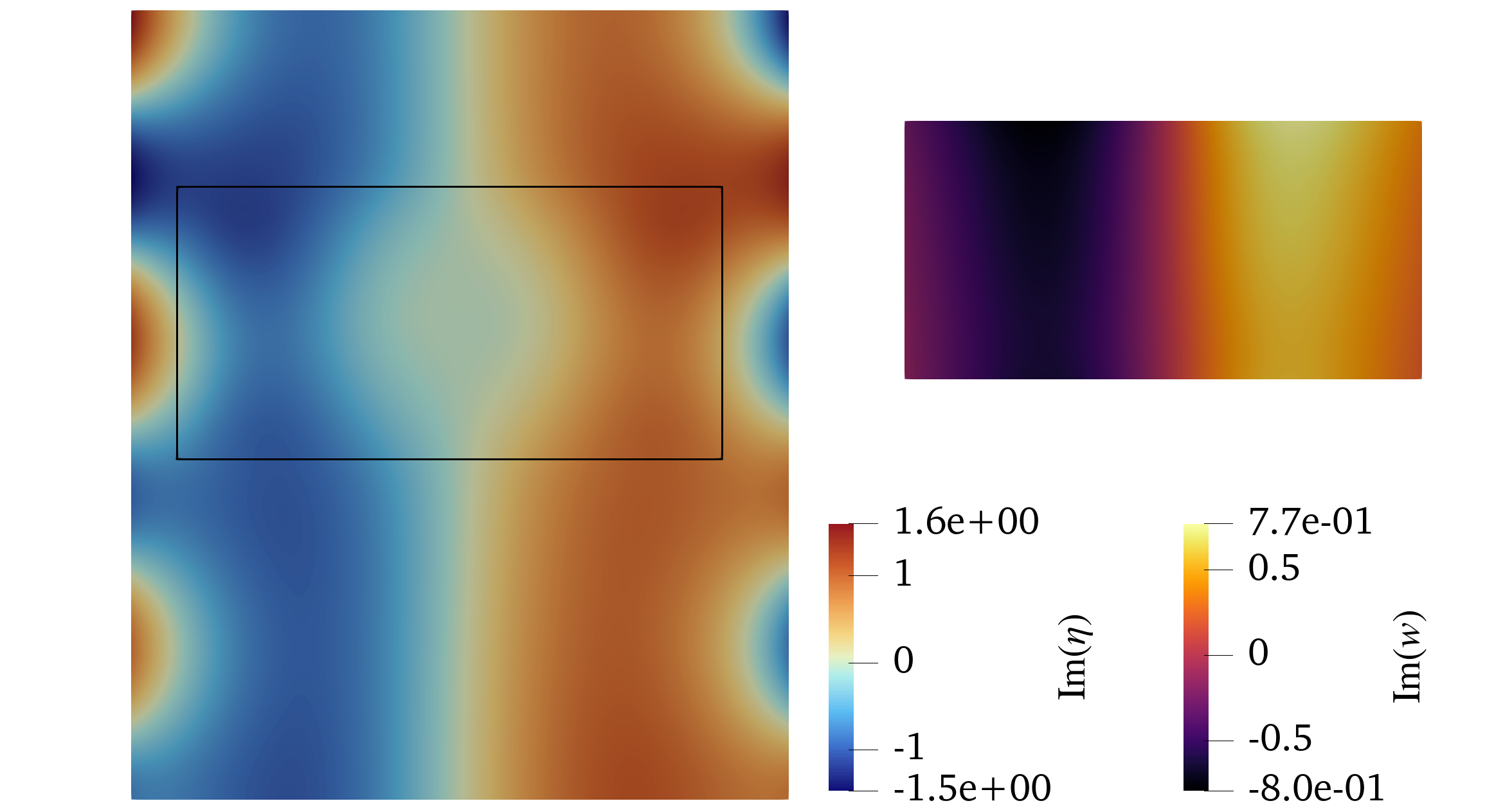}
    \captionof{figure}{Visualisations of the real and imaginary parts of the free-surface displacement and plate displacement for the three-dimensional channel with a submerged rectangular plate that is offset in the $y$ direction with an incident wave with period 4\,s.}
    \label{fig:3d_rect_offset}
    \captionof{table}{The reflection and transmission coefficients of the propagating modes for the case of an offset rectangular plate in Figure~\ref{fig:3d_rect_offset}.}
    \label{tab:3d_rect_offset}
    \begin{tabular}{c|c|c|c}
    $n$ & 1 & 2 & 3 \\\hline
    $\lvert R_{0n}\rvert$ & $2.56\times 10^{-2}$ & $1.59\times 10^{-2}$ & $5.31\times 10^{-3}$ \\
    $\lvert T_{0n}\rvert$ & $9.98\times 10^{-1}$ & $5.19\times 10^{-2}$ & $6.58\times 10^{-2}$ \\
    \end{tabular}
\end{figure}

We also check energy conservation of our solution by computing $\lvert\xi\rvert$ where, for a plane wave, $\xi$ is given by \citep{ALDOSRI2025103493}
\begin{equation}\label{eqn: energy bal 3d}
    \xi\coloneqq\left(\lvert R_{00}\rvert^2+\lvert T_{00}\rvert^2\right)+\frac{1}{2}\sum_{n=1}^q\left(\frac{k_{0n}}{k_{00}}\right)\left(\lvert R_{0n}\rvert^2+\lvert T_{0n}\rvert^2\right)-1,
\end{equation}
where $q$ is the number of propagating modes and $R_{ij}$, $T_{ij}$, and $k_{mn}$ are as in \eqref{eqn: Rij 3d}, \eqref{eqn: Tij 3d}, and \eqref{eqn: kmn 3d}, respectively. Note that $R_{ij}$ and $T_{ij}$ are scaled by $1/\beta$ so that 
$$\left(\lvert R_{00}\rvert^2+\lvert T_{00}\rvert^2\right)+\frac{1}{2}\sum_{n=1}^q\left(\frac{k_{0n}}{k_{00}}\right)\left(\lvert R_{0n}\rvert^2+\lvert T_{0n}\rvert^2\right)\approx 1.$$
Using \eqref{eqn: energy bal 3d}, we find that $\lvert\xi\rvert=3.09 \times 10^{-9}$ and $\lvert\xi\rvert=5.63 \times 10^{-10}$ for the centred and offset plates, respectively. This shows that the solution approximately satisfies energy conservation.

The final example we consider is the case of an annulus-shaped plate with an inner radius of 7\,m and an outer radius of 10\,m. Figure~\ref{fig:3d_meshes annulus} shows the mesh used for this example. Figure~\ref{fig:3d_annulus} visualises the real and imaginary parts of the free-surface displacement and plate displacement for the annulus-shaped plate and in Table~\ref{tab:3d_annulus} we give the reflection and transmission coefficients for the propagating modes. In this case, the reflection and transmission coefficients corresponding to the symmetric mode, $n=2$, are small due to the symmetry of the system. In addition, we find the solution approximately satisfies energy conservation with $\lvert\xi\rvert=1.12 \times 10^{-8}$.
As in the case of the two-dimensional examples, changing the shape of the plate requires no adjustments to the code, other than the mesh generation. This is a significant advantage of the finite element approach. We expect that this will be particularly useful for investigating the effect of the structural shape on the performance of wave energy converters.
\begin{figure}[!t]
    \centering
        \centering
        \includegraphics[trim={6.5cm 0 4cm 0},clip,width=0.49\linewidth]{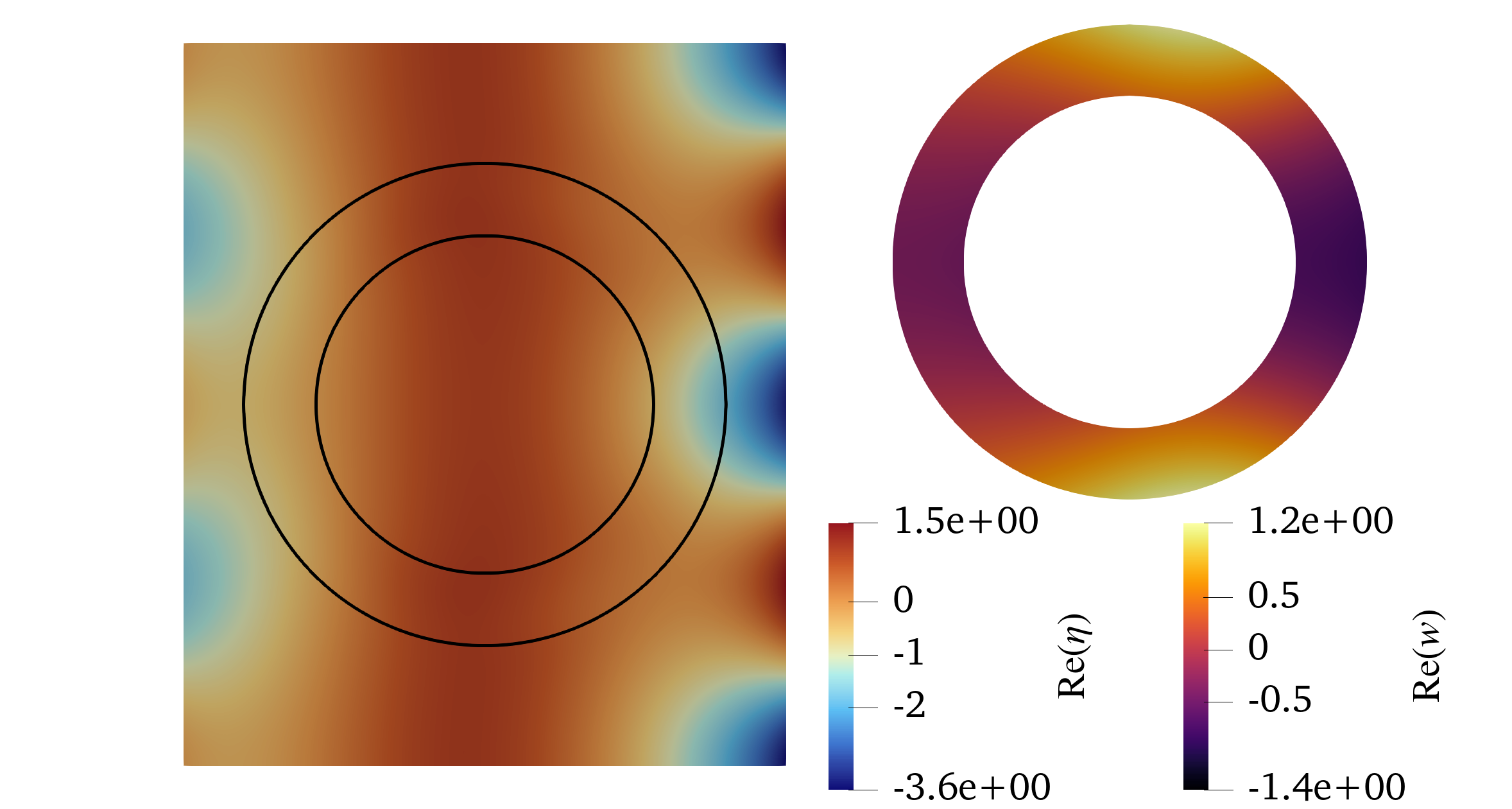}
        \includegraphics[trim={6.5cm 0 4cm 0},clip,width=0.49\linewidth]{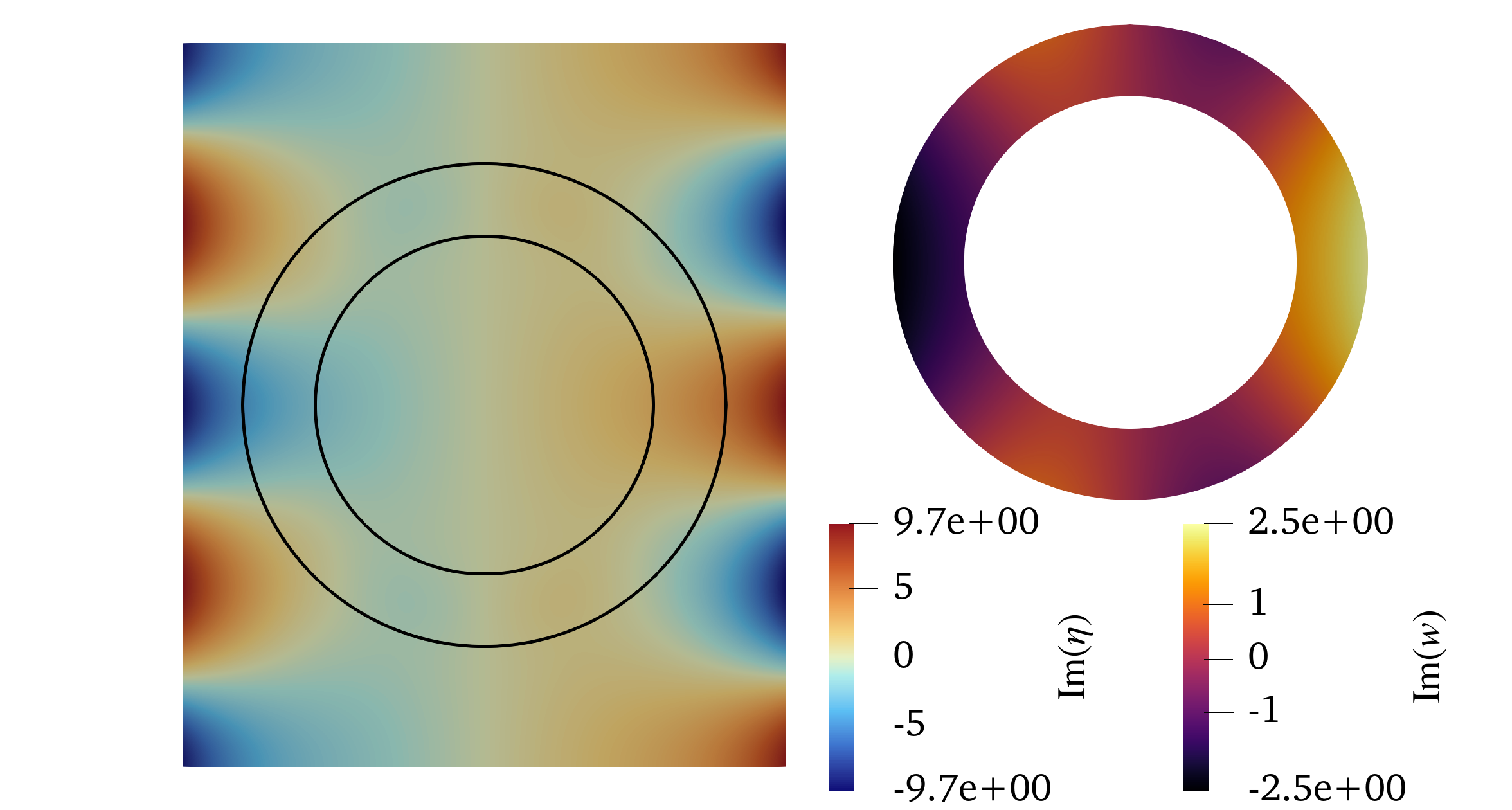}
    \caption{Visualisations of the real and imaginary parts of the free-surface displacement and plate displacement for the three-dimensional channel with a submerged annulus-shaped plate and an incident wave with period 4\,s.}
    \label{fig:3d_annulus}
    \captionof{table}{The reflection and transmission coefficients of the propagating modes for the case of an annulus-shaped plate in Figure~\ref{fig:3d_annulus}.}
	\label{tab:3d_annulus}
    \begin{tabular}{c|c|c|c}
    $n$ & 1 & 2 & 3 \\\hline
	$\lvert R_{0n}\rvert$ & $5.97\times 10^{-2}$ & $5.39\times 10^{-4}$ & $1.62\times 10^{-1}$ \\
    $\lvert T_{0n}\rvert$ & $9.93\times 10^{-1}$ & $5.07\times 10^{-4}$ & $1.02\times 10^{-1}$ \\
    \end{tabular}
\end{figure}

\section{Conclusions}\label{sec: Conclusions}\noindent
In this paper we developed a finite element method for problems involving wave-structure interactions of submerged plates. We formulated the method for the two-dimensional case as well as for the three-dimensional channel. For the fluid part of the problem, we utilised a symmetric discontinuous Galerkin scheme that allowed us to capture the discontinuity in the potential across the plate. To impose the Sommerfeld radiation conditions in a finite domain, we constructed appropriate boundary conditions from the analytic solution outside our computational domain. This yields non-local operators on the fluid boundaries of the computational domain. However, the resulting matrices are sparse compared to the overall size of the discretised system. For the plate, we used a continuous/discontinuous Galerkin (C/DG) method to resolve the 4$^{\rm th}$ order operator without requiring continuously differentiable finite elements. We have shown that the proposed formulation is well-posed by: first, proving that the sesquilinear form is bounded; second, proving the discrete G{\aa}rding inequality; and finally, showing that there exists a unique solution to the problem that is stable.

We implemented our finite element formulation using Gridap \citep{Badia2020,Verdugo2022}, a flexible finite element package written in the Julia programming language. Thanks to Gridap's programming interface, the resulting code is near one-to-one with the mathematical notation and can be used for a range of cases with only minimal changes. The code is freely available, as noted in the \textit{Data availability} statement. Releasing this open-source software allows researchers to further verify our solutions and will facilitate the numerical approximation of more complicated problems.

We validated our method against a hypersingular boundary integral method \citep{Wegert_Wilks_Thamwattana_Challis_Koley_Meylan_2026} for the case of a horizontal submerged plate and found excellent agreement between the two solution methods. As an extension to the two-dimensional case, we considered a submerged slanted plate. Extending the implementation to this case required no adjustments to the code, other than the mesh generation --- a significant benefit of the proposed methodology. Finally, we considered the case of a three-dimensional channel with three different plate configurations: a centred rectangular plate, an offset rectangular plate, and an annulus-shaped plate. We showed that the resulting reflection and transmission coefficients for the propagating waves satisfy energy balance.

We have focused here on hydroelastic problems, which are among the most complicated problems in water-wave scattering and of particular relevance to the modelling of wave energy converters. Such problems are naturally suited to FEM calculations, since this method can readily accommodate both fluid and elastic components with appropriate modifications. Furthermore, the approach is general and 
its application to fixed or rigid structures, or to variable bottom topography, would be a straightforward extension of the present work and would, in fact, represent a simplification of the problem. 
In future, we plan to extend this formulation to a radially-infinite ocean, investigate the extension to piezoelectric wave energy converters, and develop unfitted finite element methods for describing the shape of the submerged plate. The latter will enable shape and topology optimisation of the plate to investigate the effect of the plate's shape on wave energy conversion.

\section*{CRediT authorship contribution statement}\noindent
\textbf{Zachary J Wegert:} Writing -- original draft, Writing -- review and editing, Conceptualisation, Formal analysis, Investigation, Methodology, Software, Validation, Visualisation. \textbf{Oriol Colomés:} Writing -- original draft, Writing -- review and editing, Conceptualisation, Formal analysis, Methodology. \textbf{Michael H Meylan:} Writing -- original draft, Writing -- review and editing, Conceptualisation, Validation, Project administration, Funding acquisition, Resources. \textbf{Ngamta Thamwattana:} Writing -- original draft, Writing -- review and editing, Conceptualisation, Project administration, Funding acquisition, Resources. \textbf{Vivien J Challis:} Writing -- original draft, Writing -- review and editing, Conceptualisation, Project administration, Funding acquisition, Resources, Supervision.

\section*{Declaration of Competing Interest}\noindent
The authors have no competing interests to declare that are relevant to the content of this article.

\section*{Data availability}\noindent
The source code and data for this work is available at \url{https://github.com/zjwegert/Wegert_et_al_2026_SubFEM}.

\section*{Acknowledgement}\noindent
This research was supported by the Australian Government through the Australian Research Council (ARC) Discovery Grant DP240102104. The authors would also like to thank Dr Ben Wilks for many insightful discussions that have helped to improve the manuscript. Oriol Colom\'es gratefully acknowledges the support received by the project DigiOcean4Solar with file number 21225 of the research programme NWO Talent ProgrammeVidi AES 2023 which is financed by the Dutch Research Council (NWO) under the grant ID \hyperlink{https://doi.org/10.61686/OPCTU16570}{https://doi.org/10.61686/OPCTU16570}.

\appendix

\bibliographystyle{elsarticle-num-names} 
\bibliography{main}

\end{document}